\documentclass[a4paper,11pt,english]{article}

\usepackage{a4wide,bm,epsfig,booktabs}

\usepackage{amsmath}
\usepackage{tikz}
\usepackage{setspace}
\usepackage{relsize}  
\usepackage{graphicx}  
\usepackage{cancel} 
\usepackage{slashed}
\usepackage{verbatim} 
\usepackage{enumerate} 
\usepackage{stmaryrd} 
\usepackage{cite}
\usepackage[flushmargin,hang]{footmisc}
\usepackage{fancyhdr} 
\usepackage[arrow, matrix,curve]{xy}
\usepackage{hyperref}
\usepackage{amssymb}
\usepackage{amsthm}
\usepackage{eucal}
\usepackage{xcomment}
\usepackage{mathtools}

\newcommand{\myeq}[2]{\stackrel{\mathclap{\normalfont\mbox{#1}}}{#2}}

\newcommand{\setm} {\backslash}
\newcommand{\NN} {\mathbb{N}}

\newcommand{\RR} {\mathbb{R}}

\newcommand{\kk} {\mathrm{k}}

\newcommand{\bl} {\boldsymbol{[}}
\newcommand{\br} {\boldsymbol{]}}

\newcommand{\Lip} {\mathrm{Lip}}

\newcommand{\F} {\mathrm{F}}

\newcommand{\expal} {\mathfrak{i}}

\newcommand{\MD} {\mathcal{D}}

\newcommand{\Map} {\mathrm{Map}}

\newcommand{\lip}  {\mathrm{lip}}
\newcommand{\const}  {\mathrm{c}}

\newcommand{\uu} {\mathfrak{u}}
\newcommand{\mm} {\mathfrak{m}}

\newcommand{\oo} {\mathfrak{o}}
\newcommand{\pp} {\mathfrak{p}}
\newcommand{\qq} {\mathfrak{q}}

\newcommand{\vv} {\mathfrak{v}}
\newcommand{\ww} {\mathfrak{w}}

\newcommand{\hh} {\mathfrak{h}}

\newcommand{\zzs} {\mathfrak{s}}

\newcommand{\symb} {\centerdot}

\newcommand{\mmm} {{\symb}\mm}

\newcommand{\ppp} {{\symb}\pp}
\newcommand{\qqq} {{\symb}\qq}

\newcommand{\vvv} {{\symb}\vv}
\newcommand{\www} {{\symb}\ww}

\newcommand{\lleq} {\prec}

\newcommand{\cpp} {\comp{\pp}}

\newcommand{\SEMG} {\mathfrak{S}}
\newcommand{\SEMMM} {\mathfrak{H}}
\newcommand{\SEM} {\mathfrak{P}}

\newcommand{\SEMM} {\mathfrak{Q}}

\newcommand{\inv} {\mathrm{inv}}
\DeclareMathOperator*{\innt}{\ThisStyle{\vstretch{0.9}{\hstretch{1.5}{\rotatebox{10}{$\SavedStyle\hspace{-0.5pt}\!\int\!\hspace{-0.5pt}$}}}}}
\DeclareMathOperator*{\Der}{\ThisStyle{\hstretch{1.2}{\rotatebox{0}{$\SavedStyle\delta^r$}}}}
\usepackage{scalerel}

\newcommand{\DIDE} {\mathfrak{D}}
\newcommand{\DIDED} {\mathrm{D}}
\newcommand{\EV} {\mathrm{Evol}}
\newcommand{\evol} {{\mathrm{evol}}}

\newcommand{\MX} {\mathfrak{X}}

\newcommand{\wph} {\boldsymbol{\varphi}}

\newcommand{\B} {\mathrm{B}}
\newcommand{\OB} {\ovl{\B}}

\newcommand{\chart} {\Xi}
\newcommand{\chartinv} {\Xi^{-1}}
\newcommand{\RT} {\mathrm{R}}
\newcommand{\LT} {\mathrm{L}}

\newcommand{\Ad} {\mathrm{Ad}}

\newcommand{\CP} {\mathrm{CP}}

\newcommand{\DP} {\DIDE\mathrm{P}}

\newcommand{\mackarr}[1] {\rightharpoonup_{\mathfrak{m}.#1}}
\newcommand{\mackarrr} {\rightharpoonup_{\mathfrak{m}}}
\newcommand{\seqarr}[1] {\rightharpoonup_{\mathfrak{s}.#1}}
\newcommand{\seqarrr} {\rightharpoonup_{\mathfrak{s}}}
\newcommand{\netarr}[1] {\rightharpoonup_{\mathfrak{n}.#1}}

\newcommand{\rhon} {\boldsymbol{\rho}}

\newcommand{\dindp}  {\mathrm{p}}
\newcommand{\dind}  {\mathrm{s}}

\newcommand{\mackeyconst} {\mathfrak{c}}

\newcommand{\mackeyindex} {\mathfrak{l}}

\newcommand{\limin} {\lim^\infty_{n}}
\newcommand{\limih} {\lim^\infty_{h\rightarrow 0}}

\newcommand{\comp}[1] {\ovl{#1}}
\newcommand{\mgc}  {\ovl{\mg}}

\newcommand{\llleq} {\preceq}

\newcommand{\conj}  {\mathrm{Conj}}

\newcommand{\id} {\mathrm{id}}

\newcommand{\dermapdiff}  {\omega}
\newcommand{\dermapinvdiff}  {\upsilon}

\newcommand{\ovl}[1] {\overline{#1}}
\newcommand{\wt}[1] {\widetilde{#1}}
\newcommand{\wh}[1] {\widehat{#1}}

\newcommand{\dd} {\mathrm{d}}
\newcommand{\ddd} {\mathrm{d}}
\newcommand{\im} {\mathrm{im}}
\newcommand{\dom} {\mathrm{dom}}

\newcommand{\U}  {\mathcal{U}}
\newcommand{\V}  {\mathcal{V}}

\newcommand{\deff} {if and only if } 
\newcommand{\defff} {if } 

\newcommand{\equivdef} {:=} 
\newcommand{\equiveq} {=}

\newcommand{\mg} {\mathfrak{g}}

\newcommand{\cp} {\circ}
\newcommand{\COMP} {\mathfrak{K}}

\newcommand{\mult}  {\mathrm{m}}

\newcommand{\compact} {\mathrm{C}}
\newcommand{\compacto} {\mathrm{K}}

\newcommand{\he} {\hspace{1pt}}

\renewcommand{\theenumi}{\arabic{enumi})} 
\renewcommand{\labelenumi}{\theenumi}

\let\origenumerate\enumerate
\def\enumerate{\origenumerate\itemsep0pt}
\let\origitemize\itemize
\def\itemize{\origitemize\itemsep0pt}

\newtheorem{innercustomthm}{Theorem}
\newenvironment{customthm}[1]
  {\renewcommand\theinnercustomthm{#1}\innercustomthm}
  {\endinnercustomthm}
  
 \newtheorem{innercustomtpr}{Proposition}
\newenvironment{customtpr}[1]
  {\renewcommand\theinnercustomtpr{#1}\innercustomtpr}
  {\endinnercustomtpr}
  
   \newtheorem{innercustomlem}{Lemma}
\newenvironment{customlem}[1]
  {\renewcommand\theinnercustomlem{#1}\innercustomlem}
  {\endinnercustomlem}
  
     \newtheorem{innercustomco}{Corollary}
\newenvironment{customco}[1]
  {\renewcommand\theinnercustomco{#1}\innercustomco}
  {\endinnercustomco}

\newtheorem{theorem}{Theorem}
\newtheorem{proposition}{Proposition}
\newtheorem{lemma}{Lemma}
\newtheorem{corollary}{Corollary}
\newtheorem{remark}{Remark}

\newtheorem{example}{Example}

\makeatletter
\def\blfootnote{\gdef\@thefnmark{}\@footnotetext}
\makeatother

\begin{document}
\title{Differentiability of the Evolution Map\\and Mackey Continuity}
\author{
  \textbf{Maximilian Hanusch}\thanks{\texttt{mhanusch@math.upb.de}}
  \\[1cm]
  Institut f\"ur Mathematik \\
  Universit\"at Paderborn\\
  Warburger Stra\ss{e} 100 \\
  33098 Paderborn \\
  Germany
}
\date{September 6, 2019}
\maketitle

\begin{abstract}
We solve the differentiability problem for the evolution map in Milnor's infinite dimensional setting. We first show that the evolution map of each $C^k$-semiregular Lie group $G$ (for $k\in \NN\sqcup\{\lip,\infty\}$) admits a particular kind of sequentially continuity -- called Mackey k-continuity. We then prove that this continuity property is strong enough to ensure differentiability of the evolution map. In particular, this drops any continuity presumptions made in this context so far. Remarkably, Mackey k-continuity 
arises directly from the regularity problem itself, which makes it particular among the continuity conditions traditionally considered. As an application of the introduced notions, we discuss the strong Trotter property in the sequentially-, and the Mackey continuous context. We furthermore conclude that if the Lie algebra of $G$ is a Fr\'{e}chet  space, then $G$ is $C^k$-semiregular (for $k\in \NN\sqcup\{\infty\}$) \deff $G$ is $C^k$-regular. 
\end{abstract}

\tableofcontents 

\section{Introduction}
In 1983 Milnor introduced his regularity concept \cite{MIL} as a tool to extend proofs of fundamental Lie theoretical facts to infinite dimensions. Specifically, he adapted (and weakened) the regularity concept introduced in 1982 by Omori, Maeda, Yoshioka and Kobayashi for Fr\'{e}chet Lie groups \cite{OMORI} to such Lie groups that are modeled over complete Hausdorff locally convex vector spaces. Then, he used this notion to prove the integrability of Lie algebra homomorphisms to Lie group homomorphisms under certain regularity and connectedness presumptions. 
 In this paper, we work in the slightly more general setting introduced by Gl\"ockner in \cite{HG} -- specifically meaning that any completeness presumption on the modeling space is dropped.\footnote{Confer also \cite{KHN2,KHNM} for an introduction to this   area. To prevent confusion, we additionally remark that Milnor's definition of an infinite dimensional manifold $M$ involves the requirement that $M$ is a regular topological space, i.e., fulfills the separation axioms $T_2, T_3$. Deviating from that, in \cite{HG}, only the $T_2$ property of $M$ is explicitly presumed -- This, however, makes no difference in the Lie group case, because topological groups are automatically $T_3$.} 

Roughly speaking, regularity is concerned with definedness, continuity, and smoothness of the evolution map (product integral) -- a notion that naturally generalizes the concept of the Riemann integral for curves in locally convex vector spaces, to infinite dimensional Lie groups (Lie algebra valued curves are thus integrated to Lie group elements). For instance, the exponential map of a Lie group is the restriction of the evolution map to constant curves; and, given a principal fibre bundle, then holonomies are evolutions of such Lie algebra valued curves that are pairings of smooth connections with derivatives of curves in the base manifold of the bundle. Although individual arguments show that
the generic infinite dimensional Lie group is $C^\infty$-regular or stronger, only recently general regularity criteria had been found \cite{HGGG, KHN2, RGM}. Differentiability of the evolution map (hence, of the exponential map) is one of the key components of the regularity problem.  
In \cite{RGM, HGGG}, this issue had been discussed in the standard topological context -- implicitly meaning that continuity of the evolution maps w.r.t.\ to the $C^k$-topology was presumed.\footnote{The $C^k$-topology is recalled in Sect.\ \ref{jkdsodsoioidsoids}; and, the evolution maps are defined below.} 
In this paper, we solve the differentiability problem in full generality, as we drop 
any continuity presumption made in this context so far. The results obtained in particular imply that if the Lie group is modeled over  a Fr\'{e}chet space, with evolution map defined on all $C^k$-curves (the Lie group is $C^k$-semiregular),  then the evolution map is automatically smooth w.r.t.\ to the $C^k$-topology (the Lie group is $C^k$-regular). 
We furthermore generalize the results obtained in \cite{TRM, HGM} concerning the strong Trotter property by weakening the continuity presumptions made there.

More specifically, let $G$ denote an infinite dimensional Lie group as defined in \cite{HG} that is modeled over the Hausdorff locally convex vector space $E$.  
We let $\mg$ denote the Lie algebra of $G$; as well as $\dd_q\RT_g$ the differential of the right translation  $\RT_g\colon G\ni h\mapsto h\cdot g$  by $g\in G$, at the point $q\in G$. We furthermore define (right logarithmic derivative) 
\begin{align*}
	 C^0([0,1],\mg)\ni \Der(\mu):= \dd_{\mu}\RT_{\mu^{-1}}(\dot\mu)\qquad\quad\forall\: \mu\in C^1([0,1],G)
\end{align*} 
as well as $\DIDED := \{\Der(\mu) \:|\: \mu\in C^1([0,1],G)\}$ and $C_*^1([0,1],G):=\{\mu\in C^1([0,1],G)\:|\: \mu(0)=e\}$. 
The evolution maps are given by
\begin{align*}
	\hspace{40pt}\EV\colon \DIDED\ni \Der(\mu) &\mapsto   \mu\cdot \mu^{-1}(0)\in C_*^1([0,1],G)
	 \\[2pt] 
	\evol\colon \DIDED\ni \phi&\mapsto \EV(\phi)(1)\in G\\[3pt]
	\text{as } &\text{well } \text{as}\\[2pt]
	\EV_\kk:= \EV|_{\DIDED_\kk}
	\qquad&\he\text{and}\qquad 
	\evol_\kk:=\evol|_{\DIDED_\kk},
\end{align*}
with $\DIDED_\kk\equivdef \DIDED\cap C^k([0,1],\mg)$  
for each $k\in \NN\sqcup\{\lip,\infty,\const\}$.\footnote{Here, $C^\lip([0,1],\mg)$ denotes the set of Lipschitz curves, and $C^\const([0,1],\mg)$ denotes the set of constant curves.}  
We say that $G$ is $C^k$-semiregular \defff $C^k([0,1],\mg)\subseteq \DIDED$ holds; hence, \defff each $\phi \in C^k([0,1],\mg)$ admits a (necessarily unique) solution $\mu\in C^1_*([0,1],G)$ to the differential equation $\Der(\mu)=\phi$.  
It was shown in \cite{HGGG} (cf.\ Theorem E in \cite{HGGG}) that if $G$ is $C^k$-semiregular for $k\in \NN\sqcup\{\infty\}$, then $\EV_\kk$ (thus, $\evol_\kk$) is smooth 
\deff $\evol_\kk$ is of class $C^1$. 
Then, it was proven in \cite{RGM} (cf.\ Theorem 4 in \cite{RGM}) that $\evol_\kk$ is of class $C^1$ \deff it is continuous, with $\mg$ Mackey complete for $k\in \NN_{\geq 1}\sqcup \{\lip,\infty\}$ (as well as integral complete for $k\equiveq 0$). 
All these statement have been established in the standard topological context -- specifically meaning that $\evol_\kk$ (and $\EV_\kk$) was presumed to be continuous w.r.t.\ the $C^k$-topology. 
In this paper, we more generally show that, cf.\ (the more comprehensive) Theorem \ref{sasasasaasdassdsasdadds} in Sect.\ \ref{lvclkvclkvc}
\begin{customthm}{A}\label{weakdiff}
Suppose that $G$ is $C^k$-semiregular for $k\in \NN\sqcup\{\lip, \infty\}$. Then, $\evol_\kk$ is differentiable 
\begingroup
\setlength{\leftmargini}{12pt}
\begin{itemize}
\item
	for $k \equiveq 0$\hspace{46.5pt} \deff $\mg$ is integral complete.
\item
	for $k\in \NN_{\geq 1}\sqcup \{\infty\}$ \deff $\mg$ is Mackey complete.
\end{itemize}
\endgroup
\noindent
In this case, $\evol_\kk$ is differentiable, with\footnote{Notably, this formula is well known from the finite dimensional context  (cf., e.g., the proof of (1.13.4) Proposition in \cite{DUIS}), and also for regular Lie groups in the convenient setting \cite{MK}.} 
\vspace{-2pt}
\begin{align*}
	\textstyle\dd_\phi\he \evol_\kk(\psi)=\dd_e\LT_{\innt \phi}\big(\int \Ad_{[\innt_r^s \phi]^{-1}}(\psi(s))\:\dd s \big)\qquad\quad \forall\: \phi,\psi\in C^k([r,r'],\mg).
\end{align*}  
\end{customthm}
\noindent
This theorem will be derived from significantly more fundamental results established in this paper: 
Let $\chart\colon \U\rightarrow \V\subseteq E$ be a fixed chart around $e$, and $\SEM$ the system of continuous seminorms on $E$. 
A pair $(\phi,\psi)\in C^0([0,1],\mg)\times C^0([0,1],\mg)$ is said to be
\begingroup
\setlength{\leftmargini}{12pt}
\begin{itemize}
\item
	\emph{admissible} \defff\he $\phi + (-\delta,\delta)\cdot \psi\subseteq \DIDED$ holds for some $\delta>0$,
\item
	\emph{regular}\hspace{15.9pt} \defff\he it is admissible, with 
	\begin{align*}
		\textstyle\limih \chart\big(\EV(\phi)^{-1}\cdot \EV(\phi+h\cdot \psi)\big)=0.
	\end{align*}
	\vspace{-22pt}
\end{itemize}
\endgroup
\noindent
Here, the limit is understood to be uniform -- In general, we write $\limih \alpha =\beta$ for $\alpha\colon (-\delta,0)\sqcup(0,\delta)\times [0,1]\rightarrow E$ with $\delta>0$ and $\beta\colon [0,1]\rightarrow \ovl{E}$ \defff 
\begin{align*}
	\textstyle\lim_{h\rightarrow 0}\: \sup\{\ovl{\pp}(\alpha(h,t)-\beta(t)) \:|\: t\in [0,1]\}=0\qquad\quad\forall\: \pp\in \SEM
\end{align*}
holds, where $\ovl{\pp}\colon \ovl{E}\rightarrow \RR_{\geq 0}$ denotes the continuous extension of the seminorm $\pp\in \SEM$ to the completion $\ovl{E}$ of $E$.   
Then, the first result we want to mention is, cf.\ Proposition \ref{saasassansamnmsa} in Sect.\ \ref{cxpopocxpocxcx}
\begin{customtpr}{B}\label{lkdlkdlkds}
Suppose that $(\phi,\psi)$ is admissible. 
\begingroup
\setlength{\leftmargini}{16pt}{
\renewcommand{\theenumi}{{\arabic{enumi}})} 
\renewcommand{\labelenumi}{\theenumi}
\begin{enumerate}
\item
\label{saasassansamnmsa1}
The pair $(\phi,\psi)$ is regular \deff we have
\begin{align*}
	\textstyle\limih\:1/h\cdot\chart\big(\EV(\phi)^{-1}\cdot \EV(\phi+h\cdot \psi)\big)= \int_r^\bullet (\dd_e\chart\cp\Ad_{\EV(\phi)(s)^{-1}}) (\psi(s))\:\dd s\in \ovl{E}.
\end{align*} 
\item
If $(\phi,\psi)$ is regular, then $(-\delta,\delta)\ni h\mapsto \evol(\phi+h\cdot \psi)\in G$ is differentiable at $h=0$ (for $\delta>0$ suitably small) \deff 
$\int \Ad_{\EV(\phi)(s)^{-1}}(\psi(s))\:\dd s\in \mg$ holds. In this case, we have
\begin{align*}
	\textstyle\frac{\dd}{\dd h}\big|_{h=0} \he\evol(\phi+h\cdot \psi)= \dd_e\LT_{\evol(\phi)}\big(\int \Ad_{\EV(\phi)(s)^{-1}}(\psi(s))\:\dd s\he\big).
\end{align*}
\end{enumerate}}
\endgroup
\end{customtpr}
\noindent
Evidently, each $(\phi,\psi)\in C^k([0,1],\mg)\times C^k([0,1],\mg)$ is admissible \deff $G$ is $C^k$-semiregular. In Sect.\ \ref{lkfdlkfdlkfdljgljgd}, we furthermore prove that, cf.\ Theorem \ref{fdfffdfd} in Sect.\ \ref{lkfdlkfdlkfdljgljgd}
\begin{customthm}{C}\label{weakdiffdf}
If $G$ is $C^k$-semiregular for $k\in \NN\sqcup\{\lip,\infty\}$, then $G$ is Mackey {\rm k}-continuous. 
\end{customthm}
\noindent
Here, Mackey {\rm k}-continuity is a specific kind of sequentially continuity (cf.\ Sect.\ \ref{lkkldsklsdkl}) that, in particular, implies that each admissible $(\phi,\psi)\in C^k([0,1],\mg)\times C^k([0,1],\mg)$ is regular (cf.\ Lemma \ref{dsopdsopsdopsdop} in Sect.\ \ref{lkkldsklsdkl}) -- Theorem \ref{weakdiff} thus follows immediately from Proposition \ref{lkdlkdlkds} and Theorem \ref{weakdiffdf}. 
We will conclude from Theorem \ref{weakdiffdf} and Theorem 4 in \cite{RGM} that, cf.\ Corollary \ref{ofdpoopfdopfd} in Sect.\  \ref{sklkdslkdskldslkds}
\begin{customco}{D}\label{cocuddjd}
Suppose that $\mg$ is a Fr\'{e}chet space; and let $k\in \NN\sqcup\{\infty\}$ be fixed. Then, $G$ is $C^k$-regular \deff $G$ is $C^k$-semiregular. 
\end{customco}
\vspace{6pt}

\noindent
Now, Proposition \ref{lkdlkdlkds} is actually a consequence of a more fundamental differentiability result (Proposition \ref{rererererr} in Sect.\ \ref{opdfodf}) that we will also use to generalize Theorem 5 in \cite{RGM}. Specifically, we will prove that, cf.\ Theorem \ref{ofdpofdpofdpofdpofd} in Sect.\ \ref{dffdfd}
\begin{customthm}{E}\label{assasacxcxweakdiffdf}
Suppose that $G$ is Mackey {\rm k}-continuous for $k\in \NN\sqcup\{\lip,\infty,\const\}$ -- additionally abelian if $k\equiveq \const$ holds. Let $\Phi\colon I\times [0,1]\rightarrow \mg$ ($I\subseteq \RR$ open) be given with $\Phi(z,\cdot)\in \DIDED_\kk$ for each $z\in I$.  
Then, 
\begin{align*}
	\textstyle\limih 1/h\cdot\chart\big(\EV(\Phi(x,\cdot))^{-1}\cdot \EV(\Phi(x+h,\cdot))\big)= \textstyle\int_r^\bullet (\dd_e\chart\cp \Ad_{\EV(\Phi(x,\cdot))(s)})(\partial_z\Phi(x,s))\:\dd s\hspace{1pt}\in \comp{E}
\end{align*}
holds for $x\in I$, provided that
\begingroup
\setlength{\leftmargini}{15pt}{
\renewcommand{\theenumi}{{\alph{enumi}})} 
\renewcommand{\labelenumi}{\theenumi}
\begin{enumerate}
\item
\label{saasaassasasa2}
We have $(\partial_z \Phi)(x,\cdot)\in C^k([0,1],\mg)$.
\item
\label{saasaassasasa1}
For each $\pp\in \SEM$ and $\dind\llleq k$,\footnote{This means $\dind\equiveq \lip$ for $k\equiveq \lip$, $\dind\equiveq 0$ for $k\equiveq 0$, $0\leq \dind\leq k$ for $k\in \NN$, and $\dind\in \NN$ for $k\equiveq \infty$. The corresponding seminorms $\ppp_\infty^\dind$ are defined in Sect.\ \ref{dsdsdssdfdffddfdfdd}.} there exists $L_{\pp,\dind}\geq 0$, and $I_{\pp,\dind}\subseteq I$ open with $x\in I_{\pp,\dind}$, such that
\begin{align*}
	1/|h|\cdot\ppp^\dind_\infty(\Phi(x+h,\cdot)-\Phi(x,\cdot))\leq L_{\pp,\dind}\qquad\quad \forall\: h\in \RR_{\neq 0}\:\text{ with }\: x+h\in I_{\pp,\dind}.
\end{align*}
\vspace{-22pt}
\end{enumerate}}
\endgroup
\noindent
In particular, we have
\begin{align*}
	\textstyle\frac{\dd}{\dd h}\big|_{h=0} \he \evol(\Phi(x+h,\cdot))=\textstyle \dd_e\LT_{\evol(\Phi(x,\cdot))}\big(\int \Ad_{\EV(\Phi(x,\cdot))(s)}(\partial_z\Phi(x,s))\:\dd s\he\big)
\end{align*}
\deff the Riemann integral on the right side exists in $\mg$. 
\end{customthm}
\noindent
We explicitly recall at this point that, by Theorem \ref{weakdiffdf}, for $k\in \NN\sqcup\{\lip,\infty\}$, Mackey {\rm k}-continuity is automatically given if $G$ is $C^k$-semiregular. Finally, let $\exp\colon \mg\supseteq \dom[\exp]\rightarrow G$ denote the exponential map of $G$; and recall that a Lie group $G$ is said to have the \emph{strong Trotter property} \cite{HGM, TRM, MANE, N1} 
\defff for each $\mu\in C_*^1([0,1],G)$ with $\dot\mu(0)\in \dom[\exp]$, we have  
\begin{align}
\label{sddssddszz}
	\textstyle\lim_n \mu(\tau/n)^n=\exp(\tau\cdot \dot\mu(0))\qquad\quad\forall\: \tau\in [0,\ell]
\end{align} 
uniformly\footnote{Thus, for each neighbourhood $U\subseteq G$ of $e$, there exists some $n_U\in \NN$ with $\exp(-\tau\cdot \dot\mu(0))\cdot\mu(\tau/n)^n\in U$ for each $n\geq n_U$ and $\tau\in [0,\ell]$.} for each $\ell>0$. As already figured out in \cite{HGM}, the strong Trotter property implies the strong commutator property; and, also the  Trotter  and the commutator property that are relevant, e.g., in representation theory of infinite dimensional Lie groups \cite{N1}. Now, Theorem I in \cite{HGM} states that $G$ has the strong Trotter property if $G$ is $R$-regular. This was generalized in \cite{TRM} to the locally $\mu$-convex case (hence, the case where $\evol_0$ is $C^0$-continuous on its domain, cf.\ Theorem 1 in \cite{RGM}). In this paper, we go a step further, as we show (cf.\ Proposition \ref{opsdopsdopdsopsd} in Sect.\ \ref{uoudsuoiiudsds}) that $G$ has the strong Trotter property if $\EV_0$ is sequentially continuous (the precise definitions can be found in Sect.\ \ref{ljdjkjdsds}); which is much weaker than locally $\mu$-convexity  (provided that $\mg$ is not metrizable, of course). We furthermore show in Proposition \ref{opsdopsdopdsopsd} that \eqref{sddssddszz} holds for each $\mu\in C_*^{1}([0,1],G)$ with $\dot\mu(0)\in \dom[\exp]$ and $\Der(\mu)\in C^\lip([0,1],\mg)$ if $\EV_0$ is Mackey continuous, the latter condition being even less restrictive than sequentially continuity.
\vspace{6pt}

\noindent
This paper is organized as follows:
\vspace{-4pt}
\begingroup
\setlength{\leftmargini}{12pt}
\begin{itemize}
\item
In Sect.\ \ref{dsdssd}, we provide the basic definitions; and discuss the most elementary properties of the core mathematical objects of this paper.
\item
In Sect.\ \ref{mnnxcncyxciuoduioasd}, we discuss the continuity notions considered in this paper.
\item
In Sect.\ \ref{lkfdlkfdlkfdljgljgd}, we prove Theorem \ref{fdfffdfd} (i.e., Theorem \ref{weakdiffdf}). 
\item
In Sect.\ \ref{uoudsuoiiudsds}, we discuss the strong Trotter property in the sequentially\slash Mackey continuous context.
\item
In Sect.\ \ref{opdfodf}, we establish the differentiability results for the evolution map.
\item
In Sect.\ \ref{sklkdslkdskldslkds}, we prove Corollary \ref{ofdpoopfdopfd} (i.e., Corollary \ref{cocuddjd}).
\end{itemize}
\endgroup
\noindent

\section{Preliminaries}
\label{dsdssd}
In this section, we fix the notations, and discuss the properties of the product integral (evolution map) that we will need in the main text. The proofs of the facts mentioned but not verified in this section, can be found, e.g., in Sect.\ 3 and Sect.\ 4 in \cite{RGM}.

\subsection{Conventions}
In this paper, Manifolds and Lie groups are always understood to be in the sense of \cite{HG}; in particular, smooth, Hausdorff, and modeled over a Hausdorff locally convex vector space.\footnote{We explicitly refer to Definition 3.1 and Definition 3.3 in \cite{HG}. A review of the corresponding differential calculus -- including the standard differentiation rules used in this paper -- can be found, e.g., in Appendix \ref{Diffcalc} that essentially equals Sect.\ 3.3.1 in \cite{RGM}.} 
If $f\colon M\rightarrow N$ is a $C^1$-map between the manifolds $M$ and $N$, then $\dd f\colon TM \rightarrow TN$ denotes the corresponding tangent map between their tangent manifolds -- we write $\dd_xf\equiv\dd f(x,\cdot)\colon T_xM\rightarrow T_{f(x)}N$ for each $x\in M$. 
 By an interval, we understand a non-empty, non-singleton connected subset $D\subseteq \RR$. The set of all compact intervals is denoted by $\COMP=\{[r,r']\subseteq \RR\: |\: r<r'\}$. We furthermore let $\MD_\delta:=(-\delta,0)\he\sqcup\he (0,\delta)$ for each $\delta>0$. 
 A curve is a continuous map $\gamma\colon D\rightarrow M$ for a manifold $M$ and an interval $D\subseteq \RR$.  
If $D\equiv I$ is open, then $\gamma$ is said to be of class $C^k$ for $k\in \NN\sqcup \{\infty\}$ \defff it is of class $C^k$ when considered as a map between the manifolds $I$ and $M$.
If $D$ is an arbitrary interval, then $\gamma$ is said to be of class $C^k$ for $k\in \NN\sqcup \{\infty\}$ \defff $\gamma=\gamma'|_D$ holds for a $C^k$-curve $\gamma'\colon I\rightarrow M$ that is defined on an open interval $I$ containing $D$ -- we  write $\gamma\in C^k(D,M)$ in this case.  
If $\gamma\colon D\rightarrow M$ is of class $C^1$, then we denote the corresponding tangent vector at $\gamma(t)\in M$ by $\dot\gamma(t)\in T_{\gamma(t)}M$. 
The above conventions also hold if $M\equiv F$ is a Hausdorff locally convex vector space with system of continuous seminorms $\SEMM$. 
In this case, we let $\ovl{F}$ denote the completion of $F$; as well as  
$\ovl{\qq}\colon \ovl{F}\rightarrow \RR_{\geq 0}$ the continuous extension of $\qq$ to $\ovl{F}$, for each $\qq\in \SEMM$. We furthermore define
\begin{align*}
	\B_{\qq,\varepsilon}:=\{X\in F\:|\: \qq(X)<\varepsilon\}\quad\qquad\text{as well as}\qquad\quad\OB_{\qq,\varepsilon}:=\{X\in F\:|\: \qq(X)\leq\varepsilon\}
\end{align*} 
for all $\qq\in \SEMM$ and $\varepsilon>0$. 
If $X,Y$ are sets, then $\Map(X,Y)\equiv Y^X$ denotes the set of all mappings $X\rightarrow Y$.

\subsubsection{Sets of Curves}
\label{dsdsdssdfdffddfdfdd}
Let $F$ be a Hausdorff locally convex vector space with system of continuous seminorms $\SEMM$.
\begingroup
\setlength{\leftmargini}{12pt}
\begin{itemize}
\item
By $C^\lip([r,r'],F)$ we denote  
the set of all Lipschitz curves on $[r,r']\in \COMP$; i.e., all curves $\gamma\colon [r,r']\rightarrow F$, such that
\begin{align*}
	\textstyle \Lip(\qq,\gamma):=\sup\Big\{\frac{\qq(\gamma(t)-\gamma(t'))}{|t-t'|}\:\Big|\: t,t'\in [r,r'],\: t\neq t'\Big\}\in \RR_{\geq 0}
\end{align*}
exists for each $\qq\in \SEMM$ -- i.e., we have
\begin{align*} 
	\qq(\gamma(t)-\gamma(t'))\leq \Lip(\qq,\gamma)\cdot |t-t'|\qquad\quad\forall\: t,t'\in [r,r'],\:\: \qq\in \SEMM.
\end{align*}
\item 
By $C^\const([r,r'],F)$ we denote  
the set of all constant curves on $[r,r']\in \COMP$; i.e., all curves of the form 
\begin{align*}
	\gamma_X\colon [r,r']\rightarrow F,\qquad t\mapsto X
\end{align*}
for some $X\in F$. 
\end{itemize}
\endgroup
\noindent
We define  
$\const+1:=\infty$, $\infty +1:=\infty$, $\lip+1:=1$; as well as 
\begin{align*}
	\qq_\infty^\lip(\gamma)&:= \max(\qq_\infty(\gamma), \Lip\big(\qq,\gamma)\big)\qquad\quad\hspace{82.4pt}\forall\: \gamma\in C^\lip([r,r'],F)\\
\qq^\dind_\infty(\gamma)&:=\sup\big\{\qq\big(\gamma^{(m)}(t)\big)\:\big|\: 0\leq m\leq \dind,\:\:t\in [r,r']\big\}\qquad\quad\hspace{0.15pt}\forall\: \gamma\in C^k([r,r'],F)\\
\qq_\infty(\gamma)&:=\qq_\infty^0(\gamma)\qquad\quad\hspace{158.7pt}\forall\: \gamma\in C^0([r,r'],F)
\end{align*}
for each $\qq\in \SEMM$, $k\in \NN\sqcup\{\infty,\const\}$, $\dind\llleq k$, and $[r,r']\in \COMP$ -- 
Here, $s\llleq k$ means 
\begingroup
\setlength{\leftmargini}{12pt}
\begin{itemize}
\item
	$\dind\equiveq \lip$\hspace{20pt} for\: $k\equiveq \lip$,
\item
	$\NN\ni \dind\leq k$\hspace{2.3pt}\: for\: $k\in \NN$, 
\item
	$\dind\in \NN$\hspace{3pt}\:\hspace{20pt} for\: $k\equiveq\infty$,
\item
	$\dind=0$\hspace{4.16pt}\:\hspace{20pt} for\: $k\equiveq \const$.
\end{itemize}
\endgroup
\noindent
The $C^k$-topology on $C^k([r,r'],F)$ for $k\in \NN\sqcup\{\lip,\infty,\const\}$ is the Hausdorff locally convex topology that is generated by the seminorms $\qq_\infty^\dind$ for all $\qq\in \SEMM$ and $\dind\llleq k$. 
\begin{remark}
In the Lipschitz case, the above conventions deviate from the conventions used, e.g., in \cite{RGM, AEH} as there the $\pp_\infty$-seminorms, i.e., the $C^0$-topology is considered on $C^\lip([r,r'],F)$.\hspace*{\fill}$\ddagger$  
\end{remark}
\noindent
Finally, we let $\CP^0([r,r'],F)$ denote the set of piecewise $C^0$-curves on $[r,r']\in \COMP$; i.e., all $\gamma\colon [r,r']\rightarrow F$, such that there exist $r=t_0<{\dots}<t_n=r'$ and
\begin{align}
\label{opsdopsposdopsdnmyxmnyxmybnymnyxuzwuzwe}
	\gamma[p]\in C^0([t_p,t_{p+1}],F)\qquad\text{with}\qquad\gamma|_{(t_p,t_{p+1})}=\gamma[p]|_{(t_p,t_{p+1})}\qquad\text{for}\qquad p=0,\dots,n-1.
\end{align}

\subsubsection{Lie Groups}
In this paper, $G$ will always denote an infinite dimensional Lie group in the sense of \cite{HG} (cf.\ Definition 3.1 and Definition 3.3 in \cite{HG}) that is modeled over the Hausdorff locally convex vector space $E$, with corresponding system of continuous seminorms $\SEM$. We  
denote the Lie algebra of $G$ by $(\mg, \bl \cdot,\cdot\br)$, fix a chart 
\begin{align*}
	\chart\colon G\supseteq \U\rightarrow \V\subseteq E,
\end{align*}
with $\V$ convex, $e\in \U$ and $\chart(e)=0$; 
 and define
\begin{align*}
	\ppp:= \pp\cp\dd_e\chart\qquad\quad\forall\: \pp\in \SEM. 
\end{align*}  
 We let $\mult\colon G\times G\rightarrow G$ denote the Lie group multiplication, $\RT_g:=\mult(\cdot, g)$ the right translation by $g\in G$, $\inv\colon G\ni g\mapsto g^{-1}\in G$ the inversion, and $\Ad\colon G\times \mg\rightarrow \mg$ the adjoint action -- i.e., we have
\begin{align*}
	\Ad(g,X)\equiv \Ad_g(X):=\dd_e\conj_g(X)\qquad\quad\text{with}\qquad\quad \conj_g\colon G\ni h\mapsto g\cdot  h\cdot g^{-1}\in G
\end{align*}
for each $g\in G$ and $X\in  \mg$. We furthermore recall the product rule
\begin{align}
\label{LGPR}
	\dd_{(g,h)}\mult(v,w)= \dd_g\RT_h(v) + \dd_h\LT_g(w)\qquad\quad\forall\: g,h\in G,\:\: v\in T_gG,\:\: w\in T_h G.
\end{align}

\subsubsection{Uniform Limits}
Let $\mu\in \Map([r,r'],G)$, $\{\mu_n\}_{n\in \NN}\subseteq  \Map([r,r'],G)$, and $\{\mu_h\}_{h\in \MD_\delta}\subseteq  \Map([r,r'],G)$ for $\delta>0$ be given. We write
\begingroup
\setlength{\leftmargini}{12pt}
\begin{itemize}
\item
	$\limin \mu_n=\mu$\quad \hspace{5.15pt} \defff\quad for each open neighbourhood $U\subseteq G$ of $e$, there exists some $n_U\in \NN$ with 
	
	\quad\hspace{89.2pt}$\mu^{-1}\cdot \mu_n \in U$ for each $n\geq n_U$. 
\item
	$\limih \mu_h=\mu$\quad \defff\quad for each open neighbourhood $U\subseteq G$ of $e$, there exists some $0<\delta_U<\delta$ 
	
	\quad\hspace{89.2pt}with $\mu^{-1}\cdot \mu_h \in U$ for each $h\in \MD_{\delta_U}$.
\end{itemize}
\endgroup
\noindent
Evidently, then we have
\begin{lemma}
\label{aslksalksalksakasklaslksalklkalksa}
	Suppose $\delta>0$ and $\{\mu_h\}_{h\in \MD_\delta}\subseteq  C^0([r,r'],G)$ are given. If $\limin \mu_{h_n}=e$ holds for each sequence $\MD_\delta \supseteq \{h_n\}_{n\in \NN}\rightarrow 0$, then we have $\limih \mu_h=e$.
\end{lemma}
\begin{proof}
If the claim is wrong, then there exists a neighbourhood $U\subseteq G$ of $e$, such that the following holds: For each $n\in \NN$, there exists some $h_n\neq 0$ with $|h_n|\leq \frac{1}{n}$ as well as some $\tau_n\in [r,r']$, such that $\mu^{-1}(\tau_n)\cdot \mu_{h_n}(\tau_n)\notin U$ holds. Since we have $\{h_n\}_{n\in \NN}\rightarrow 0$, this contradicts the presumptions.
\end{proof}
\noindent
The same conventions (and Lemma \ref{aslksalksalksakasklaslksalklkalksa}) also hold if $(G,\cdot)\equiv (F,+)$ is a Hausdorff locally convex vector space (or its completion) -- In this case, we use the following convention:
\vspace{6pt}

\noindent
Let $\delta>0$ and $\alpha\colon \MD_\delta\times [r,r']\rightarrow F$ be given, with $\alpha(h,\cdot)\in \Map([r,r'],F)$ for each $h\in \MD_\delta$. 
Then, for $\beta\in \Map([r,r'],F)$, we write 
\begin{align*}
	\textstyle\frac{\dd}{\dd h}\big|^\infty_{h=0}\: \alpha=\beta\qquad\quad\stackrel{\text{def.}}{\Longleftrightarrow}\qquad\quad \limih \big[1/h\cdot \alpha(h,\cdot )\big] =\beta.
\end{align*}
\begin{remark}
\label{fdfdfdfdfdfdnbcxnbcxnbcbnx}
In this paper, the above convention will mainly be used in the following form: $F=\ovl{E}$ will be the completion of a Hausdorff locally convex vector space $E$; and we will have  
$\alpha\colon \MD_\delta\times [r,r']\rightarrow E\subseteq \ovl{E}$ as well as $\beta\in \Map([r,r'],\ovl{E})$. \hspace*{\fill}$\ddagger$
\end{remark}
\subsection{The Evolution Map}
In this subsection, we provide the relevant facts and definitions concerning the right logarithmic derivative and the evolution map. 
\subsubsection{Basic Definitions}
\label{dsdsdsdsdspopopo}
We define 
\begin{align*}
		C_*^k([r,r'],G):=\{\mu\in C^k([r,r'],G)\:|\: \mu(r)=e\}\qquad\quad\forall\: [r,r']\in \COMP,\:\: k\in \NN\sqcup\{\infty\}.
\end{align*}  
The {right logarithmic derivative} is given by
\begin{align*}
	\Der\colon C^1([r,r'],G)\rightarrow C^0([r,r'],\mg),\qquad \mu\mapsto \dd_\mu\RT_{\mu^{-1}}(\dot \mu)
\end{align*}
for each $[r,r']\in \COMP$; and we define $\DIDE_{[r,r']}:=\Der(C^1([r,r'],G))$ for each $[r,r']\in \COMP$, as well as
\begin{align*}
\textstyle\DIDE_{[r,r']}^k:=\DIDE_{[r,r']}\cap C^k([r,r'],\mg)\qquad\quad\forall\: [r,r']\in \COMP,\:\: k\in \NN\sqcup\{\lip,\infty,\const\}.
\end{align*}
Then, $\Der$ restricted to $C_*^1([r,r'],G)$ is injective for each $[r,r']\in \COMP$; so that
\begin{align*}
	\textstyle\EV\colon \bigsqcup_{[r,r']\in \COMP}\DIDE_{[r,r']}\rightarrow \bigsqcup_{[r,r']\in \COMP}C_*^1([r,r'],G)
\end{align*}
is well defined by 
\begin{align*}
	\EV\colon \DIDE_{[r,r']}\rightarrow C_*^{1}([r,r'],G),\qquad\Der(\mu)\mapsto \mu\cdot \mu(r)^{-1}
\end{align*}  
for each $[r,r']\in \COMP$. Here,
\begin{align*}
	\EV|_{\DIDE_{[r,r']}^k}\colon \DIDE_{[r,r']}^k\rightarrow C^{k+1}([r,r'],G)
\end{align*}
holds for each $[r,r']\in \COMP$, and each $k\in \NN\sqcup\{\lip,\infty,\const\}$.
\subsubsection{The Product Integral}
\label{wqwqztwzuwqzuiusxy}
The product integral is given by
\begin{align*}
	\textstyle\innt_s^t\phi:= \EV\big(\phi|_{[s,t]}\big)(t)\in G\qquad\quad \forall \:[s,t]\subseteq \dom[\phi],\:\: \phi\in \bigsqcup_{[r,r']\in \COMP}\DIDE_{[r,r']};
\end{align*}
and we let $\innt\phi\equiv\innt_r^{r'}\phi$ as well as $\innt_c^c\phi:= e$ for $\phi\in \DIDE_{[r,r']}$ and $c\in [r,r']$. Moreover, we set
\begin{align*}
\evol_{[r,r']}^k&\textstyle\equiv \innt\big|_{\DIDE_{[r,r']}^k}\qquad\quad\forall\:  k\in \NN\sqcup\{\lip,\infty,\const\} ,\:\:[r,r']\in \COMP; 
\end{align*} 
and let $\evol_\kk\equiv \evol_{[0,1]}^k$ as well as $\DIDED_\kk\equiv \DIDE^k_{[0,1]}$ for each $k\in \NN\sqcup\{\lip,\infty,\const\}$.
We furthermore let
\begin{align*}
	\evol\equiv \evol_{\mathrm{0}}\colon \DIDED\equiv \DIDED_0\rightarrow G.
\end{align*}
Then, we have the following elementary identities, cf.,  \cite{HGGG,MK} or Sect.\ 3.5.2 in \cite{RGM} 
\vspace{2pt}
\begingroup
\setlength{\leftmargini}{17pt}
{
\renewcommand{\theenumi}{\emph{\alph{enumi})}} 
\renewcommand{\labelenumi}{\theenumi}
\begin{enumerate}
\item
\label{kdsasaasassaas}
For each $\phi,\psi\in \DIDE_{[r,r']}$, we have $\phi+\Ad_{\innt_r^\bullet\phi}(\psi)\in \DIDE_{[r,r']}$, with
\begin{align*}
	\textstyle\innt_r^t \phi \cdot \innt_r^t\psi=\innt_r^t \phi+\Ad_{\innt_r^\bullet\phi}(\psi).
\end{align*}
	\vspace{-18pt}
\item
\label{kdskdsdkdslkds}
For each $\phi,\psi\in \DIDE_{[r,r']}$, we have $\Ad_{[\innt_r^\bullet\phi]^{-1}}(\psi-\phi)\in \DIDE_{[r,r']}$, with
\begin{align*}
	\textstyle\big[\!\innt_r^t \phi\big]^{-1} \big[\innt_r^t\psi\big]=\innt_r^t\Ad_{[\innt_r^\bullet\phi]^{-1}}(\psi-\phi).
\end{align*}
	\vspace{-18pt}
\item
\label{pogfpogf}
\hspace{4pt}For $r=t_0<{\dots}<t_n=r'$ and $\phi\in \DIDE_{[r,r']}$, we have 
	\begin{align*}
		\textstyle\innt_r^t\phi=\innt_{t_{p}}^t\! \phi\cdot \innt_{t_{p-1}}^{t_{p}} \!\phi \cdot {\dots} \cdot \innt_{r}^{t_1}\!\phi\qquad\quad\forall\:t\in (t_p,t_{p+1}],\:\: p=0,\dots,n-1.
	\end{align*}
		\vspace{-15pt}
\item
\label{subst}
	\hspace{4pt}For $\varrho\colon [\ell,\ell']\rightarrow [r,r']$ 
of class $C^1$ and $\phi\in \DIDE_{[r,r']}$, we have $\dot\varrho\cdot (\phi\cp\varrho)\in \DIDE_{[\ell,\ell']}$, with
\begin{align*}
	 \textstyle\innt_r^{\varrho}\phi=\big[\innt_\ell^\bullet\dot\varrho\cdot (\phi\cp\varrho)\he\big]\cdot \big[\innt_r^{\varrho(\ell)}\phi\he\big].
\end{align*} 
\item
\label{homtausch}
For each homomorphism $\Psi\colon G\rightarrow H$ between Lie groups $G$ and $H$ that is of class $C^1$, we have
\begin{align*}
	\textstyle\Psi\cp \innt_r^\bullet \phi = \innt_r^\bullet \dd_e\Psi\cp\phi\qquad\quad\forall\:\phi\in \DIDE_{[r,r']}.
\end{align*}
\end{enumerate}}
\endgroup
\noindent	
We say that $G$ is {\bf $\boldsymbol{C^k}$-semiregular} for $k\in \NN\sqcup\{\lip,\infty\}$ \defff $\DIDED_\kk=C^k([0,1],\mg)$ holds; which, by \textrm{\ref{subst}}, is equivalent to that $\DIDE^k_{[r,r']}=C^k([r,r'],\mg)$ holds for each $[r,r']\in \COMP$, cf.\, e.g., Lemma 12 in \cite{RGM}.
\subsubsection{The Exponential Map}
\label{kfdlkfdlkfdscvpdfpofdofd}
The exponential map is defined by 
\begin{align*}
	\textstyle\exp\colon \dom[\exp]\equiv \expal^{-1}(\DIDED_\const)\rightarrow G,\qquad X\mapsto \innt \phi_X|_{[0,1]}\equiveq (\evol_\const\cp \expal)(X)
\end{align*}
with $\expal\colon \mg\ni X\rightarrow \phi_X|_{[0,1]}\in C^\const([0,1],\mg)$. 
\vspace{6pt}

\noindent
Then, instead of saying that $G$ is $C^\const$-semiregular, in the following we will rather say that $G$ admits an exponential map. 
We furthermore remark that \textrm{\ref{subst}} implies $\RR\cdot \dom[\exp]\subseteq \dom[\exp]$; and that $t\mapsto \exp(t\cdot X)$ is a $1$-parameter group for each $X\in \dom[\exp]$, with
\begin{align}
\label{odaidaooipidadais}
	\textstyle\exp(t\cdot X)\equiveq \innt t\cdot \phi_X|_{[0,1]}\stackrel{\textrm{\ref{subst}}}{=} \innt_0^t\phi_X|_{[0,1]}\qquad\quad\forall\:  t\geq 0, 
\end{align}
cf., e.g., Remark 2.1) in \cite{RGM}. Finally, if $G$ is abelian, then  $X+Y\in \dom[\exp]$ holds for all $X,Y\in \dom[\exp]$, because we have
\begin{align*}
	\textstyle\exp(X)\cdot \exp(Y)\stackrel{\textrm{\ref{kdsasaasassaas}}}{=}\innt \phi_X|_{[0,1]} \cdot \innt \phi_Y|_{[0,1]}  = \innt \phi_{X+Y}|_{[0,1]}.
\end{align*}

\subsubsection{Standard Topologies}
\label{jkdsodsoioidsoids}
We say that $G$ is \textbf{$\boldsymbol{C^k}$-continuous} for $k\in \NN\sqcup\{\lip,\infty,\const\}$ \defff $\evol_{\kk}$ is continuous w.r.t.\ the $C^k$-topology. We explicitly remark that under the identification $\expal\colon \mg\rightarrow \{\phi_X|_{[0,1]}\:|\: X\in \mg\}$, the $C^\const$-topology just equals the subspace topology on $\dom[\exp]$ that is inherited by the locally convex topology on $\mg$. So, instead of saying that $G$ is $C^\const$-continuous \defff $\evol_\const$ is continuous w.r.t.\ this topology, we will rather say that the exponential map is continuous.
  
\subsection{The Riemann Integral}
\label{opsdpods}
Let $F$ be a Hausdorff locally convex vector space with system of continuous seminorms $\SEMM$, and completion $\ovl{F}$. For each $\qq\in \SEM$, we let $\ovl{\qq}\colon \ovl{F}\rightarrow \RR_{\geq 0}$ denote the continuous extension of $\qq$ to $\ovl{F}$. 
The Riemann integral of $\gamma\in C^0([r,r'],F)$ (for $[r,r']\in \COMP$) is denoted by 
$\int \gamma(s) \:\dd s\in \comp{F}$; and we define 
\begin{align}
\label{fdopfdpo}
	\textstyle\int_a^b \gamma(s)\:\dd s:= \int \gamma|_{[a,b]}(s) \:\dd s,\qquad\quad\:\int_b^a \gamma(s) \:\dd s:= - \int_a^b \gamma(s) \:\dd s,\qquad\quad\:
	 \int_c^c \gamma(s)\: \dd s:=0\qquad
\end{align}
for $r\leq a<b\leq r'$ and $c\in [r,r']$. Clearly, the Riemann integral is linear, with	
\begin{align}
\label{isdsdoisdiosd0}
	\textstyle\int_a^c \gamma(s) \:\dd s&\textstyle=\int_a^b \gamma(s)\:\dd s+ \int_b^c \gamma(s)\:\dd s\qquad\quad\he \forall\: r\leq a< b< c\leq r'\\
	\label{isdsdoisdiosd}
		\textstyle\gamma-\gamma(r)\hspace{4pt}&\textstyle=\int_r^\bullet \dot\gamma(s)\:\dd s\qquad\quad\hspace{60.8pt}\forall\:  \gamma\in C^1([r,r'],F),\\
	\label{isdsdoisdiosd1}
		\qq(\gamma-\gamma(r))&\textstyle\leq \int_r^\bullet \qq(\dot\gamma(s))\: \dd s\qquad\quad\hspace{46.8pt}\forall\: \qq\in \SEMM,\:\: \gamma\in C^1([r,r'],F),
	\end{align} 
as well as
	\begin{align}
	\label{ffdlkfdlkfd}
		\textstyle\ovl{\qq}\big(\int_r^\bullet \gamma(s)\:\dd s\big)\leq \int_r^\bullet \qq(\gamma(s))\:\dd s\qquad\quad\forall\: \qq\in \SEMM,\:\: \gamma\in C^0([r,r'],F).
	\end{align}	
	We furthermore have the substitution formula 
	\begin{align}
\label{substitRI}
	\textstyle\int_r^{\varrho(t)} \gamma(s)\: \dd s=\int_\ell^t \dot\varrho(s)\cdot (\gamma\cp\varrho)(s)\:\dd s
\end{align}
for each $\gamma\in C^0([r,r'],F)$, and each $\varrho\colon \COMP\ni [\ell,\ell'] \rightarrow [r,r']$ of class $C^1$ with $\varrho(\ell)=r$ and $\varrho(\ell')=r'$.
Moreover, if $E$ is a Hausdorff locally convex vector space, and 
$\mathfrak{L}\colon F\rightarrow E$ is a continuous linear map, then we have
\begin{align}
\label{pofdpofdpofdsddsdsfd}
	\textstyle\int \gamma(s)\:\dd s\in F \quad\text{for}\quad \gamma\in C^0([r,r'],F)\qquad\:\:\Longrightarrow\qquad\:\: 
	\textstyle \mathfrak{L}(\int \gamma(s)\:\dd s)=\int \mathfrak{L}(\gamma(s)) \: \dd s.
\end{align}
Finally, for $\gamma\in \CP^0([r,r'],F)$ with $\gamma[0],\dots,\gamma[n-1]$ as in \eqref{opsdopsposdopsdnmyxmnyxmybnymnyxuzwuzwe}, we define
\begin{align}
\label{opofdpopfd}
	\textstyle\int \gamma(s)\:\dd s:=\sum_{p=0}^{n-1}\int \gamma[p](s)\:\dd s. 
\end{align} 
A standard refinement argument in combination with \eqref{isdsdoisdiosd0} then shows that this is well defined; i.e., independent of any choices we have made. We define $\int_a^b\gamma(s)\: \dd s$, $\int_b^a\gamma(s)\: \dd s$ and $\int_c^c \gamma(s)\: \dd s$ as in \eqref{fdopfdpo}; and observe that \eqref{opofdpopfd} is linear and fulfills \eqref{isdsdoisdiosd0}.

\subsection{Standard Facts and Estimates}
\label{fsdsddssddscxycyxycvvnmkjkj}
Let $F_1,\dots,F_n,E$ be Hausdorff locally convex vector spaces with corresponding system of continuous seminorms $\SEMM_1,\dots,\SEMM_n,\SEM$. 
We recall that
\begin{lemma}
\label{alalskkskaskaskas}
Let $X$ be a topological space; and let $\Phi\colon X\times F_1\times{\dots}\times F_n\rightarrow E$ be continuous with $\Phi(x,\cdot)$ 
	$n$-multilinear for each $x\in X$. Then, for each compact $\compacto\subseteq X$ and each $\pp\in \SEM$, there exist seminorms $\qq_1\in \SEMM_1,\dots,\qq_n\in \SEMM_n$ as well as $O\subseteq X$ open with $\compacto\subseteq O$, such that
	\begin{align*}
		(\pp\cp\Phi)(y,X_1,\dots,X_n) \leq \qq_1(X_1)\cdot {\dots}\cdot \qq_n(X_n)\qquad\quad\forall\: y\in O
	\end{align*} 
	holds for all $X_1\in F_1,\dots,X_n\in F_n$. 
\end{lemma}
\begin{proof}
Confer, e.g., Corollary 1 in \cite{RGM}.
\end{proof}
\noindent
Next, given  Hausdorff locally convex vector spaces $F_1,F_2$, and a continuous linear map $\Phi\colon F_1\rightarrow F_2$, we denote its unique continuous linear  extension by $\comp{\Phi}\colon \comp{F}_1\rightarrow \comp{F}_2$ (cf., 2.\ Theorem in Sect.\ 3.4 in \cite{JAR}).
We recall that
\begin{lemma}
\label{sdsdds}
Let $F_1,F_2$ be Hausdorff locally convex vector spaces; and let $f\colon F_1\supseteq U\rightarrow F_2$ be of class $C^{2}$. Suppose that 
$\gamma\colon D\rightarrow F_1\subseteq \ovl{F}$ is continuous at $t\in D$, such that $\lim_{h\rightarrow 0} 1/h \cdot (\gamma(t+h)-\gamma(t))=:X\in \ovl{F}_1$ exists. Then, we have
\begin{align*}
	\textstyle\lim_{h\rightarrow 0} 1/h\cdot (f(\gamma(t+h))-f(\gamma(t)))=\ovl{\dd_{\gamma(t)}f}\he(X).
\end{align*}  
\end{lemma}
\begin{proof}
Confer, e.g., Lemma 7 in \cite{RGM}.
\end{proof}

\begin{remark}
\label{lkcxlkcxlkxlkcxpoidsoids}
Let $F$ be a Hausdorff locally convex vector space, let $U\subseteq F$ be open, and let $G$ be a Lie group. A map $f\colon U\rightarrow G$ is said to be   
\begingroup
\setlength{\leftmargini}{12pt}
\begin{itemize}
\item
differentiable at $x\in U$ \defff there exists a chart $(\chart',\U')$ of $G$ with $f(x)\in \U'$, such that 
\begin{align}
\label{pofdpofdpofdfd}
	\textstyle (D^{\chart'}_v f)(x):=\lim_{h\rightarrow 0} 1/h\cdot ((\chart'\cp f)(x+h\cdot v)-(\chart'\cp f)(x))\in E\qquad\quad\forall\: v\in F
\end{align}
exists. Then, Lemma \ref{sdsdds} applied to coordinate changes shows that \eqref{pofdpofdpofdfd} holds for one chart around $f(x)$ \deff it holds for each chart around $f(x)$ -- and that
\begin{align*}
	\ddd_x f(v):= \big(\dd_{\chart'(f(x))}\chart'^{-1} \cp (D^{\chart'}_v f)\big)(x)\in T_{f(x)}G \qquad\quad\forall\: v\in F	
\end{align*}  
is independent of the explicit choice of $(\chart',U')$.
\item
differentiable \defff $f$ is differentiable at each $x\in U$. \hspace*{\fill}$\ddagger$
\end{itemize}
\endgroup
\end{remark}
\noindent
In particular, Lemma \ref{alalskkskaskaskas} provides us with the following statements (cf.\ also Sect.\ 3.4.1 in \cite{RGM}): 
\begingroup
\setlength{\leftmargini}{22pt}
{
\renewcommand{\theenumi}{\Roman{enumi})} 
\renewcommand{\labelenumi}{\theenumi}
\begin{enumerate}
\item
\label{as1}
Since $\Ad\colon G\times \mg\ni (g,X)\mapsto \Ad_g(X)\in \mg$ is smooth as well as linear in the second argument (by Lemma \ref{alalskkskaskaskas}), to each compact $\compact\subseteq G$ and each $\vv\in \SEM$, there exists some $\vv\leq \ww\in \SEM$, such that  
	$\vvv\cp \Ad_g\leq \www$ holds for each  
	$g\in \compact$. 
\item
\label{as5}
By Lemma \ref{alalskkskaskaskas} applied to $\Phi\equiv \Ad$ and $\compacto\equiv \{e\}$, to each $\mm\in \SEM$, there exists some $\mm\leq \qq\in \SEM$, as well as $O\subseteq G$ symmetric open with $e\in O$, such that 
	$\mmm\cp \Ad_g\leq \qqq$ holds for each 
	$g\in O$.
\item
\label{as2}
Suppose that $\im[\mu]\subseteq \U$ holds for $\mu\in C^1([r,r'],G)$. Then, we have 
\begin{align}
\label{kldlkdldsl}
	\Der(\mu)=\dermapdiff(\chart\cp\mu,\partial_t(\chart\cp\mu)),
\end{align}      
for the smooth map 
	$\dermapdiff\colon \V\times E\ni (x,X)\mapsto \dd_{\chartinv(x)}\RT_{[\chartinv(x)]^{-1}}(\dd_x\chartinv(X))\in \mg$. 
Since $\dermapdiff$ is linear in the second argument, (by Lemma \ref{alalskkskaskaskas}) for each $\qq\in \SEM$, there exists some $\qq\leq \mm\in \SEM$ with
\begin{align}
\label{omegaklla}
	(\qqq\cp\dermapdiff)(x,X)\leq \mm(X)\qquad\quad\forall\: x\in \OB_{\mm,1},\:\: X\in E.
\end{align} 
\item
\label{as3}
Suppose that $\im[\mu]\subseteq \U$ holds for $\mu\in C^1([r,r'],G)$. Then, we have
\begin{align}
\label{ixxxsdsdoisdiosd}
	\partial_t\he(\chart\cp\mu)=\dermapinvdiff(\chart\cp\mu,\Der(\mu)),
\end{align}   
for the smooth map 
	$\dermapinvdiff\colon \V\times \mg\ni(x,X)\mapsto \big(\dd_{\chartinv(x)}\chart\cp \dd_{e}\RT_{\chartinv(x)}\big)(X)\in E$. 
Since $\dermapinvdiff$ is linear in the second argument, (by Lemma \ref{alalskkskaskaskas}) for each $\qq\in \SEM$, there exists some $\uu\leq\mm\in \SEM$ with
\begin{align*}
	(\uu\cp \dermapinvdiff)(x,X)\leq \mmm(X)\qquad\quad\forall\: x\in \OB_{\mm,1},\:\: X\in \mg.
\end{align*}
For each $\mu\in C^1([r,r'],G)$ with $\im[\chart\cp\mu]\subseteq \OB_{\mm,1}$, we thus obtain from \eqref{ixxxsdsdoisdiosd}, \eqref{isdsdoisdiosd}, and \eqref{isdsdoisdiosd1} that
\begin{align}
\label{sadsndsdsnmdsdsa}
\begin{split}
	\textstyle\uu(\chart\cp\mu)&\stackrel{}{=}\textstyle\uu\big(\int_r^\bullet \dermapinvdiff((\chart\cp\mu)(s),\Der(\mu)(s))\:\dd s\big)\textstyle\leq \int_r^\bullet \mmm(\Der(\mu)(s))\:\dd s.
\end{split}
\end{align}
\end{enumerate}}
\endgroup
\noindent
For instance, we immediately obtain from \eqref{sadsndsdsnmdsdsa} that
\begin{lemma}
\label{lkdslklkdslkdsldsds}
For each $\uu\in \SEM$, there exist $\uu\leq \mm\in \SEM$, and $U\subseteq G$ open with $e\in U$, such that
\begin{align*}
	\textstyle(\uu\cp\chart)(\innt_r^\bullet \chi)\leq \int_r^\bullet \mm(\chi(s))\: \dd s
\end{align*}
holds, for each $\chi\in \DIDE_{[r,r']}$ with $\innt_r^\bullet \chi \in U$; for all $[r,r']\in \COMP$. 
\end{lemma}
\noindent
Moreover,
\begin{lemma}
\label{Adlip}
We have $\Ad_\mu(\phi)\in C^k([r,r'],\mg)$ for each $\mu\in C^{k+1}([r,r'],G)$, $\phi\in C^k([r,r'],\mg)$, and $k\in \NN\sqcup\{\lip,\infty\}$.
\end{lemma}
\begin{proof}
Confer, e.g., Lemma 13 in \cite{RGM}.
\end{proof}
\begin{lemma}
\label{opopsopsdopds}
	Let $[r,r']\in \COMP$, $k\in \NN\sqcup\{\infty\}$, and $\phi\in \DIDE_{[r,r']}^k$ be fixed. Then, 
	for each $\pp\in \SEM$ and $\dind\llleq k$, there exists some $\pp\leq \qq\in \SEM$ with
		\begin{align*}
		\ppp_\infty^\dindp\big(\Ad_{[\innt_r^\bullet\phi]^{-1}}(\psi)\big)\leq \qqq_\infty^\dindp(\psi)\qquad\quad\forall\: \psi\in C^k([r,r'],\mg),\:\: 0\leq \dindp\leq \dind.
	\end{align*}
\end{lemma}
\begin{proof}
\noindent
Confer, e.g., Lemma 14 in \cite{RGM}.
\end{proof}
\noindent
Then, modifying the argumentation used in the proof of the Lipschitz case in Lemma 13 in \cite{RGM} to our deviating convention concerning the topology on the set of Lipschitz curves, we also obtain
\begin{lemma}
	\label{xccxcxcxcxyxycllvovo}
	Let $[r,r']\in \COMP$, and $\phi\in \DIDE_{[r,r']}$ be fixed. Then, 
	for each $\pp\in \SEM$, there exists some $\pp\leq \qq\in \SEM$ with
		\begin{align*}
		\ppp^\lip_\infty\big(\Ad_{[\innt_r^\bullet\phi]^{-1}}(\psi)\big)\leq \qqq_\infty^\lip(\psi)\qquad\quad\forall\: \psi\in C^\lip([r,r'],\mg).
	\end{align*}
\end{lemma}
\begin{proof}
Confer Appendix \ref{asassadsdsdsdsdsdsds}.
\end{proof}

\subsection{Continuity Statements}
\label{ooiufdiuafjnfdsxyncx}
For $h\in G$, we define $\chart_h(g):=\chart(h^{-1}\cdot g)$ for each $g\in h\cdot \U$; 
and recall that, cf.\ Lemma 8 in \cite{RGM}
\begin{lemma}
\label{fhfhfhffhaaaa}
Let $\compact\subseteq \U$ be compact. Then, for each $\pp\in\SEM$, there exists some $\pp\leq \uu\in \SEM$, and a symmetric open neighbourhood $V\subseteq\U$ of $e$ with $\compact\cdot V\subseteq \U$ and $\OB_{\uu,1}\subseteq \chart(V)$, 
such that 
\begin{align*}
	\pp(\chart(q)-\chart(q'))\leq \uu(\chart_{g\cdot h}(q)-\chart_{g\cdot h}(q'))\qquad\quad \forall\: q,q'\in g\cdot V,\:\: h\in V
\end{align*} 
holds for each $g\in \compact$. 
\end{lemma}
\noindent
Now, combining Lemma \ref{lkdslklkdslkdsldsds} with Lemma \ref{fhfhfhffhaaaa}, we obtain the following variation of Proposition 1 in \cite{RGM}:
\begin{lemma}
	\label{hghghggh}
	For each $\pp\in \SEM$, there exist $\pp\leq\qq\in \SEM$ and $V\subseteq G$ open with $e\in V$, such that 
		\begin{align*}
		\textstyle\pp\big(\chart\big(\innt_r^\bullet\phi\big)-\chart\big(\innt_r^\bullet\psi\big)\big)\leq \int_r^\bullet\qqq(\phi(s)-\psi(s))\:\dd s
	\end{align*}
	holds for all $\phi,\psi\in \DIDE_{[r,r']}$ with $\innt_r^\bullet\phi,\:\innt_r^\bullet\psi\in V$; for each $[r,r']\in \COMP$.
	\end{lemma}
	\begin{proof}
	We let $\pp\leq\uu\in \SEM$ and $V$ be as in Lemma \ref{fhfhfhffhaaaa} for $\compact\equiv \{e\}$ there (i.e., $V$ is symmetric with $\OB_{\uu,1}\subseteq \chart(V)$). We choose $U\subseteq G$ and $\uu\leq \mm\in \SEM$ as in Lemma \ref{lkdslklkdslkdsldsds}. We furthermore let $\mm\leq \qq\in \SEM$ and $O\subseteq G$ be as in \ref{as5}. Then, shrinking $V$ if necessary, we can assume that $V^{-1}\cdot V \subseteq U$ as well as $V\subseteq O$ holds.   
Then, for $\phi,\psi$ as in the presumptions, Lemma \ref{fhfhfhffhaaaa} applied to $q\equiv\innt_r^\bullet\phi,\:q'\equiv\innt_r^\bullet\psi,\: h\equiv\innt_r^\bullet\phi\in V$, and $g\equiv e$ gives 
	\begin{align*}
		\textstyle\pp\big(\chart\big(\innt_r^\bullet\phi\big)-\chart\big(\innt_r^\bullet\psi\big)\big)\leq
		\textstyle\uu\big(\chart_{\innt_r^\bullet\phi}\big(\innt_r^\bullet\phi\big)-\chart_{\innt_r^\bullet\phi}\big(\innt_r^\bullet\psi\big)\big)= \textstyle(\uu\cp \chart)\big([\innt_r^\bullet\phi]^{-1}[\innt_r^\bullet\psi]\big).
	\end{align*} 
	By assumption, for each $t\in [r,r']$, we have
	\begin{align*}
		\textstyle U\supseteq V^{-1}\cdot V\ni [\innt_r^t\phi]^{-1}[\innt_r^t\psi]\stackrel{\textrm{\ref{kdskdsdkdslkds}}}{=} \innt_r^t \Ad_{[\innt_r^\bullet \phi]^{-1}}(\psi-\phi)\qquad\text{with}\qquad [\innt_r^\bullet \phi]^{-1}\in V^{-1}=V\subseteq  O.
	\end{align*}
	We obtain from  Lemma \ref{lkdslklkdslkdsldsds} and \ref{as5} that
\begin{align*}
	\textstyle(\uu\cp \chart)\big([\innt_r^t\phi]^{-1}[\innt_r^t\psi]\big)\leq \int_r^{t}\mmm\big(\Ad_{[\innt_r^s \phi]^{-1}}(\psi(s)-\phi(s))\big)\:\dd s \leq \int_r^{t}\qqq(\psi(s)-\phi(s))\:\dd s
\end{align*}  
holds for each $t\in [r,r']$; which proves the claim.
\end{proof}
\noindent
We furthermore observe that
\begin{lemma}
\label{fddfxxxxfd}
Suppose that $\exp\colon \dom[\exp]\rightarrow G$ is continuous; and let $X\in \dom[\exp]$ be fixed. Then, for each open neighbourhood $V\subseteq G$ of $e$,  
there exists some $\mm\in \SEM$, such that 
\begin{align*}
	\textstyle\mmm(Y-X)\leq 1\quad\:\text{for}\quad\: Y\in \dom[\exp]\qquad\quad\Longrightarrow\qquad\quad \innt_0^\bullet \phi_Y|_{[0,1]} \in \innt_0^\bullet \phi_X|_{[0,1]}\cdot V.
\end{align*} 
\end{lemma}
\begin{proof}
	By assumption,
	$\alpha\colon [0,1]\times\dom[\exp]\ni (t,Y)\mapsto \exp(t\cdot X)^{-1}\cdot \exp(t\cdot Y)$ 	
	is continuous; and we have $\alpha(\cdot,X)=e$. For $\tau\in [0,1]$ fixed, there thus exists an open interval $I_\tau\subseteq \RR$ containing $\tau$, as well as an open neighbourhood $O_\tau\subseteq \mg$ of $X$, such that we have
\begin{align}
\label{sdopaopsaopopasasas}
	\exp(t\cdot X)^{-1}\cdot \exp(t\cdot Y)\in V\qquad\quad\forall\: t\in I_\tau\cap [0,1],\:\: Y\in O_\tau\cap \dom[\exp].
\end{align} 
We choose $\tau_1,\dots,\tau_n\in [0,1]$ with $[0,1]\subseteq I_{\tau_1}\cup{\dots}\cup I_{\tau_n}$; and define $O:=O_{\tau_1}\cap{\dots}\cap O_{\tau_n}$. Then,   
\eqref{sdopaopsaopopasasas} holds for each $t\in [0,1]$ and $Y\in O\cap \dom[\exp]$; so that the claim holds for each fixed $\mm\in \SEM$ with $\OB_{\mm,1}\subseteq O$.
\end{proof}

\section{Auxiliary Results}
\label{mnnxcncyxciuoduioasd}
In this section, we introduce the continuity notions that we will need to formulate our main results.  
We furthermore provide some elementary continuity statements that we will need in the main text. 

\subsection{Sets of Curves}
Let $[r,r']\in \COMP$ be fixed.  We will tacitly use in the following that $C^k([r,r'],\mg)$ is a real vector space for each $k\in \NN\sqcup\{\lip,\infty,\const\}$. We will furthermore use that:
\begingroup
\setlength{\leftmargini}{19pt}
{\renewcommand{\theenumi}{\Alph{enumi})} 
\renewcommand{\labelenumi}{\theenumi}
\begin{enumerate}
\item
\label{remkkk2}
	For each $k\in \NN\sqcup\{\lip,\infty\}$, $\phi\in \DIDE^k_{[r,r']}$, and $\psi\in C^k([r,r'],\mg)$,  
	we have $\Ad_{[\innt_r^\bullet\phi]^{-1}}(\psi)\in C^k([r,r'],\mg)$ 
	by Lemma \ref{Adlip}. Evidently, the same statement also holds for $k\equiveq \const$ if $G$ is abelian.
\item
\label{remkkk3}
	For each $k\in \NN\sqcup\{\lip,\infty,\const\}$, $\phi\in \DIDE^k_{[r,r']}$, $[\ell,\ell']\in \COMP$, and  
	\begin{align*}
		\textstyle\varrho\colon [\ell,\ell']\rightarrow [r,r'],\qquad t\mapsto r + (t-\ell)\cdot (r'-r)\slash (\ell'-\ell),
	\end{align*}
	we have $\dot\varrho\cdot (\phi\cp\varrho)= (r'-r)\slash (\ell'-\ell) \cdot (\phi\cp\varrho) \in \DIDE^k_{[\ell,\ell']}$ by \textrm{\ref{subst}}; with
\begin{align*}
	\textstyle\ppp_\infty^\dind(\dot\varrho\cdot (\phi\cp\varrho))&\textstyle=\Big[\frac{(r'-r)}{(\ell'-\ell)}\Big]^{\dind+1}\cdot \ppp_\infty^\dind(\phi)\quad\text{with}\quad\dind\llleq k\quad\text{for}\quad k\in \NN\sqcup\{\infty,\const\},\\[4pt]
	\textstyle\Lip(\ppp,\dot\varrho\cdot(\phi\cp\varrho))&\textstyle=\Big[\frac{(r'-r)}{(\ell'-\ell)}\Big]^{2}\cdot \Lip(\ppp, \phi)\quad\:\text{for}\quad\: k\equiveq \lip.
\end{align*}
\end{enumerate}}
\endgroup
\noindent
We say that $\mg$ is  {\bf k-complete} for $k\in \NN\sqcup\{\lip,\infty,\const\}$ \defff
\begin{align}
\label{dspopodspds}
	\textstyle\int \Ad_{[\innt_r^s \phi]^{-1}}(\chi(s))\:\dd s\in \mg 
\end{align}
holds for all $\phi,\chi\in \DIDE^k_{[r,r']}$, for each $[r,r']\in \COMP$. Then,
\begin{remark}
\label{fdfdfdfdfdscxycxycx}
\begingroup
\setlength{\leftmargini}{12pt}
\begin{itemize}
\item[]
\item
$\mg$ is $\const$-complete if $G$ is abelian.
\item
$\mg$ is {\rm k}-complete for $k\in \NN\sqcup\{\lip,\infty,\const\}$ \deff \eqref{dspopodspds} holds for $[r,r']\equiv [0,1]$. 

For this, let $\phi,\chi\in \DIDE^k_{[r,r']}$ be given. Then, for $\varrho\colon [0,1]\rightarrow [r,r']$ as in \textrm{\ref{remkkk3}} with $[\ell,\ell']\equiv [0,1]$ there, we have $\dot\varrho\cdot (\phi\cp\varrho), \dot\varrho\cdot (\chi\cp\varrho) \in \DIDE^k_{[0,1]}$ with
\begin{align*}
	\textstyle\int \Ad_{[\innt_r^s \phi]^{-1}}(\chi(s))\:\dd s &\textstyle \myeq{}{=}  \int_r^{\varrho(1)} \Ad_{[\innt_r^s\hspace{-1pt} \phi]^{-1}}(\chi(s))\:\dd s\\
	&\textstyle \myeq{\eqref{substitRI}}{=} \int_0^1 \dot\varrho(s) \cdot \Ad_{[\innt_r^{\varrho(s)}\hspace{-1pt} \phi]^{-1}}(\chi(\varrho(s)))\:\dd s\\[3pt]
	&\textstyle \myeq{}{=} \int_0^1  \Ad_{[\innt_r^{\varrho(s)}\hspace{-1pt} \phi]^{-1}}((\dot\varrho \cdot (\chi\cp \varrho))(s))\:\dd s
	\\
	&\textstyle \myeq{\textrm{\ref{subst}}}{=} \int_0^1  \Ad_{[\innt_0^{s}\dot\varrho\he\cdot\he (\phi\he\cp\he\varrho)]^{-1}}((\dot\varrho \cdot (\chi\cp \varrho))(s))\:\dd s.
\end{align*}
\end{itemize}
\endgroup
\noindent
In particular, Point \ref{remkkk2} then shows:
\begingroup
\setlength{\leftmargini}{13pt}
\begin{itemize}
\item
If $G$ is $C^0$-semiregular, then $\mg$ is {\rm 0}-complete \deff $\mg$ is integral complete -- i.e., \deff $\int \phi(s)\:\dd s \in \mg$ holds for each $\phi\in C^0([0,1],\mg)$.
\item
If $G$ is $C^k$-semiregular for $k\in \NN_{\geq 1}\sqcup\{\lip,\infty\}$, then $\mg$ is {\rm k}-complete \deff $\mg$ is Mackey-complete.\footnote{Recall that $\mg$ is Mackey complete \deff $\int \phi(s)\:\dd s \in \mg$ holds for each $\phi\in C^k([0,1],\mg)$, for any $k\in \NN_{\geq 1}\sqcup\{\lip,\infty\}$, cf., 2.14 Theorem in \cite{COS}.}\hspace*{\fill}$\ddagger$
\end{itemize}
\endgroup
\end{remark}

\subsection{Weak Continuity} 
A pair $(\phi,\psi)\in C^0([r,r'],\mg)\times C^0([r,r'],\mg)$ is said to be 
\begingroup
\setlength{\leftmargini}{12pt}
\begin{itemize}
\item
	{\bf admissible} \defff $\phi + (-\delta,\delta)\cdot \psi \subseteq \DIDE_{[r,r']}$ holds for some $\delta>0$.
\item
	{\bf regular}\hspace{17.8pt} \defff it is admissible with
	\begin{align*}
	\textstyle\limih\innt_r^\bullet \phi + h\cdot \psi = \innt_r^\bullet \phi. 
\end{align*}
\vspace{-25pt}
\end{itemize}
\endgroup
\noindent
Then, 
\begin{remark}
\label{lkfdklfdlkfdlkfd}
\begingroup
\setlength{\leftmargini}{17pt}
{
\renewcommand{\theenumi}{\arabic{enumi})} 
\renewcommand{\labelenumi}{\theenumi}
\begin{enumerate}
\item[]
\item
\label{lkfdklfdlkfdlkfd1}
It follows from \textrm{\ref{pogfpogf}} that $(\phi,\chi)$ is admissible/regular \deff $(\phi|_{[\ell,\ell']},\chi|_{[\ell,\ell']})$ is admissible/regular for each $r\leq \ell<\ell'\leq r'$. 
\item
\label{lkfdklfdlkfdlkfd2}
	Each $(0,\expal(X))$ with $X\in \dom[\exp]$ is regular; because we have
\begin{align*}
	\textstyle\innt_0^t h\cdot \phi_X|_{[0,1]} \stackrel{\eqref{odaidaooipidadais}}{=} \innt t h\cdot \phi_X|_{[0,1]} \stackrel{\eqref{odaidaooipidadais}}{=} \innt_0^{t h} \phi_X|_{[0,1]}
\end{align*}
for each $t\in[0,1]$, and each $h\in \RR$.
\hspace*{\fill}$\ddagger$
\end{enumerate}}
\endgroup
\end{remark}
\noindent
We say that $G$ is {\bf weakly k-continuous} for $k\in \NN\sqcup\{\lip,\infty,\const\}$ \defff each admissible $(\phi,\psi)\in C^k([0,1],\mg)\times C^k([0,1],\mg)$ is regular. 
\begin{lemma}
\label{fdlkfdlkfd}
If $G$ is weakly {\rm k}-continuous for $k\in \NN\sqcup\{\lip,\infty,\const\}$, then each admissible $(\phi,\psi)\in C^k([r,r'],\mg)\times C^k([r,r'],\mg)$ (for each $[r,r']\in \COMP$) is regular. 
\end{lemma}
\begin{proof}
We define $\varrho\colon [0,1]\ni t\mapsto r+ t\cdot (r'-r) \in [r,r']$; and observe that
\begin{align*}
	\textstyle\innt_r^{\varrho}\: \phi &\textstyle\stackrel{\textrm{\ref{subst}}}{=}\innt_0^\bullet\: \dot\varrho\cdot (\phi\cp\varrho), \\
	\textstyle\innt_r^{\varrho}\: [\phi +h\cdot \psi] &\textstyle\stackrel{\textrm{\ref{subst}}}{=}\innt_0^\bullet\: [\dot\varrho\cdot (\phi\cp\varrho) + h\cdot \dot\varrho\cdot (\psi\cp\varrho)]
\end{align*}
holds for $h>0$ suitably small. Since we have $\varrho\cdot (\phi\cp\varrho),\: \dot\varrho\cdot (\psi\cp\varrho) \in C^k([0,1],\mg)$ by Point \ref{remkkk3},  
the claim is clear from the presumptions.   
\end{proof}
\begin{lemma}
\label{fdppofdpofdpofd}
	$G$ is weakly {\rm k}-continuous for $k\in \NN\sqcup\{\lip,\infty\}$ \deff 
\begin{align}
\label{dlklkdsalklkalksaklsa}
		\textstyle\limih\innt_0^\bullet h\cdot \chi= e
\end{align}
	holds, for each $\chi\in \DIDED_\kk$ with $(-\delta,\delta)\cdot \chi \subseteq \DIDED_\kk$ for some $\delta>0$. The same statement also holds for $k\equiveq \const$ if $G$ is abelian.
\end{lemma}
\begin{proof}
The one implication is evident. For the other implication, we suppose that $(\phi,\psi)\in C^k([0,1],\mg)\times C^k([0,1],\mg)$ is admissible. Since $\phi\in \DIDE^k_{[0,1]}$ holds, we have  
\begin{align*}
	\textstyle[\innt_0^t\phi]^{-1}[\innt_0^t\phi + h\cdot \psi]\stackrel{\textrm{\ref{kdskdsdkdslkds}}}{=}
	\innt_0^t h\cdot \Ad_{[\innt_0^\bullet \phi]^{-1}}(\psi)\qquad\quad\forall\: t\in [0,1]
\end{align*}
with $\chi:= \Ad_{[\innt_0^\bullet \phi]^{-1}}(\psi)\in C^k([0,1],\mg)$ by 
Point \ref{remkkk2}. The claim is thus clear from \eqref{dlklkdsalklkalksaklsa}. 
\end{proof}
\begin{corollary}
\label{ffddf}
If $G$ is abelian, then $G$ is weakly $\const$-continuous.
\end{corollary}
\begin{proof}
This is clear from Lemma \ref{fdppofdpofdpofd} and Remark \ref{lkfdklfdlkfdlkfd}.\ref{lkfdklfdlkfdlkfd2}. 
\end{proof}

\subsection{Mackey Continuity}
\label{lkkldsklsdkl}
We write $\{\phi_n\}_{n\in \NN}\mackarr{\kk} \phi$ for $k\in \NN\sqcup\{\lip,\infty,\const\}$, $\{\phi_n\}_{n\in \NN} \subseteq C^k([r,r'],\mg)$, and $\phi\in C^k([r,r'],\mg)$ \defff
\begin{align}
\label{podsopdspodssd}
	\ppp^\dind_\infty(\phi-\phi_n)\leq \mackeyconst^\dind_\pp\cdot \lambda_{n} \qquad\quad\forall\: n\geq \mackeyindex^\dind_\pp,\:\:\pp\in \SEM,\:\:\dind\llleq k
\end{align}
holds for certain $\{\mackeyconst^\dind_\pp\}_{\dind\llleq k,\:\pp\in \SEM}\subseteq \RR_{\geq 0}$, $\{\mackeyindex^\dind_\pp\}_{\dind\llleq k,\:\pp\in \SEM}\subseteq \NN$, and $\RR_{\geq 0}\supseteq \{\lambda_{n}\}_{n\in \NN}\rightarrow 0$.
\begin{remark}
\label{idsoidsodso}
Suppose that $\DIDE^k_{[r,r']}\supseteq \{\phi_n\}_{n\in \NN}\mackarr{\kk} \phi\in \DIDE^k_{[r,r']}$ holds for $k\in \NN\sqcup\{\lip,\infty,\const\}$. Then,
\begin{align*}
	\{\phi_{\iota(n)}\}_{n\in \NN}\mackarr{\kk} \phi 
\end{align*}
holds for each strictly increasing $\iota\colon \NN\rightarrow\NN$. 
\hspace*{\fill}$\ddagger$
\end{remark}
\noindent
We say that $G$ is {\bf Mackey k-continuous} \defff 
\begin{align}
\label{opdopdsopaaaa}
	\textstyle\DIDED_k\supseteq \{\phi_n\}_{n\in \NN}\mackarr{\kk} \phi\in \DIDED_k \qquad\quad\Longrightarrow\qquad\quad \limin\innt_r^\bullet \phi_n=\innt_r^\bullet \phi.
\end{align} 
In analogy to Lemma \ref{fdlkfdlkfd}, we obtain 
\begin{lemma}
\label{fdlkfdlkfd1}
$G$ is Mackey {\rm k}-continuous for $k\in \NN\sqcup\{\lip,\infty,\const\}$ \deff
\begin{align*}
	\textstyle\DIDE^k_{[r,r']}\supseteq \{\phi_n\}_{n\in \NN}\mackarr{\kk} \phi\in \DIDE^k_{[r,r']} \qquad\quad\Longrightarrow\qquad\quad \limin\innt_r^\bullet \phi_n=\innt_r^\bullet \phi,
\end{align*}
for each $[r,r']\in \COMP$.
\end{lemma}
\begin{proof}
The one implication is evident. For the other implication, we suppose that \eqref{opdopdsopaaaa} holds.  
Then, for $[r,r']\in \COMP$ fixed, we let $\varrho\colon [0,1]\ni t\mapsto r+ t\cdot (r'-r) \in [r,r']$; and obtain 
\begin{align*}
	\textstyle \DIDE^k_{[r,r']}\supseteq \{\phi_n\}_{n\in \NN}\mackarr{\kk} \phi\in \DIDE^k_{[r,r']}\qquad
	&\stackrel{\text{\ref{remkkk3}}}{\Longrightarrow} \qquad \DIDED_\kk\supseteq\{\dot\varrho\cdot (\phi_n\cp\varrho)\}_{n\in \NN}\mackarr{\kk} \dot\varrho\cdot (\phi\cp\varrho)\in \DIDED_\kk\\[2pt]
	&\textstyle\Longrightarrow \qquad \limin\innt_r^\bullet \dot\varrho\cdot (\phi_n\cp\varrho)=\innt_r^\bullet \dot\varrho\cdot (\phi\cp\varrho)\\
	&\textstyle\stackrel{\textrm{\ref{subst}}}{\Longrightarrow} \qquad \limin\innt_r^\bullet \phi_n=\innt_r^\bullet \phi,
\end{align*}
whereby the second step is due to the presumptions.
\end{proof}
\noindent
In analogy to Lemma \ref{fdppofdpofdpofd}, we obtain 
\begin{lemma}
\label{fdppofdpofdpofd2}
	$G$ is Mackey {\rm k}-continuous for $k\in \NN\sqcup\{\lip,\infty\}$ \deff   
\begin{align}
\label{opdopdsopvcvcvc}
	\textstyle\DIDED_\kk\supseteq \{\phi_n\}_{n\in \NN}\mackarr{\kk} 0 \qquad\quad\Longrightarrow\qquad\quad \limin\innt_0^\bullet\phi_n=e.
\end{align}
The statement also holds for $k\equiveq \const$ if $G$ is abelian.
\end{lemma}
\begin{proof}
The one implication is evident. For the other implication, we suppose that  $\DIDED_\kk\supseteq \{\phi_n\}_{n\in \NN}\mackarr{\kk} \phi\in \DIDED_\kk$ holds; and observe that
\begin{align*}
	\textstyle[\innt_0^t\phi]^{-1}[\innt_0^t\phi_n]\stackrel{\textrm{\ref{kdskdsdkdslkds}}}{=}
	\innt_0^t \underbrace{\Ad_{[\innt_0^\bullet \phi]^{-1}}(\phi_n-\phi)}_{\psi_n\in \DIDED_\kk}\qquad\quad\forall\: n\in \NN,\:\: t\in [0,1]
\end{align*}
holds by Point \ref{remkkk2}. 
Then, by Lemma \ref{opopsopsdopds} and Lemma \ref{xccxcxcxcxyxycllvovo}, we have $\DIDED_\kk\supseteq \{\psi_n\}_{n\in  \NN}\mackarr{\kk}0$; from which the claim is clear.   
\end{proof}
\noindent
We furthermore observe that
\begin{lemma}
\label{dsopdsopsdopsdop}
	 If $G$ is Mackey {\rm k}-continuous for $k\in \NN\sqcup\{\lip,\infty,\const\}$, then $G$ is weakly {\rm k}-continuous. 
\end{lemma}
\begin{proof}
	If $G$ is not weakly {\rm k}-continuous, then there exists an admissible $(\phi,\psi)\in C^k([r,r'],\mg)\times C^k([r,r'],\mg)$, an open neighbourhood $U\subseteq G$ of $e$, as well as sequences $\{\tau_n\}_{n\in \NN}\subseteq [r,r']$ and $\RR_{\neq 0}\supseteq\{h_n\}\rightarrow 0$, such that
\begin{align*}
	\textstyle[\innt_r^{\tau_n}\phi]^{-1}[\innt_r^{\tau_n} \phi+h_n\cdot \psi ]\notin   U\qquad\quad\forall\: n\in \NN
\end{align*}
holds. 
Then, $G$ cannot be Mackey {\rm k}-continuous, because we have $\{\phi+h_n\cdot \psi\}_{n\in \NN}\mackarr{\kk} \phi$. 
\end{proof}
\noindent
Remarkably, the uniform convergence on the right side of \eqref{opdopdsopvcvcvc} in Lemma \ref{fdppofdpofdpofd2} can be replaced by a weaker convergence property; namely,
\begin{lemma}
\label{fdfdfdfdfddffdasaasassasaasgfgfgf} 
$G$ is  
	Mackey {\rm k}-continuous for $k\in \NN\sqcup\{\lip,\infty\}$ \deff 
	\begin{align}
	\label{lkdslkdslkdslkdslklkdsdsppfddfdggg}
	\textstyle\DIDED_\kk\supseteq \{\phi_n\}_{n\in \NN}\mackarr{\kk} 0 \qquad\quad\Longrightarrow\qquad\quad \lim_n\innt \phi_n=e.
\end{align}
The statement also holds for $k\equiveq\const$ if $G$ is abelian.
\end{lemma}
\begin{proof}
The one implication is evident. For the other implication, we suppose that \eqref{lkdslkdslkdslkdslklkdsdsppfddfdggg} holds; and that $G$ is not Mackey k-continuous. By Lemma \ref{fdppofdpofdpofd2}, there exist $\DIDED_\kk\supseteq \{\phi_n\}_{n\in \NN}\mackarr{\kk} 0$, $U\subseteq G$ open with $e\in G$, a sequence $\{\tau_n\}_{n\in \NN}\subseteq [0,1]$, and $\iota\colon \NN\rightarrow \NN$ strictly increasing, such that  
\begin{align}
\label{khsasapoisoippoisasasaq}
	\textstyle 
	\innt_0^{\tau_n} \phi_{\iota_n}\notin  U\qquad\quad\forall\: n\in \NN
\end{align}
holds. For each $n\in \NN$,  
we define  
\begin{align*}
	\DIDED_\kk\ni \chi_n:= \dot\varrho_n\cdot (\phi_{\iota_n}\cp\varrho_n)\qquad\quad\text{with}\qquad\quad\varrho_n\colon [0,1]\ni t\mapsto t\cdot \tau_n  \in [0,\tau_n];
\end{align*}
and conclude from Remark \ref{idsoidsodso} and Point \ref{remkkk3} that $\{\chi_n\}_{n\in \NN}\mackarr{\kk} 0$ holds.  
Then, \eqref{lkdslkdslkdslkdslklkdsdsppfddfdggg} implies  
\begin{align*}
	\textstyle\lim_n \innt_0^{\tau_n} \phi_{\iota_n}\stackrel{\textrm{\ref{subst}}}{=}\lim_n \innt \chi_n=e,  
\end{align*} 
which contradicts \eqref{khsasapoisoippoisasasaq}.
\end{proof}

\subsection{Sequentially Continuity}
\label{ljdjkjdsds}
We write $\{\phi_n\}_{n\in \NN}\seqarr{\kk} \phi$ for $k\in \NN\sqcup\{\lip,\infty,\const\}$, $\{\phi_n\}_{n\in \NN} \subseteq C^k([r,r'],\mg)$, and $\phi\in C^k([r,r'],\mg)$ \defff
\begin{align*}
	\textstyle\lim_n\ppp^\dind_\infty(\phi-\phi_n)=0\qquad\quad\forall\:\pp\in \SEM,\:\: \dind\llleq k 
\end{align*}
holds. 
We say that $G$ is {\bf sequentially k-continuous} \defff 
\begin{align*}
	\textstyle\DIDED_\kk\supseteq \{\phi_n\}_{n\in \NN}\seqarr{\kk} \phi\in \DIDED_\kk \qquad\quad\Longrightarrow\qquad\quad \limin\innt_r^\bullet \phi_n=\innt_r^\bullet \phi.
\end{align*}
\begin{remark}
\label{dssdsdsdds}
\begingroup
\setlength{\leftmargini}{17pt}
{
\renewcommand{\theenumi}{\arabic{enumi})} 
\renewcommand{\labelenumi}{\theenumi}
\begin{enumerate}
\item[]
\item
\label{dssdsdsdds2}
	Suppose that $G$ is sequentially {\rm k}-continuous for $k\in \NN\sqcup \{\lip,\infty,\const\}$. Evidently, then $G$ is Mackey {\rm k}-continuous; thus, weakly {\rm k}-continuous by Lemma \ref{dsopdsopsdopsdop}.
\item
\label{dssdsdsdds3}
	$G$ is sequentially {\rm k}-continuous for $k\in \NN\sqcup \{\lip,\infty,\const\}$ \deff 
\begin{align*}
	\textstyle\DIDE^k_{[r,r']}\supseteq \{\phi_n\}_{n\in \NN}\seqarr{\kk} \phi\in \DIDE^k_{[r,r']} \qquad\quad\Longrightarrow\qquad\quad \limin\innt_r^\bullet \phi_n=\innt_r^\bullet \phi
\end{align*}	
holds for each $[r,r']\in \COMP$. 
This just follows as in Lemma \ref{fdlkfdlkfd1}. 
\item
\label{dssdsdsdds4}
	Let $k\in \NN\sqcup\{\lip,\infty,\const\}$, with $G$ abelian for $k\equiveq \const$. Then, the same arguments as in Lemma \ref{fdppofdpofdpofd2} show that $G$ is  
	sequentially {\rm k}-continuous \deff 
	\begin{align*}
	\textstyle\DIDED_\kk\supseteq \{\phi_n\}_{n\in \NN}\seqarr{\kk} 0 \qquad\quad\Longrightarrow\qquad\quad \lim_n\innt_0^\bullet \phi_n=e.
\end{align*}
\vspace{-17pt}
\item
\label{dssdsdsdds5}
Let $k\in \NN\sqcup\{\lip,\infty,\const\}$, with $G$ abelian for $k\equiveq \const$. Then, the same arguments as in Lemma \ref{fdfdfdfdfddffdasaasassasaasgfgfgf} show that $G$ is  
	sequentially {\rm k}-continuous \deff 	  
\begin{align*}
	\textstyle\DIDED_\kk\supseteq \{\phi_n\}_{n\in \NN}\seqarr{\kk} 0 \qquad\quad\Longrightarrow\qquad\quad \lim_n\innt\phi_n=e.
\end{align*}
\item
\label{dssdsdsdds1}
	If $G$ is $C^k$-continuous for $k\in \NN\sqcup\{\lip,\infty,\const\}$, then $G$ is sequentially {\rm k}-continuous -- This is clear 
\begingroup
\setlength{\leftmarginii}{12pt}
\begin{itemize}
\item
for $k\in \NN \sqcup \{\lip,\infty\}$\quad from Point \ref{dssdsdsdds5},
\item 
for $k\equiveq\const$\hspace{52.9pt}\quad from Lemma \ref{fddfxxxxfd}.
\end{itemize}
\endgroup
\item
\label{dssdsdsdds6}
Let $k\in \NN\sqcup\{\lip,\infty,\const\}$ be given; and suppose that the $C^k$-topology on $C^k([0,1],\mg)$ is first countable, and that $G$ is sequentially {\rm k}-continuous. Then, $G$ is $C^k$-continuous.

In fact, if $\evol_\kk$ is not $C^k$-continuous, then 
there exists $U\subseteq G$ open, such that $W:=\evol_\kk^{-1}(U)\subseteq \DIDED_\kk$ is not open, i.e., not a neighbourhood of some $\phi\in W$.  
Since $C^k([0,1],\mg)$ (thus, $\DIDED_\kk$) is first countable, there exists a sequence $\{\phi_n\}_{n\in \NN}\subseteq \DIDED_\kk \setm W$ with $\phi_n\seqarr{\kk} \phi$. We obtain $\evol_\kk(\phi_n)\in A:= G\setm U$ for each $n\in \NN$; thus, $\evol_\kk(\phi)=\lim_n\evol_\kk(\phi_n)\in A$, as $\evol_\kk$ is sequentially continuous, and since $A$ is closed. This contradicts that $\evol_\kk(\phi)\in U=G\setm A$ holds.     \hspace*{\fill}$\ddagger$
\end{enumerate}}
\endgroup
\noindent
\end{remark}

\subsection{Piecewise Integrable Curves}
We now finally discuss piecewise integrable curves. Specifically, we provide the basic facts and definitions\footnote{Confer Sect.\ 4.3 in \cite{RGM} for the statements mentioned but not proven here.}; and furthermore show that sequentially 0-continuity and Mackey 0-continuity carry over to the 
piecewise integrable category. This will be used in Sect.\  \ref{uoudsuoiiudsds} to generalize Theorem 1 in \cite{TRM}. 

\subsubsection{Basic Facts and Definitions}
For $k\in \NN\sqcup\{\lip,\infty,\const\}$ and $[r,r']\in \COMP$, we let $\DP^k([r,r'],\mg)$ denote the set of all $\psi\colon [r,r']\rightarrow \mg$, such that there exist $r=t_0<{\dots}<t_n=r'$ and 
\begin{align}
\label{opopooppo}
	\psi[p]\in \DIDE^k_{[t_p,t_{p+1}]}\qquad\text{with}\qquad\psi|_{(t_p,t_{p+1})}=\psi[p]|_{(t_p,t_{p+1})}\qquad\text{for}\qquad p=0,\dots,n-1.
\end{align}  
In this situation, we define $\innt_r^r\psi:=e$, as well as 
\begin{align}
\label{defpio}
	\textstyle\innt_r^t\psi&\textstyle:=\innt_{t_{p}}^t \psi[p] \cdot \innt_{t_{p-1}}^{t_p} \psi[p-1]\cdot {\dots} \cdot \innt_{t_0}^{t_1}\psi[0]\qquad\quad \forall\: t\in (t_{p}, t_{p+1}].
\end{align} 
A standard refinement argument in combination with \textrm{\ref{pogfpogf}} then shows that this is well defined; i.e., independent of any choices we have made. 
It is furthermore not hard to see that for $\phi,\psi\in \DP^k([r,r'],\mg)$, we have 
 $\Ad_{[\innt_r^\bullet\phi]^{-1}}(\psi -\phi)\in \DP^k([r,r'],\mg)$ with 
\begin{align}
\label{dfdssfdsfd}
	\textstyle\big[\innt_r^t \phi\big]^{-1}\big[\innt_r^t \psi\big]=\innt_r^t \Ad_{[\innt_r^\bullet\phi]^{-1}}(\psi -\phi)
	\qquad\quad\forall\: t\in [r,r'].
\end{align} 
We write 
\begingroup
\setlength{\leftmargini}{12pt}
\begin{itemize}
\item
$\{\phi_n\}_{n\in \NN}\seqarrr \phi$ for $\{\phi_n\}_{n\in \NN} \subseteq \DP^0([r,r'],\mg)$ and $\phi\in \DP^0([r,r'],\mg)$ \defff
\begin{align*}
	\textstyle\lim_n\ppp_\infty(\phi-\phi_n)=0\qquad\quad\forall\:\pp\in \SEM\hspace{38pt} 
\end{align*}	
holds.
\item
$\{\phi_n\}_{n\in \NN}\mackarrr \phi$ for $\{\phi_n\}_{n\in \NN} \subseteq \DP^0([r,r'],\mg)$ and $\phi\in \DP^0([r,r'],\mg)$ \defff
\begin{align*}
	\ppp_\infty(\phi-\phi_n)\leq \mackeyconst_\pp\cdot \lambda_{n} \qquad\quad\forall\: n\geq \mackeyindex_\pp,\:\: \pp\in \SEM
\end{align*}
holds for certain $\{\mackeyconst_\pp\}_{\pp\in \SEM}\subseteq \RR_{\geq 0}$, $\{\mackeyindex_\pp\}_{\pp\in \SEM}\subseteq \NN$, and $\RR_{\geq 0}\supseteq \{\lambda_{n}\}_{n\in \NN}\rightarrow 0$.
\end{itemize}
\endgroup
\subsubsection{A Continuity Statement}
We recall the construction made in Sect.\ 4.3 in \cite{RGM}. 
\begingroup
\setlength{\leftmargini}{16pt}
{
\renewcommand{\theenumi}{\roman{enumi})} 
\renewcommand{\labelenumi}{\theenumi}
\begin{enumerate}
\item
\label{anmanmsnmsaanms1}
	We fix (a bump function) $\rhon\colon [0,1]\rightarrow [0,2]$ smooth with 
\begin{align}
\label{odsposdposdpopospoadsasa}
	\textstyle\rhon|_{(0,1)}>0,\:\:\int_0^1 \rhon(s)\:\dd s=1\qquad\quad\:\:\text{as well as}\qquad\quad\:\: \rhon^{(k)}(0)=0=\rhon^{(k)}(1)\qquad\forall\: k\in \NN.
\end{align} 
Then, given $[r,r']\in \COMP$ and $r=t_0<{\dots}<t_n=r'$, we let
\begin{align*}
	\rho_p\colon [t_p,t_{p+1}]\rightarrow [0,2],\qquad t\mapsto \rhon((t-t_p)\slash(t_{p+1}-t_p))\qquad\qquad\forall\: p=0,\dots,n-1;
\end{align*}
and define $\rho\colon [r,r']\rightarrow [0,2]$ by 
\begin{align*}
	\rho|_{[t_p,t_{p+1}]}:=\rho_p\qquad\quad \forall\: p=0,\dots,n-1. 
\end{align*}
Then, $\rho$ is smooth, with $\rho^{(k)}(t_p)=0$ for each $k\in \NN$, $p=0,\dots,n$; and \eqref{substitRI} shows that
\begin{align*}
	\textstyle\varrho\colon [r,r']\rightarrow [r,r'],\qquad t\mapsto r+\int_r^t \rho(s)\:\dd s 
\end{align*}
holds, with $\varrho(t_p)=t_p$ for $p=0,\dots,n-1$.
\item
\label{anmanmsnmsaanms2}
	For $\psi\in \DP^0([r,r'],\mg)$ with $r=t_0<{\dots}<t_n=r'$ as well as $\psi[0],\dots,\psi[n-1]$ as in \eqref{opopooppo}, we let $\varrho\colon [r,r']\rightarrow [r,r']$ and  $\rho\equiv \dot\varrho\colon [r,r']\rightarrow [0,2]$ be as in \ref{anmanmsnmsaanms1}. 	 
Then, it is straightforward from the definitions that  
\begin{align*}
	\textstyle\chi:= \rho\cdot (\psi\cp\varrho)\in \DIDE^0_{[r,r']}\qquad\text{holds, with}\qquad \textstyle \innt_r^\varrho \psi=\innt_r^\bullet \chi\qquad\text{and}\qquad \ppp_\infty(\chi)\leq 2\cdot \ppp_\infty(\psi)
\end{align*}
for each $\pp\in \SEM$.\footnote{In the proof of Lemma 24 in \cite{RGM},  this statement was more generally verified for the case that $k\in \NN\sqcup\{\lip,\infty\}$ holds.}
\end{enumerate}}
\endgroup
\noindent
We obtain
\begin{lemma}
\label{dfooddfop}
\begingroup
\setlength{\leftmargini}{17pt}
{
\renewcommand{\theenumi}{\arabic{enumi})} 
\renewcommand{\labelenumi}{\theenumi}
\begin{enumerate}
\item[]
\item
\label{dskdslkdskl1}
If $G$ is sequentially {\rm 0}-continuous, then 
\begin{align*}
	\:\:\textstyle\DP^0([r,r'],\mg)\supseteq \{\phi_n\}_{n\in \NN}\seqarrr \phi\in\DP^0([r,r'],\mg)\: \:\:\qquad\Longrightarrow\qquad\:\: \limin\innt_r^\bullet \phi_n=\innt_r^\bullet \phi.\:\:
\end{align*}
\vspace{-19pt}
\item
\label{dskdslkdskl2}
If $G$ is Mackey {\rm 0}-continuous, then
\begin{align*}
	\textstyle\DP^0([r,r'],\mg)\supseteq \{\phi_n\}_{n\in \NN}\mackarrr \phi\in\DP^0([r,r'],\mg) \qquad\:\:\Longrightarrow\qquad\:\: \limin\innt_r^\bullet \phi_n=\innt_r^\bullet \phi.
\end{align*} 
\end{enumerate}}
\endgroup
\vspace{-2pt}
\noindent
Specifically, in both situations, for each $\pp\in \SEM$, there exist some $\pp\leq \qq\in \SEM$ and $n_\pp\in \NN$ with 
\begin{align*}
	\textstyle(\pp\cp\chart)([\innt_r^\bullet \phi]^{-1}[\innt_r^\bullet \phi_n])\leq \int \qqq(\phi_n(s)-\phi(s))\:\dd s\qquad\quad\forall\: n\geq n_\pp.
\end{align*} 
\end{lemma}
\begin{proof}
Let $\phi\in \DP^0([r,r'],\mg)$ and $\{\phi_n\}_{n\in \NN} \subseteq \DP^0([r,r'],\mg)$ be given. 
For $\pp\equiv \uu\in \SEM$ fixed, we choose $U\subseteq G$ and $\uu\leq \mm\in \SEM$ as in Lemma \ref{lkdslklkdslkdsldsds}. 
We furthermore let $\mm\leq \qq\equiv \ww\in \SEM$ be as in \ref{as1}, for $\compact\equiv \im[[\innt_r^\bullet \phi]^{-1}]$
and $\vv\equiv \mm$ there. Then, for each $n\in \NN$, we let $\varrho_n\equiv\varrho$, $\rho_n\equiv \rho$, and $\chi_n\equiv \chi$ be as in \ref{anmanmsnmsaanms2}, for
\begin{align*}
	\psi\equiv \psi_n:= \Ad_{[\innt_r^\bullet\phi]^{-1}}(\phi_n -\phi) \in \DP^0([r,r'],\mg)
\end{align*}
there. Then, 
\begingroup
\setlength{\leftmargini}{12pt}
{
\begin{itemize}
\item
we have
\begin{align}
\label{poaopaopasdds}
	\textstyle[\innt_r^{\varrho_n(t)} \phi]^{-1}[\innt_r^{\varrho_n(t)} \phi_n]\stackrel{\eqref{dfdssfdsfd}}{=}\innt_r^{\varrho_n(t)} \Ad_{[\innt_r^\bullet\phi]^{-1}}(\phi_n -\phi)\stackrel{\text{\ref{anmanmsnmsaanms2}}}{=}\innt_r^t \chi_n
	\qquad\quad\forall\: n\in \NN,\:\: t\in [r,r'].
\end{align}
\item
we have  
	$\mmm_\infty(\chi_n)\leq 2\cdot \mmm_\infty(\psi_n)\leq 2\cdot \qqq_\infty(\phi_n -\phi)$ for each $n\in \NN$ by Lemma \ref{opopsopsdopds},  
which shows that
\begingroup
\setlength{\leftmarginii}{12pt}
\begin{itemize}
\item
$\DIDE_{[r,r']}^0\supseteq \{\chi_n\}_{n\in \NN} \seqarr{0} 0$\:\hspace{6pt} holds if we are in the situation of \textit{\ref{dskdslkdskl1}},
\item
$\DIDE_{[r,r']}^0\supseteq \{\chi_n\}_{n\in \NN} \mackarr{0} 0$\: \hspace{0pt} 
holds if we are in the situation of \textit{\ref{dskdslkdskl2}}.
\end{itemize}
\endgroup
\noindent
In both situations, there thus exists some $n_\pp\in \NN$ with $\innt_r^\bullet \chi_n \in U$ for each $n\geq n_\pp$. 
\end{itemize}}
\endgroup
\noindent
We obtain from Lemma \ref{lkdslklkdslkdsldsds} (second step), and \ref{as1}  (last step) that\footnote{For the third step observe that $\rho_n\geq 0$ holds for each $n\in \NN$.} 
\begin{align*}
	\textstyle(\pp\cp\chart)([\innt_r^{\varrho_n(t)} \phi]^{-1}[\innt_r^{\varrho_n(t)} \phi_n])
	&\textstyle\stackrel{\eqref{poaopaopasdds}}{=}(\pp\cp\chart)(\innt_r^t\chi_n)\\[-4pt]
	&\textstyle\stackrel{\phantom{\eqref{substitRI}}}{\leq} \int_r^t (\mmm\cp \chi_n)(s) \:\dd s\\
	&\textstyle\stackrel{\eqref{substitRI}}{=} \int_r^{\varrho_n(t)} (\mmm\cp \psi_n)(s) \:\dd s\\[-4pt]
	&\textstyle\stackrel{\phantom{\eqref{substitRI}}}{\leq} \int_r^{\varrho_n(t)} \qqq(\phi_n(s)-\phi(s))\:\dd s
\end{align*} 
holds for all $n\geq n_\pp$ and $t\in[r,r']$; which proves the claim.
\end{proof}

\section{Mackey Continuity}
\label{lkfdlkfdlkfdljgljgd}
In this section, we show that
\begin{theorem}
\label{fdfffdfd}
If $G$ is $C^k$-semiregular for $k\in \NN\sqcup\{\lip,\infty\}$, then $G$ is Mackey {\rm k}-continuous. 
\end{theorem} 
\noindent 
The proof of Theorem \ref{fdfffdfd} is based on a bump function argument similar to that one used in the proof of Theorem 2 in \cite{RGM}. It furthermore  makes use of the fact that  $[0,1]\ni t \mapsto \innt_0^t \phi \in G$ is continuous for each $\phi\in \DIDED_\kk$. However, before we can provide the proof, we need some technical preparation first. 
\subsection{Some Estimates}
\label{ljdflkfdlkfdlkfd}
Suppose we are given $\varrho\colon [r,r']\rightarrow [r,r']$; and let $\rho\equiv\dot\varrho$ as well as  
\begin{align*}
	\textstyle C[\rho,\dind]:=\max\big(1,\max_{0\leq m,n \leq \dind}(\sup\{ |\rho^{(m)}(t)|^{n+1}\:| \: t\in [r,r']\})\big)\qquad\quad \forall\: \dind\in \NN.
\end{align*}
We observe the following: 
\begingroup
\setlength{\leftmargini}{12pt}
\begin{itemize}
\item
Let $\psi\in C^k([r,r'],\mg)$ for $k\in \NN\sqcup\{\infty\}$ and $\dind\llleq k$ be given.  
By \ref{chainrule}, \ref{productrule} in Appendix \ref{Diffcalc}, we have
\begin{align*}
	\textstyle(\rho\cdot (\psi\cp\varrho))^{(\dind)}&\textstyle=\sum_{q,m,n=0}^{\dind} h_\dind(q,m,n)\cdot \big(\he\rho^{(m)}\big)^{n+1}\cdot (\psi^{(q)}\cp\varrho),
\end{align*}
for a map $h_\dind\colon (0,\dots,\dind)^3\rightarrow \{0,1\}$ that is independent of $\varrho,\rho,\phi$.\footnote{More concretely, $h_\dind(q,m,n)$ are the coefficients appearing in the Leibniz rule for iterated derivatives of compositions.} We obtain
\begin{align}
\label{opsdposd}
	\ppp\big((\rho\cdot (\psi\cp\varrho))^{(\dind)}\big)\leq (\dind+1)^3\cdot C[\rho,\dind]\cdot \ppp_\infty^\dind(\psi) 
\end{align}
for each $\pp\in \SEM$, $0\leq \dind\llleq k$, and $\psi\in C^k([r,r'],\mg)$.
\item
Let $\psi\in C^\lip([r,r'],\mg)$ be given. Then we have 
\begin{align*}
	\textstyle\ppp((\rho\cdot (\psi\cp\varrho))(t)&-(\rho\cdot (\psi\cp\varrho))(t'))\\[3pt]
	&\leq|\rho(t)-\rho(t')|\cdot \ppp(\psi(\varrho(t)))
\:\: +\:\: |\rho(t')|\cdot \ppp(\psi(\varrho(t))- \psi(\varrho(t')))\\[3pt]
 &\leq |t-t'|\cdot C[\rho,1] \cdot \ppp_\infty(\psi)\quad\hspace{1.4pt}  +\:\:\he  C[\rho,0]\cdot  \Lip(\ppp,\psi)\cdot \underbrace{|\varrho(t)-\varrho(t')|}_{\leq \:C[\rho,0]\:\cdot\: |t-t'|}\\
 &\leq 2\cdot C[\rho,1]^2\cdot \ppp_\infty^\lip(\psi)\cdot |t-t'|
\end{align*}
for each $t,t'\in [r,r']$; thus,
\begin{align}
\label{ljfdljfdkfdjkff}
	\Lip(\ppp,\rho\cdot (\psi\cp\varrho))\leq 2\cdot C[\rho,1]^2\cdot\pp_\infty^\lip(\psi).	
\end{align}
\end{itemize}
\endgroup
\noindent
Let now $\rhon\colon [0,1]\rightarrow [0,2]$ be a fixed bump function  as in \eqref{odsposdposdpopospoadsasa}; as well as $\{t_n\}_{n\in \NN}\subseteq [0,1]$ strictly decreasing with $t_0=1$. 
For each $n\in \NN$, we define $0<\delta_n:=t_{n}-t_{n+1}<1$, as well as
\begin{align*}
	\textstyle\rho_n:=\delta_n^{-1}\cdot(\rhon\cp\kappa_n)\qquad\quad\text{for}\qquad\quad\kappa_n\colon [t_{n+1},t_{n}]\ni t\mapsto \delta_n^{-1}\cdot (t-t_{n+1})\in [0,1].
\end{align*}
We obtain from \eqref{substitRI} that 
\begin{align*}
	\textstyle\varrho_n\colon [t_{n+1},t_{n}]\ni t\mapsto  \int_{t_{n+1}}^t\rho_n(s)\:\dd s \in [0,1]\qquad\quad\forall\: n\in \NN
\end{align*}
holds; and furthermore observe that 
\begin{align}
\label{ljcvljvclkvccvc}
	C[\rho_n,\dind]\leq \delta_n^{-(\dind+1)^2}\cdot C[\rhon,\dind]\qquad\quad\forall\: n\in \NN.
\end{align}
\subsection{A Construction}
\label{dsdsdsdsdssd}
Suppose we are given 
$\{\phi_n\}_{n\in \NN}\subseteq \DIDE_{[0,1]}^k$ with $k\in \NN\sqcup\{\lip,\infty\}$; and let $\rhon,\rho_n,\varrho_n,\{\tau_n\}_{n\in \NN}, \{\delta_n\}_{n\in \NN}$ be as in Sect.\ \ref{ljdflkfdlkfdlkfd}. 
 Then,
\begingroup
\setlength{\leftmargini}{12pt}
\begin{itemize}
\item
We obtain from \eqref{opsdposd} and \eqref{ljcvljvclkvccvc} that    
\begin{align}
\label{dspoopsdopdsaasas}
\begin{split}
	\ppp\big((\rho_n\cdot(\phi_n\cp\varrho_n))^{(\dind)}\big)&\stackrel{\eqref{opsdposd}}{\leq} (\dind+1)^3\cdot C[\rho_n,\dind]\cdot \ppp_\infty^\dind(\phi_n)\\
	&\stackrel{\eqref{ljcvljvclkvccvc}}{\leq} (\dind+1)^3\cdot \delta_n^{-(\dind+1)^2}\cdot C[\rhon,\dind]\cdot \ppp_\infty^\dind(\phi_n)
\end{split}
\end{align}
holds, for each $\pp\in \SEM$, $\dind\llleq k$, and $n\in \NN$. 
\item
We obtain from \eqref{ljfdljfdkfdjkff} and \eqref{ljcvljvclkvccvc} that 
\begin{align}
\label{dspoopsdopdsaasassddsds}
\begin{split}
	\Lip(\ppp,\rho_n\cdot (\phi_n\cp\varrho_n))&\stackrel{\eqref{ljfdljfdkfdjkff}}{\leq} 2\cdot C[\rho_n,1]^2\cdot\ppp_\infty^\lip(\phi_n)
	\stackrel{\eqref{ljcvljvclkvccvc}}{\leq} 2\cdot \delta_n^{-8}\cdot C[\rhon,1]^2\cdot\ppp_\infty^\lip(\phi_n)
\end{split}
\end{align}
holds, for each $\pp\in \SEM$  
and $n\in \NN$. 
\end{itemize}
\endgroup 
\noindent
We define $\phi\colon [0,1]\rightarrow \mg$, by
\begin{align}
\label{dsodsiiodsoids}
	\phi(0):=0\qquad\quad\text{and}\qquad\quad\phi|_{[t_{n+1},t_{n}]}:=\rho_n\cdot (\phi_n\cp\varrho_n)\qquad\forall\: n\in \NN. 
\end{align}
Then, it is straightforward to see that\footnote{The technical details can be found, e.g., in the proof of Lemma 24 in \cite{RGM}.}    
$\phi|_{[t_ {n+1}, 1]}\in \DIDE_{[t_ {n+1}, 1]}^k$ holds for each $n\in \NN$, with  
\begin{align}
\label{sdsdsdsaasaad}
	\textstyle \innt_{0}^{\varrho_n} \phi_{n}\stackrel{\textrm{\ref{subst}}}{=} \innt_{t_{n+1}}^\bullet \phi|_{[t_{n+1},t_n]}\qquad\quad\forall\: n\in\NN.
\end{align}
Moreover, for $k\equiveq \lip$, we obtain from \eqref{dspoopsdopdsaasassddsds} that
\begin{align}
\label{asnbsanbsanbsanbyxxy}
	\Lip(\ppp,\phi|_{[t_ {n+1}, 1]})\leq 2\cdot C[\rhon,1]^2\cdot\max\big(\delta_0^{-8}\cdot \ppp_\infty^\lip(\phi_{0}),\dots,\delta_{n}^{-8}\cdot \ppp_\infty^\lip(\phi_{n})\big)
\end{align}
holds, cf.\ Appendix \ref{asassadsdsdsdsdsdsddssddss}.

\subsection{Proof of Theorem \ref{fdfffdfd}}
We are ready for the
\begin{proof}[Proof of Theorem \ref{fdfffdfd}]
Suppose that the claim is wrong, i.e., that $G$ is $C^k$-semiregular for $k\in \NN\sqcup\{\lip,\infty\}$ but not Mackey k-continuous. Then, by Lemma \ref{fdppofdpofdpofd2}, there exists a sequence $\DIDED_\kk\supseteq \{\phi_n\}_{n\in \NN}\mackarr{\kk} 0$ (with $\{\mackeyconst^\dind_\pp\}_{\dind\lleq k,\:\pp\in \SEM}\subseteq \RR_{\geq 0}$, $\{\mackeyindex^\dind_\pp\}_{\dind\llleq k,\:\pp\in \SEM}\subseteq \NN$, and $\RR_{\geq 0}\supseteq \{\lambda_{n}\}_{n\in \NN}\rightarrow 0$ as in \eqref{podsopdspodssd} for $\phi\equiv 0$ there), as well as $U\subseteq G$ open with $e\in U$, such that $\im[\innt_0^\bullet \phi_n]\not\subseteq U$ holds  for infinitely many $n\in \NN$. Passing to a subsequence if necessary, we thus can assume that
\begin{align}
\label{posdopsdopsd}
	\textstyle \im[\innt_0^\bullet \phi_n]\not\subseteq U\qquad\text{and}\qquad \lambda_n\leq 2^{-(n+1)^2} \qquad\quad\forall\: n\in \NN
\end{align}
holds. 
We let $t_0:=1$, and $t_n:= 1-\sum_{k=1}^n 2^{-k}$ for each $n\in \NN_{\geq 1}$; so that $\delta_n=t_{n}-t_{n+1}=2^{-(n+1)}$ holds for each $n\in \NN$. We  
construct $\phi\colon [0,1]\rightarrow \mg$ as in \eqref{dsodsiiodsoids} in Sect.\ \ref{dsdsdsdsdssd}; and fix $V\subseteq U$ open with $e\in V$ and $V\cdot V^{-1}\subseteq U$. 
\vspace{6pt}

\noindent
Suppose now that we have shown that $\phi$ is of class $C^k$; i.e., that $\phi \in \DIDED_\kk$ holds as $G$ is $C^k$-semiregular. Since $[0,1]\ni t\mapsto \innt_0^t\phi\in G$ is continuous, there exists some $\ell\geq 1$ with $\innt_0^t\phi\in V$ for each $t\in [0,t_\ell]$; thus,
\begin{align*}
	\textstyle\innt_{0}^{\varrho_n(t)} \phi_n\stackrel{\eqref{sdsdsdsaasaad}}{=}\innt_{t_{n+1}}^t \phi|_{[t_{n+1},t_{n}]}\stackrel{\textrm{\ref{pogfpogf}}}{=}[\innt_0^t\phi] \cdot [\innt_0^{t_{n+1}}\phi]^{-1} \in V\cdot V^{-1}\subseteq U 
\end{align*} 
for each $t\in [t_{n+1},t_n]$ with $n\geq \ell$, which contradicts \eqref{posdopsdopsd}. 
\vspace{6pt}

\noindent
To prove the claim, it thus suffices to show that $\phi$ is of class $C^k$:
\begingroup
\setlength{\leftmargini}{12pt}
\begin{itemize}
\item
Suppose first that $k\in \NN\sqcup\{\infty\}$ holds. Then, it suffices to show that  
	\begin{align*}
		\textstyle\lim_{(0,1]\ni h\rightarrow 0} 1/h \cdot \phi^{(\dind)}(h)=0\qquad\quad\forall\: \dind\llleq k
	\end{align*}
	holds, because $\phi$ is of class $C^k$ on $(0,1]$. 
	
	For this, we let $h\in [t_{n+1},t_n]$ for $n\in \NN$ be given; and observe that
	\begin{align*}
		\textstyle h\geq t_{n}=t_{n} - t_{n+1} +t_{n+1}=2^{-(n+1)} + 1-\sum_{k=1}^{n+1} 2^{-k}\geq 2^{-(n+1)}
	\end{align*}
	holds. Then, for $\pp\in \SEM$ fixed and $n\geq \mackeyindex^\dind_\pp$, we obtain from \eqref{dspoopsdopdsaasas} (with $\delta_n=2^{-(n+1)}$) as well as \eqref{posdopsdopsd} that 
\begin{align*} 
	1/h\cdot \ppp\big(\phi^{(\dind)}(h)\big)&\textstyle\stackrel{\phantom{\eqref{fdfdfdfdfdfdsf}}}{\leq} 2^{n+1}\cdot  \ppp\big((\rho_n\cdot(\phi_n\cp\varrho_n))^{(\dind)}\big)\\[1pt]
	&\textstyle\stackrel{\eqref{dspoopsdopdsaasas}}{\leq} (\dind+1)^3\cdot 2^{((\dind+1)^2+1)\cdot (n+1)}\cdot C[\rhon,\dind]\cdot \ppp_\infty^\dind(\phi_n) \\[1pt]
	&\stackrel{\eqref{posdopsdopsd}}{\leq} (\dind+1)^3\cdot C[\rhon,\dind]\cdot  \mackeyconst^\dind_\pp \cdot 2^{((\dind+1)^2+1)\cdot (n+1) -(n+1)^2 }\\
	&\stackrel{\phantom{\eqref{fdfdfdfdfdfdsf}}}{=} (\dind+1)^3\cdot C[\rhon,\dind]\cdot  \mackeyconst^\dind_\pp \cdot 2^{(((\dind+1)^2+1)-(n+1))\cdot (n+1)}
\end{align*}
holds;  
which clearly tends to zero for $n\rightarrow \infty$.
\item
Suppose now that $k\equiveq \lip$ holds. The previous point then shows $\phi\in C^0([0,1],\mg)$. For $\pp\in \SEM$ fixed, we thus have $\ppp_\infty(\phi)<\infty$. We let $\mackeyindex_\pp\equiv \mackeyindex^\lip_\pp$ for each $\pp\in \SEM$, define 
\begin{align*}
	D_\pp:=\max\big(2 \cdot\delta_0^{-8}\cdot \ppp_\infty^\lip(\phi_{0}),\dots,2 \cdot\delta_{\mackeyindex_\pp}^{-8}\cdot \ppp_\infty^\lip(\phi_{\mackeyindex_\pp})\big),
\end{align*}
and obtain for $n\geq \mackeyindex_\pp$ that 
\begin{align*}
	\Lip(\ppp, \phi|_{[t_{n+1},1]})&\stackrel{\eqref{asnbsanbsanbsanbyxxy}}{\leq}  C[\rhon,1]^2\cdot \max\big(2\cdot\delta_0^{-8}\cdot \ppp_\infty^\lip(\phi_{0}),\dots,2\cdot\delta_{n}^{-8}\cdot \ppp_\infty^\lip(\phi_{n})\big)\\[-4pt]
	&\stackrel{\eqref{podsopdspodssd}}{\leq}  C[\rhon,1]^2\cdot \max\big(D_\pp,\mackeyconst_\pp^\lip\cdot  \max\big(2^{1+8(\mackeyindex_\pp+2)}\cdot \lambda_{\mackeyindex_\pp+1},\dots,2^{1+8(n+1)}\cdot \lambda_n\big)\big)\\
	&\textstyle\stackrel{\eqref{posdopsdopsd}}{\leq}  C[\rhon,1]^2\cdot \max\big(D_\pp,\mackeyconst_\pp^\lip\cdot \max\big\{ 2^{1+8(\ell+1)}\cdot 2^{-(\ell+1)^2}\:\big|\: \mackeyindex_\pp+1\leq \ell\leq n\big\}\big)\\
	&\textstyle\stackrel{\phantom{\eqref{fdfdfdfdfdfdsf}}}{=} C[\rhon,1]^2\cdot\max\big(D_\pp, \mackeyconst_\pp^\lip\cdot\max\big\{2^{-\ell^2+ 6\ell+8}\:\big|\:  \mackeyindex_\pp+1\leq \ell\leq n\big\}\big)\\
	&\textstyle\stackrel{\phantom{\eqref{fdfdfdfdfdfdsf}}}{=} C[\rhon,1]^2\cdot \max\big(D_\pp,\mackeyconst_\pp^\lip\cdot\max\big\{2^{17-(\ell-3)^2}\:\big|\: \mackeyindex_\pp+1\leq \ell\leq n\big\}\big)\\[-3pt]
	&\textstyle\stackrel{\phantom{\eqref{fdfdfdfdfdfdsf}}}{\leq}   C[\rhon,1]^2\cdot \max\big(D_\pp,\mackeyconst_\pp^\lip\cdot 2^{17}\big)=: \mathfrak{L}
\end{align*}
holds; thus,
\vspace{-6pt}
\begin{align*}
	\ppp(\phi(t)-\phi(t'))\leq \mathfrak{L}\cdot |t-t'|\qquad\quad\forall\: t,t'\in (0,1].
\end{align*}
Moreover, since $\phi$ is continuous with $\phi(0)=0$, for each $t\in [0,1]$ we have
\begin{align*}
	\ppp(\phi(0)-\phi(t))&\textstyle=\lim_{\ell\rightarrow \infty} \ppp(\phi(0)-\phi(1/\ell)+ \phi(1/\ell)-\phi(t))\\
	&\textstyle\leq \lim_{\ell\rightarrow \infty} \ppp(\phi(1/\ell))+ \lim_{\ell\rightarrow \infty}\ppp(\phi(1/\ell)-\phi(t))\\
	&\textstyle=  \lim_{\ell\rightarrow \infty}\ppp(\phi(1/\ell)-\phi(t))\\
		&\textstyle\leq   \lim_{\ell\rightarrow \infty} \mathfrak{L}\cdot |1/\ell - t|\\
		&=\mathfrak{L}\cdot |0-t|.
\end{align*}
This shows $\Lip(\ppp,\phi)\leq \mathfrak{L}$, i.e., $\phi\in C^\lip([0,1],\mg)$.\qedhere
\end{itemize}
\endgroup
\end{proof}

\section{The Strong Trotter Property}
\label{uoudsuoiiudsds}
In this Section, we want to give a brief application of the notions introduced so far.   
For this, we recall that a Lie group $G$ is said to have the strong Trotter property 
\defff \eqref{sddssddszz} holds;  
and now will show 
\begin{proposition}
\label{opsdopsdopdsopsd}
\begingroup
\setlength{\leftmargini}{17pt}
\begin{enumerate}
\item
\label{opsdopsdopdsopsd1}
	If $G$ is sequentially {\rm 0}-continuous, then $G$ has the strong Trotter property.
\item
\label{opsdopsdopdsopsd2}
If $G$ is Mackey $0$-continuous, then 
	\eqref{sddssddszz} holds for each $\mu\in C_*^{1}([0,1],G)$ with $\dot\mu(0)\in \dom[\exp]$ and $\Der(\mu)\in C^\lip([0,1],\mg)$. 
\end{enumerate}
\endgroup
\end{proposition}
\noindent
Here,
\begingroup
\setlength{\leftmargini}{12pt}
{
\renewcommand{\theenumi}{\arabic{enumi})} 
\renewcommand{\labelenumi}{\theenumi}
\begin{itemize}
\item
By Remark \ref{dssdsdsdds}.\ref{dssdsdsdds1}, 
Proposition \ref{opsdopsdopdsopsd}.\ref{opsdopsdopdsopsd1} 
generalizes Theorem 1 in \cite{TRM}, stating that $G$ admits the strong Trotter property if $G$ is locally $\mu$-convex (recall that, by Theorem 1 in \cite{RGM}, local $\mu$-convexity is  equivalent to that $G$ is $C^0$-continuous).
\item
By Theorem \ref{fdfffdfd}, the presumptions made in Proposition \ref{opsdopsdopdsopsd}.\ref{opsdopsdopdsopsd2} are always fulfilled, e.g., if $G$ is $C^0$-semiregular, and $\mu$ is of class $C^2$ with $\dot\mu(0)\in \dom[\exp]$.
\end{itemize}}
\endgroup
\noindent
We will need the following observations:  
Let $\ell>0$, $\mu\in C^1_*([0,1],\U)$ be given; and define $\phi\equivdef \Der(\mu)$, $X:=\dot\mu(0)=\phi(0)$, as well as  
\begin{align*}
	\mu_{\tau}\colon [0,1/\ell]\ni t\mapsto \mu(\tau\cdot t)\in G\qquad\quad\forall\: \tau\in [0,\ell].
\end{align*}
Then, for each $t\in [0,1/\ell]$, we have 
\begin{align*}
	\Der(\mu_{\tau})(t)-\tau\cdot X&\stackrel{\eqref{kldlkdldsl}}{=}\dermapdiff((\chart\cp\mu_\tau)(t),\partial_t(\chart\cp\mu_\tau)(t))-\tau\cdot X\\[-4pt]
	&\stackrel{\phantom{\eqref{kldlkdldsl}}}{=}\tau\cdot \dermapdiff((\chart\cp\mu)(\tau\cdot t),\partial_t(\chart\cp\mu)(\tau\cdot t))-\tau\cdot X\\[-4pt]
	&\stackrel{\phantom{\eqref{kldlkdldsl}}}{=}\tau\cdot \Der(\mu)(\tau\cdot t)-\tau\cdot X\\[-4pt]
	&\stackrel{\phantom{\eqref{kldlkdldsl}}}{=}\tau\cdot (\phi(\tau\cdot t)-X).
\end{align*} 
For each $\pp\in \SEM$, $\tau\in[0,\ell]$, and $s\leq 1/\ell$, we thus  obtain
\begin{align}
\label{opopasopsaopas1}
		\textstyle\ppp_\infty(\Der(\mu_\tau)|_{[0,s]}-\tau\cdot X)=\ell\cdot \sup\{\ppp(\phi(\tau\cdot t)-\phi(0))\: |\: t\in [0,s] \};
	\end{align} 
whereby, for the case that $\phi\in C^\lip([0,1],\mg)$ holds, we additionally have
\begin{align}
\label{opopasopsaopas2}
	\textstyle \sup\{\ppp(\phi(\tau\cdot t)-\phi(0))\: |\: t\in [0,s] \}
\leq  s\cdot \ell\cdot \Lip(\ppp,\phi). 
\end{align}
We are ready for the 
\begin{proof}[Proof of Proposition \ref{opsdopsdopdsopsd}]
Let $\phi\equivdef \Der(\mu)$ and $X:=\dot\mu(0)=\phi(0)$, for
\begingroup
\setlength{\leftmargini}{16pt}
\begin{enumerate}
\item
	$\mu\in C_*^1([0,1],G)$ with $\dot\mu(0)\in \dom[\exp]$ if $G$ is sequentially {\rm 0}-continuous.
\item
$\mu\in C_*^1([0,1],G)$ with $\dot\mu(0)\in \dom[\exp]$ and $\phi\in C^\lip([0,1],\mg)$    
if $G$ is Mackey $0$-continuous. 
\end{enumerate}
\endgroup
\noindent
We suppose that \eqref{sddssddszz} is wrong; i.e., that     
 there exists some $\ell>0$, an open neighbourhood $U\subseteq G$ of $e$, a sequence $\{\tau_n\}_{n\in \NN}\subseteq [0,\ell]$, and a strictly increasing sequence $\{\iota_n\}_{n\in \NN}\subseteq \NN_{\geq 1}\cap [\ell,\infty)$ with
\begin{align}
\label{pospopods}
	\exp(-\tau_n\cdot X)\cdot \mu(\tau_n/\iota_n)^{\iota_n}\notin U\qquad\quad\forall\: n\in \NN. 
\end{align}
Passing to a subsequence if necessary, we can additionally assume that $\lim_n\tau_n=\tau\in [0,\ell]$ exists. We choose $V\subseteq G$ open with $e\in V$ and $V\cdot V\subseteq U$, and fix some $n_V\in \NN$ with
\begin{align}
\label{podsoppodspodsds}
	\textstyle
	\exp((\tau-\tau_n)\cdot X)\in V\qquad\quad\forall\: n\geq n_V.
\end{align}
Moreover, for each $n\in \NN$: 
 \begingroup
\setlength{\leftmargini}{12pt}
\begin{itemize}
\item 
We define 
\begin{align*}
	\chi_n:= \Der(\mu_{\tau_n})|_{[0,1/\iota_n]}\in \DIDE_{[0,1/\iota_n]}^k\qquad\quad\text{with}\qquad\quad \mu_{\tau_n}\colon [0,1/\ell]\ni t\mapsto \mu(t\cdot \tau_n).
\end{align*}
\item
We define $t_{n,m}:=m/\iota_n$ for $m=0,\dots,\iota_{n}$; as well as 
	$\phi_n[m]\colon [t_{n,m},t_{n,m+1}]\ni t\mapsto \chi_n(\cdot - t_{n,m})\in \mg$  
for $m=0,\dots,\iota_{n-1}$. Then, we have\footnote{In the first step below, \textrm{\ref{subst}} is applied with $\varrho\colon [t_{n,m},t_{n,m+1}] \mapsto t- t_{n,m}\in [0,1/\iota_n]$.}
\begin{align}
\label{dsljfdsoidsoidsoipds}
	\textstyle\innt\phi_n[m]\stackrel{\textrm{\ref{subst}}}{=}\innt\chi_n=\mu_{\tau_n}(1/\iota_n)=\mu(\tau_n/\iota_n)\qquad\quad\forall\: m=0,\dots,\iota_{n-1}.
\end{align}
\item
We define $\phi_n\in\DP^0([0,1],\mg)$ by
	\begin{align*}
	\phi_{n}|_{[t_{n,m},t_{n,m+1})}\hspace{2pt}&:=\phi_{n}[m]|_{[t_{n,m},t_{n,m+1})}\qquad\quad\forall\: m=0,\dots,\iota(n)-2,\\[2pt]
	\phi_{n}|_{[t_{n,\iota_{n-1}},t_{n,\iota_n}]}&:=\phi_{n}[\iota_n-1]; 
\end{align*}
and obtain
	\begin{align}
\label{fdlklkfdslkfds}
	\textstyle\innt \phi_n \stackrel{\eqref{defpio}}{=}\innt^{t_{n,\iota_n}}_{t_{n,\iota_n-1}}\phi_n[\iota_n-1]\cdot {\dots}\cdot \innt^{t_{n,1}}_{t_{n,0}}\phi_n[0] \stackrel{\eqref{dsljfdsoidsoidsoipds}}{=}\mu(\tau_n/\iota_n)^{\iota_n}\qquad\quad\forall\: n\in \NN.
\end{align}
\end{itemize}
\endgroup
\noindent
Then, for each $n\in \NN$ and $\pp\in \SEM$, we have 
\begin{align}
\label{fpogfpopogfpogfaasacc}
\begin{split}
	\ppp_\infty(\phi_n- \tau\cdot \phi_X)&\stackrel{\phantom{\eqref{opopasopsaopas1}}}{\leq} \ppp_\infty(\phi_n- \tau_n\cdot X)+\ppp(\tau_n\cdot X- \tau\cdot X)\\
	&\textstyle\stackrel{\phantom{\eqref{opopasopsaopas1}}}{=}\ppp_\infty\big(\Der(\mu_{\tau_n})|_{[0,1/{\iota_n}]}-\tau_n\cdot X\big)+ |\tau-\tau_n|\cdot \ppp(X)\\
	&\textstyle\stackrel{\eqref{opopasopsaopas1}}{\leq} \ell\cdot \sup\{\ppp(\phi(\tau_n\cdot t)-\phi(0))\: |\: t\in [0,1/\iota_n] \}+ |\tau-\tau_n|\cdot \ppp(X).
\end{split}
\end{align} 
For the case that $\phi\in C^\lip([0,1],\mg)$ holds, we furthermore obtain
\begin{align*}
	\textstyle\ppp_\infty(\phi_n- \tau\cdot \phi_X)&\stackrel{\eqref{fpogfpopogfpogfaasacc}, \eqref{opopasopsaopas2}}{\leq} \ell^2\slash\iota_n \cdot \Lip(\ppp,\phi) + |\tau-\tau_n|\cdot \ppp(X)\\
	&\quad\he\leq\quad\he  \underbrace{(\Lip(\ppp,\phi) + \ppp(X))}_{\mackeyconst_\pp}\cdot \underbrace{\textstyle(\ell^2\slash\iota_n+|\tau-\tau_n|)}_{\lambda_n}
\end{align*} 
for each $n\in \NN$ and $\pp\in \SEM$; 
whereby $\lim_n \lambda_n=0$ holds. We thus have
\begin{align*}
	\DP^0([0,1],\mg)\supseteq &\{\phi_n\}_{n\in \NN}\seqarrr\hspace{2.55pt} \tau\cdot\phi_X\\
	\text{as well as}\qquad\quad\DP^0([0,1],\mg)\supseteq &\{\phi_n\}_{n\in \NN}\mackarrr \tau\cdot\phi_X\qquad\text{if}\qquad \phi\in C^\lip([0,1],\mg)\qquad\text{holds}.
\end{align*} 
In both cases, by Lemma \ref{dfooddfop},  
there exists some $\NN\ni n'_V\geq n_V$ with 
\begin{align}
\label{idsiodsidsoidsoids}
	\textstyle[\innt \tau \cdot \phi_X|_{[0,1]}]^{-1}\cdot[\innt \phi_n]\in V\qquad\quad\forall\: n\geq n'_V;
\end{align}
and we obtain for $n\geq n'_V$ that
\begin{align*}
	\exp(-\tau_n\cdot X)\cdot \mu(\tau_n/\iota_n)^{\iota_n}&=\overbrace{\exp\big((\tau-\tau_n)\cdot X\big)}^{\stackrel{\eqref{podsoppodspodsds}}{\in} V}\:\cdot\:\underbrace{\big(\overbrace{\exp(-\tau\cdot X)}^{\equiveq\: [\innt\tau\he\cdot\he \phi_X|_{[0,1]}]^{-1}}\:\cdot\: \overbrace{\mu(\tau_n/\iota_n)^{\iota_n}}^{\stackrel{\eqref{fdlklkfdslkfds}}{=} \innt \phi_n}\big)}_{\stackrel{\eqref{idsiodsidsoidsoids}}{\in} V}
\in V\cdot V\subseteq  U
\end{align*}  
holds,   
which contradicts \eqref{pospopods}.
\end{proof}

\section{Differentiation}
\label{opdfodf}
In this section, we discuss several differentiability properties of the evolution map. The whole discussion is based on the following generalization of Proposition 7 in \cite{RGM}.  
\begin{proposition}
\label{rererererr}
Let $\{\varepsilon_n\}_{n\in \NN}\subseteq C^0([r,r'],\mg)$,   
$\chi\in C^0([r,r'],\mg)$, and $\RR_{\neq 0}\supseteq \{h_n\}_{n\in \NN}\rightarrow 0$ be given with  
$\{h_n\cdot \chi + h_n\cdot \varepsilon_n\}_{n\in \NN}\subseteq \DIDE_{[r,r']}$, such that
\begingroup
\setlength{\leftmargini}{19pt}{
\renewcommand{\theenumi}{{\roman{enumi}})} 
\renewcommand{\labelenumi}{\theenumi}
\begin{enumerate}
\item
\label{saaaaaw1}
	$\lim_{n}\varepsilon_n(t)=0$\hspace{74.7pt}holds for each $t\in [r,r']$,
\item
\label{saaaaaw2}
	$\sup\{\pp_\infty(\varepsilon_n)\:|\: n\in \NN)\}<\infty$\:\:\:\: holds for each $\pp\in \SEM$.
\end{enumerate}}
\endgroup
\noindent
Then, the following two conditions are equivalent:
\begingroup
\setlength{\leftmargini}{17pt}
{
\renewcommand{\theenumi}{{\alph{enumi})}} 
\renewcommand{\labelenumi}{\theenumi}
\begin{enumerate}
\item
\label{dpopoffdpolkfdlkfdlklkfd1}
$\limin \chart(\innt_r^\bullet h_n\cdot \chi + h_n\cdot \varepsilon_n)=0$.
\item
\label{dpopoffdpolkfdlkfdlklkfd2}
\:
$\limin 1/h_n\cdot \chart\big(\innt_r^\bullet h_n\cdot \chi+h_n\cdot \varepsilon_n\big)=\int_r^\bullet (\dd_e\chart\cp \chi)(s)\:\dd s\in \comp{E}$.
\end{enumerate}}
\endgroup
\end{proposition}
\noindent
\noindent
The proof of Proposition \ref{rererererr} will be established in Sect.\ \ref{dsdsdsdsds}. We now first use this proposition, to discuss the differential of the evolution maps as well as the differentiation of parameter-dependent integrals.

\subsection{Some Technical Statements}
We will need the following variation of Proposition \ref{rererererr}:
\begin{corollary}
\label{pofdpofdop}
Let $\delta>0$, $\{\varepsilon_h\}_{h \in \MD_\delta}\subseteq C^0([r,r'],\mg)$, and    
$\chi\in C^0([r,r'],\mg)$ be given with $\{h\cdot \chi + h\cdot \varepsilon_h\}_{h\in \MD_\delta}\subseteq \DIDE_{[r,r']}$, such that 
\begingroup
\setlength{\leftmargini}{19pt}{
\renewcommand{\theenumi}{{\rm\roman{enumi})}} 
\renewcommand{\labelenumi}{\theenumi}
\begin{enumerate}
\item
\label{csaaaaaw1}
	$\lim_{h\rightarrow 0}\varepsilon_h(t)=0$\hspace{71.8pt} holds for each $t\in [r,r']$,
\item
\label{csaaaaaw2}
	$\sup\{\pp_\infty(\varepsilon_h)\:|\: h\in \MD_{\delta_\pp}
	\}<\infty$\qquad holds for some $0<\delta_\pp\leq \delta$, for each $\pp\in \SEM$.
\end{enumerate}}
\endgroup
\noindent 
Then, the following  conditions are equivalent:\footnote{Recall Remark \ref{fdfdfdfdfdfdnbcxnbcxnbcbnx} for the notation used in \ref{podspods4}.}
\begingroup
\setlength{\leftmargini}{19pt}
{
\renewcommand{\theenumi}{{\rm\alph{enumi})}} 
\renewcommand{\labelenumi}{\theenumi}
\begin{enumerate}
\item
\label{podspods1}
$\limih\he \chart(\innt_r^\bullet h\cdot \chi + h\cdot \varepsilon_h)=0$.
\vspace{2pt}
\item
\label{podspods2}
$\limin \chart(\innt_r^\bullet h_n\cdot \chi + h_n\cdot \varepsilon_{h_n})=0$ for each sequence $\MD_\delta\supseteq \{h_n\}_{n\in \NN}\rightarrow 0$.
\item
\label{podspods3}
$\limin \he 1/h_n\cdot \chart\big(\innt_r^\bullet h_n\cdot \chi+h_n\cdot \varepsilon_{h_n}\big)=\int_r^\bullet (\dd_e\chart\cp \chi)(s)\:\dd s\in \comp{E}$ for each sequence $\MD_\delta\supseteq \{h_n\}_{n\in \NN}\rightarrow 0$.
\item
\label{podspods4}
$\frac{\dd}{\dd h}\big|^\infty_{h=0}\he \chart\big(\innt_r^\bullet h\cdot \chi+h\cdot \varepsilon_h\big)=\int_r^\bullet (\dd_e\chart\cp \chi)(s)\:\dd s\in \comp{E}$.
\end{enumerate}}
\endgroup 
\end{corollary}
\begin{proof}
By Lemma \ref{aslksalksalksakasklaslksalklkalksa} (applied to $(G,+)\equiv (\he\ovl{E},+\he)$ there), \ref{podspods1} is equivalent to \ref{podspods2}.  
Moreover, by Proposition \ref{rererererr}, \ref{podspods2} is equivalent to \ref{podspods3}, because 
\begingroup
\setlength{\leftmargini}{15pt}
\begin{itemize}
\item[$-$]
Condition \ref{csaaaaaw1}\hspace{3pt} implies Condition {\it\ref{saaaaaw1}}\hspace{2.5pt} in Proposition \ref{rererererr}, for $\varepsilon_n\equiv \varepsilon_{h_n}$ there, 
\item[$-$]
Condition \ref{csaaaaaw2} implies Condition {\it\ref{saaaaaw2}} in Proposition \ref{rererererr}, for $\varepsilon_n\equiv \varepsilon_{h_n}$ there.
\end{itemize}
\endgroup
\noindent
Finally, by Lemma \ref{aslksalksalksakasklaslksalklkalksa} (applied to $(G,+)\equiv (\he\ovl{E},+\he)$ there), \ref{podspods3} is equivalent to \ref{podspods4}.  
\end{proof}
\noindent
Given a net $\{\psi_\alpha\}_{\alpha\in I}\subseteq  C^0([r,r'],\mg)$, and some $\psi \in C^0([r,r'],\mg)$, we write  $\{\psi_\alpha\}_{\alpha\in I}\netarr{0}\psi$ \defff
\begin{align*}
	\textstyle\lim_\alpha\ppp(\psi-\psi_\alpha)=0\qquad\quad\forall\:\pp\in \SEM.
\end{align*}
\begin{lemma}
\label{podspodspodsodsp}
Suppose that $G$ is Mackey {\rm k}-continuous for $k\in \NN\sqcup\{\lip,\infty,\const\}$. 
Suppose furthermore that we are given $\DIDE^k_{[r,r']}\supseteq \{\phi_n\}_{n\in \NN}\mackarr{\kk} \phi \in \DIDE^k_{[r,r']}$ as well as $\DIDE^k_{[r,r']}\supseteq \{\psi_\alpha\}_{\alpha\in I}\netarr{0} \psi \in \DIDE^k_{[r,r']}$ for $[r,r']\in \COMP$, such that the expressions
\begin{align*}
\xi(\phi,\psi)&\textstyle:=\dd_e\LT_{\innt \phi}\big(\int \Ad_{[\innt_r^s \phi]^{-1}}(\psi(s))\:\dd s \big)\\
\xi(\phi_n,\psi_\alpha)&\textstyle:=\dd_e\LT_{\innt \phi_n}\big(\int \Ad_{[\innt_r^s \phi_n]^{-1}}(\psi_\alpha(s))\:\dd s \big)\qquad\quad\forall\: n\in \NN
\end{align*} 
are well defined; i.e., such that the occurring Riemann integrals exist in $\mg$. Then, we have
\begin{align*}
	\textstyle \lim_{(n,\alpha)}\xi(\phi_n,\psi_\alpha)=\xi(\phi,\psi).
\end{align*} 
\end{lemma}
\begin{proof}
This follows by the same arguments as in Corollary 13 and Lemma 41 in \cite{RGM}. For completeness, the adapted argumentation is provided in Appendix \ref{asasajkakjsajkjaskjakjs}.
\end{proof}

\subsection{The Differential of the Evolution Map}
\label{cxpopocxpocxcx}
We now discuss the differential of the evolution map -- for which 
we recall the conventions fixed in Remark \ref{lkcxlkcxlkxlkcxpoidsoids}.  
Then, Corollary \ref{pofdpofdop} (with $\varepsilon_h\equiv 0$ there)  provides us with  
\begin{proposition} 
\label{saasassansamnmsa}
Suppose that $(\phi,\psi)$ is admissible, with $\dom[\phi], \dom[\psi]=[r,r']$. 
\begingroup
\setlength{\leftmargini}{16pt}{
\renewcommand{\theenumi}{{\arabic{enumi}})} 
\renewcommand{\labelenumi}{\theenumi}
\begin{enumerate}
\item
\label{saasassansamnmsa1}
The pair $(\phi,\psi)$ is regular \deff we have
\begin{align*}
	\textstyle\frac{\dd}{\dd h}\big|^\infty_{h=0}\:\chart\big([\innt_r^\bullet\phi]^{-1}[\innt_r^\bullet\phi+h\cdot \psi]\big)= \int_r^\bullet (\dd_e\chart\cp\Ad_{[\innt_r^s \phi]^{-1}})(\psi(s))\:\dd s\in \ovl{E}.
\end{align*} 
\item
\label{saasassansamnmsa2}
If $(\phi,\psi)$ is regular, then $(-\delta,\delta)\ni h\mapsto \innt \phi+h\cdot \psi\in G$ is differentiable at $h=0$ (for $\delta>0$ suitably small) \deff 
$\int \Ad_{[\innt_r^s \phi]^{-1}}(\psi(s))\:\dd s\in \mg$ holds. In this case, we have
\begin{align}
\label{fpidpfpofdfdfd}
	\textstyle\frac{\dd}{\dd h}\big|_{h=0} \he\innt\phi+h\cdot \psi= \dd_e\LT_{\innt \phi}\big(\int \Ad_{[\innt_r^s \phi]^{-1}}(\psi(s))\:\dd s\he\big).
\end{align}
\end{enumerate}}
\endgroup
\end{proposition}
\begin{proof}
\begingroup
\setlength{\leftmargini}{16pt}{
\renewcommand{\theenumi}{{\arabic{enumi}})} 
\renewcommand{\labelenumi}{\theenumi}
\begin{enumerate}
\item
For $|h|<\delta$, with $\delta>0$ suitably small, we have
\vspace{-5pt}
\begin{align}
\label{nmdsnmdsnmsd}
	\textstyle \chart\big([\innt_r^t \phi]^{-1}[\innt_r^t\phi+h\cdot \psi]\big)\stackrel{\textrm{\ref{kdskdsdkdslkds}}}{=}\chart\big(\innt_r^t h\cdot \overbrace{\Ad_{[\innt_r^\bullet \phi]^{-1}}(\psi)}^{=:\:\chi}\big)\qquad\quad\forall\: t\in [r,r'].
\end{align}
We obtain from the Equivalence of \ref{podspods1} and \ref{podspods4} in Corollary \ref{pofdpofdop} for $\varepsilon_h\equiv 0$ there that (third step) 
\begin{align*}
	&\textstyle \limih\innt_r^\bullet \phi+h\cdot \psi=\innt_r^\bullet \phi\\
	\textstyle\stackrel{\phantom{\eqref{nmdsnmdsnmsd}}}{\Longleftrightarrow}\qquad\quad&\textstyle \limih\chart\big([\innt_r^\bullet \phi]^{-1}[\innt_r^\bullet\phi+h\cdot \psi]\big)=0\\
	\textstyle\stackrel{\eqref{nmdsnmdsnmsd}}{\Longleftrightarrow}\qquad\quad& \textstyle \limih \chart(\innt_r^\bullet h\cdot \chi)=0\\[4pt]
	\textstyle\Longleftrightarrow\qquad\quad& \textstyle\frac{\dd}{\dd h}\big|^\infty_{h=0}\:\chart(\innt_r^\bullet h\cdot \chi)= \int_r^\bullet (\dd_e\chart\cp\chi)(s)\:\dd s\in \ovl{E}\\[2pt]
	\textstyle\stackrel{\eqref{nmdsnmdsnmsd}}{\Longleftrightarrow}\qquad\quad& \textstyle\frac{\dd}{\dd h}\big|^\infty_{h=0}\:\chart\big([\innt _r^\bullet\phi]^{-1}[\innt_r^\bullet\phi+h\cdot \psi]\big)= \int_r^\bullet (\dd_e\chart\cp\Ad_{[\innt_r^s \phi]^{-1}})(\psi(s))\:\dd s\in \ovl{E}.
\end{align*} 
\item
Let $(\phi,\psi)$ be regular; and $\mu\colon (-\delta,\delta)\ni h\mapsto \innt \phi+h\cdot \psi\in G$. 
\begingroup
\setlength{\leftmarginii}{12pt}
\begin{itemize}
\item
Suppose that $\int \Ad_{[\innt_r^s \phi]^{-1}}(\psi(s))\:\dd s\in \mg$ holds; and let (shrink $\delta>0$ if necessary)
\begin{align*}
	\textstyle\gamma\colon (-\delta,\delta)\rightarrow \V,\qquad h\mapsto \chart([\innt\phi]^{-1}[\innt\phi+h\cdot \psi]).
\end{align*}
Then, we have
\vspace{-3pt}
\begin{align}
\label{salsadlklkdsadsa}
	\textstyle\dot\gamma(0)\stackrel{\text{Part }\ref{saasassansamnmsa1}}{=}\int(\dd_e\chart\cp\Ad_{[\innt_r^s \phi]^{-1}})(\psi(s))\:\dd s\stackrel{\eqref{pofdpofdpofdsddsdsfd}}{=}\dd_e\chart\big(\int \Ad_{[\innt_r^s \phi]^{-1}}(\psi(s))\:\dd s\big).
\end{align}
Since $\mu=\chart'^{-1}\cp\gamma$ (thus, $\gamma=\chart'\cp\mu$) holds for the chart 
\begin{align}
\label{podoposdddsdsaacxyc}
	\textstyle\chart'\colon \innt\phi\cdot \U=:\U' \rightarrow \V,\qquad g\mapsto \chart([\innt\phi]^{-1}\cdot g), 	
\end{align}
\eqref{salsadlklkdsadsa} shows that $\mu$ is differentiable at $0$ -- Specifically, 
we have, cf.\ Remark \ref{lkcxlkcxlkxlkcxpoidsoids}
\vspace{-5pt} 
\begin{align*}
	\textstyle\dot\mu(0)&\textstyle\stackrel{\phantom{\eqref{salsadlklkdsadsa}}}{\equiveq}\dd_0\chart'^{-1}\big( \frac{\dd}{\dd h}\big|_{h=0}\: (\chart'\cp\mu)(h)\big)\\
	&\textstyle\stackrel{\eqref{salsadlklkdsadsa}}{=}(\dd_0\chart'^{-1}\cp \dd_e\chart)\big(\int \Ad_{[\innt_r^s \phi]^{-1}}(\psi(s))\:\dd s\big)\\[-2pt]
	&\textstyle\stackrel{\phantom{\eqref{salsadlklkdsadsa}}}{=}\big(\dd_e\LT_{\innt\phi}\cp \dd_0\chart^{-1}\cp \dd_e\chart\big)\big(\int \Ad_{[\innt_r^s \phi]^{-1}}(\psi(s))\:\dd s\big)\\[-2pt]
	&\textstyle\stackrel{\phantom{\eqref{salsadlklkdsadsa}}}{=}\dd_e\LT_{\innt \phi}\big(\int \Ad_{[\innt_r^s \phi]^{-1}}(\psi(s))\:\dd s\big); 
\end{align*}
which shows \eqref{fpidpfpofdfdfd}.
\item
Suppose that $\mu$ is differentiable at $h=0$. Then, for $\chart'$ as in \eqref{podoposdddsdsaacxyc} we have, cf.\ Remark \ref{lkcxlkcxlkxlkcxpoidsoids}
\begin{align*}
	\textstyle E\ni \frac{\dd}{\dd h}\big|_{h=0}\: (\chart'\cp\mu)(h)= \textstyle\frac{\dd}{\dd h}\big|_{h=0}\:\chart\big([\innt\phi]^{-1}[\innt\phi+h\cdot \psi]\big)\textstyle \stackrel{\text{Part }\ref{saasassansamnmsa1}}{=}\int (\dd_e\chart\cp\Ad_{[\innt_r^s \phi]^{-1}})(\psi(s))\:\dd s.
\end{align*}
We obtain
\vspace{-8pt}
\begin{align*}
	\textstyle\mg\ni \dd_0\chartinv\big(\int (\dd_e\chart\cp\Ad_{[\innt_r^s \phi]^{-1}})(\psi(s))\:\dd s\big)\stackrel{\eqref{pofdpofdpofdsddsdsfd}}{=} \int \Ad_{[\innt_r^s \phi]^{-1}}(\psi(s))\:\dd s.
\end{align*}
In particular, \eqref{fpidpfpofdfdfd} holds by the previous point. \qedhere 
\end{itemize}
\endgroup
\end{enumerate}}
\endgroup
\end{proof}

\subsubsection{The Generic Case}
\label{lvclkvclkvc}
Combining Proposition \ref{saasassansamnmsa} with Theorem \ref{fdfffdfd} and Lemma \ref{dsopdsopsdopsdop}, we obtain
\begin{theorem}
\label{sasasasaasdassdsasdadds}
Suppose that $G$ is $C^k$-semiregular for $k\in \NN\sqcup\{\lip, \infty\}$. Then, $\evol_\kk$ is differentiable \deff $\mg$ is {\rm k}-complete.  In this case, $\evol^k_{[r,r']}$ is differentiable for each $[r,r']\in \COMP$, with 
\begin{align*}
	\textstyle\dd_\phi\he \evol^k_{[r,r']}(\psi)=\dd_e\LT_{\innt \phi}\big(\int \Ad_{[\innt_r^s \phi]^{-1}}(\psi(s))\:\dd s \big)\qquad\quad \forall\: \phi,\psi\in C^k([r,r'],\mg).
\end{align*}  
In particular, 
\begingroup
\setlength{\leftmargini}{17pt}
{
\renewcommand{\theenumi}{\alph{enumi})} 
\renewcommand{\labelenumi}{\theenumi}
\begin{enumerate}
\item
\label{nmcv1}
$\dd_\phi\he \evol^k_{[r,r']}\colon C^k([r,r'],\mg)\rightarrow T_{\innt\!\phi} G\he$ is linear and $C^0$-continuous for each $\phi\in C^k([r,r'],\mg)$,
\item
\label{nmcv2}
for each sequence $\{\phi_n\}_{n\in \NN}\mackarr{\kk} \phi$, and each net $\{\psi_\alpha\}_{\alpha\in I}\netarr{0} \psi$, 
we have
\begin{align*}
	\textstyle \lim_{(n,\alpha)}\dd_{\phi_n}\he\evol_{[r,r']}^k(\psi_\alpha)=\dd_{\phi}\evol_{[r,r']}^k\:(\psi).
\end{align*}	
\end{enumerate}}
\endgroup	
\end{theorem}
\begin{proof}
The first part is clear from Theorem \ref{fdfffdfd}, Lemma \ref{dsopdsopsdopsdop},  Remark \ref{fdfdfdfdfdscxycxycx}, and Proposition \ref{saasassansamnmsa}.\ref{saasassansamnmsa2}. Then, {\it\ref{nmcv2}} is clear from Lemma {\it \ref{podspodspodsodsp}}. Moreover, (by the first part) $\dd_\phi\he \evol^k_{[r,r']}$ is linear; with (cf.\ \eqref{LGPR}) 
\begin{align*}
	\textstyle\dd_\phi\he \evol^k_{[r,r']}(\psi)= \dd_{(\innt\phi,e)}\:\mult(0,\Gamma_\phi(\psi))\quad\:\:\:\text{for}\quad\:\:\: \Gamma_\phi\colon C^k([0,1],\mg)\ni \psi\rightarrow \int \Ad_{[\innt_r^s \phi]^{-1}}(\psi(s))\:\dd s \in \mg.
\end{align*}
Then, since $\Gamma_\phi$ is $C^0$-continuous by \eqref{ffdlkfdlkfd} and \ref{as1}, {\it \ref{nmcv1}} is clear from smoothness of the Lie group multiplication.
\end{proof}
\begin{corollary}
\label{lkdlkflkfdlkdf}
Suppose that $G$ is $C^k$-semiregular for $k\in \NN\sqcup\{\lip, \infty\}$, and that $\mg$ is {\rm k}-complete. Then, $\mu\colon \RR\ni h\mapsto \innt \phi+h\cdot \psi$ is of class $C^1$ for each $\phi,\psi\in C^k([r,r'],\mg)$ and $[r,r']\in \COMP$. 
\end{corollary}
\begin{proof}
Theorem \ref{fdfffdfd} and Lemma  \ref{dsopdsopsdopsdop} show that $\mu$ is continuous. Moreover, for each  $t\in \RR$, and each sequence $\{h_n\}_{n\in \NN}\rightarrow 0$, we have $\{\phi +(t+h_n)\cdot \psi\}_{n\in \NN}\mackarr{\kk} (\phi + t\cdot \psi)$; thus, 
\begin{align*}
	\textstyle\lim_n \dot\mu(t+h_n)= \lim_n \dd_{\phi+(t+ h_n )\cdot \psi}\:\evol_{[r,r']}^k( \psi)= \dd_{\phi+t\cdot \psi}\:\evol_{[r,r']}^k(\psi)=\dot\mu(t)
\end{align*}
by Theorem \ref{sasasasaasdassdsasdadds}.\textit{\ref{nmcv2}}. This shows that $\dot\mu$ is continuous, i.e., that $\mu$ is of class $C^1$.  
\end{proof}
\begin{remark}
\label{sdsddssdsdasad}
\begingroup
\setlength{\leftmargini}{17pt}{
\begin{enumerate}
\item[]
\item
It is straightforward from Corollary \ref{lkdlkflkfdlkdf},  the differentiation rules \ref{productrule} and \ref{chainrule}, as well as   
\eqref{LGPR}, \eqref{opofdpopfd}, {\rm\ref{pogfpogf}}, and {\rm\ref{homtausch}} (for $\Psi\equiv \conj_g$ there) that \eqref{fpidpfpofdfdfd} also holds for all $\phi,\psi\in \DP^k([r,r'],\mg)$, for each $[r,r']\in \COMP$. 
\item
Expectably, $\mu$ as defined in Corollary \ref{lkdlkflkfdlkdf} is even of class $C^\infty$. A detailed proof of this fact, however, would require further technical preparation -- which we do not want to carry out at this point.
\item
Expectably, the equivalence  
\begin{align*}
	\textstyle\limih \chart(\innt_r^\bullet h\cdot \chi)=0\qquad\quad \Longleftrightarrow\qquad\quad \frac{\dd}{\dd h}\big|^\infty_{h=0}\: \chart\big(\innt_r^\bullet h\cdot \chi\big)=\int_r^\bullet (\dd_e\chart\cp \chi)(s)\:\dd s\in \comp{E}
\end{align*}
also holds for $\chi\in \DP^0([r,r'],\mg)$ -- implying that Proposition \ref{saasassansamnmsa}.\ref{saasassansamnmsa1} carries over to the piecewise category. This might be shown by the same arguments (Taylor expansion) as used in the proof of Lemma 7 in \cite{RGM} (cf.\ Lemma \ref{sdsdds}) additionally using \eqref{defpio} as well as that for $n\in \NN$ fixed,  
\begin{align*}
	f\colon G^n\rightarrow G,\qquad (g_1,\dots,g_n)\mapsto g_1\cdot {\dots}\cdot g_n	
\end{align*}
is smooth, with $\dd_{(e,\dots,e)}f(X_1,\dots,X_n)=X_1+\dots+X_n$ for all $X_1,\dots,X_n\in \mg$. 
The details, however, appear to be quite technical, so that we leave this issue to another paper.   \hspace*{\fill}$\ddagger$ 
\end{enumerate}}
\endgroup
\end{remark}

\subsubsection{The Exponential Map}
We recall the conventions fixed in Sect.\ \ref{kfdlkfdlkfdscvpdfpofdofd},  specifically that $\exp= \evol^\const_{[0,1]}\cp\expal|_{\dom[\exp]}$ holds. Then, Proposition \ref{saasassansamnmsa}.\ref{saasassansamnmsa2}, for $k\equiv \const$ and $[r,r']\equiv [0,1]$ there, reads as follows.
\begin{corollary}
\label{vclklkvclkvcoivciocvoivc}
Suppose that $(\expal(X),\expal(Y))$ is regular for $X,Y\in \mg$. Then, $(-\delta,\delta)\ni h\mapsto \exp(X+h\cdot Y)$ is differentiable at $h=0$ (for $\delta>0$ suitably small) \deff $\int \Ad_{\exp(-s\cdot X)}(Y)\:\dd s\in \mg$ holds. In this case, we have
\begin{align*}
	\textstyle\frac{\dd}{\dd h}\big|_{h=0}\he\exp(X+h\cdot Y)=\dd_e\LT_{\exp(X)}\big(\int \Ad_{\exp(-s\cdot X)}(Y)\:\dd s\big).
\end{align*} 
\end{corollary}

\begin{remark}
\label{kjdskjdskjdsdkjs}
\begingroup
\setlength{\leftmargini}{17pt}{
\renewcommand{\theenumi}{{\arabic{enumi}})} 
\renewcommand{\labelenumi}{\theenumi}
\begin{enumerate}
\item[]
\item
Suppose that $G$ admits an exponential map; and that $G$ is weakly $\const$-continuous. Then, Corollary \ref{vclklkvclkvcoivciocvoivc} shows that we have
\begin{align}
\label{sdssddssd}
	\textstyle\frac{\dd}{\dd h}\big|_{h=0}\he\exp(X+h\cdot Y)=\dd_e\LT_{\exp(X)}\big(\int \Ad_{\exp(-s\cdot X)}(Y)\:\dd s\big)\qquad\quad\forall\: X,Y\in \mg
\end{align} 
\deff $\mg$ is $\const$-complete. For instance, $G$ is weakly $\const$-continuous, and $\mg$ is $\const$-complete if
\begingroup
\setlength{\leftmarginii}{12pt}
\begin{itemize}
\item
$\exp\colon \mg\rightarrow G$ is of class $C^1$, by Remark \ref{dssdsdsdds}.\ref{dssdsdsdds2}, Remark \ref{dssdsdsdds}.\ref{dssdsdsdds1}, and Corollary \ref{vclklkvclkvcoivciocvoivc}.   
\item
$G$ is abelian, by Corollary \ref{ffddf} and Remark \ref{fdfdfdfdfdscxycxycx}.
\end{itemize}
\endgroup
\item
Suppose that $\mg$ is $\const$-complete; and that $G$ admits a continuous exponential map. Then, $G$ is $C^\const$-semiregular; and $G$ is weakly $\const$-continuous by  Remark \ref{dssdsdsdds}.\ref{dssdsdsdds2} and Remark \ref{dssdsdsdds}.\ref{dssdsdsdds1}. More formally,  \eqref{sdssddssd} then reads 
\begin{align}
\label{sdsdsdlksdklsd}
\textstyle\dd_{\phi}\he \evol_\const(\psi)=\dd_e\LT_{\innt \phi}\big(\int \Ad_{[\innt_r^s \psi]^{-1}}(\psi(s))\:\dd s \big)\qquad\quad \forall\: \phi,\psi\in C^\const([0,1],\mg).
\end{align}
The same arguments as in \cite{RGM} then show that $\evol_\const$ (thus $\exp$) is of class $C^1$. More specifically, one has to replace Lemma 23 by Lemma \ref{fddfxxxxfd} in the proof of Lemma 41 in \cite{RGM}. Then, substituting Equation (95) in \cite{RGM} by \eqref{sdsdsdlksdklsd}, the proof of Corollary 13 in \cite{RGM} just carries over to the case where $k\equiveq \const$ holds (a similar adaption has been done in the proof of  Lemma \ref{podspodspodsodsp}). 

As in the Lipschitz case, cf.\ Remark 7 in \cite{RGM}, it is to be expected that a (quite elaborate and technical) induction shows that $\exp$ is even smooth if $\mg$ is Mackey complete (or, more generally, if all the occurring iterated Riemann integrals exist in $\mg$).   
\hspace*{\fill}$\ddagger$
\end{enumerate}}
\endgroup
\end{remark}

\subsection{Integrals with Parameters}
\label{dffdfd}
Given an open interval $J\subseteq \RR$ as well as $x\in J$, in the following, we denote 
\begin{align*}
	J[x]:=\{h\in \RR_{\neq 0}\:|\: x+h\in J\}.
\end{align*} 
The next theorem generalizes Theorem 5 in \cite{RGM} (with significantly simplified proof).
\begin{theorem}
\label{ofdpofdpofdpofdpofd}
Suppose that $G$ is Mackey {\rm k}-continuous for $k\in \NN\sqcup\{\lip,\infty,\const\}$ -- additionally abelian if $k\equiveq \const$ holds. Let $\Phi\colon I\times [r,r']\rightarrow \mg$ ($I\subseteq \RR$ open) be given with $\Phi(z,\cdot)\in \DIDE^k_{[r,r']}$ for each $z\in I$.  
Then, 
\begin{align*}
	\textstyle\frac{\dd}{\dd h}\big|^\infty_{h=0}\: \chart\big([\innt_r^\bullet \Phi(x,\cdot)]^{-1}[\innt_r^\bullet\Phi(x+h,\cdot)]\big)=\textstyle\int_r^\bullet (\dd_e\chart\cp \Ad_{[\innt_r^s\Phi(x,\cdot)]^{-1}})(\partial_z\Phi(x,s))\:\dd s \in \comp{E}
\end{align*}
holds for $x\in I$, provided that
\begingroup
\setlength{\leftmargini}{15pt}{
\renewcommand{\theenumi}{{\alph{enumi}})} 
\renewcommand{\labelenumi}{\theenumi}
\begin{enumerate}
\item
\label{saasaassasasa2}
We have $(\partial_z \Phi)(x,\cdot)\in C^k([r,r'],\mg)$.\footnote{More specifically, this means that for each $t\in [r,r']$, the map $I\ni z\mapsto \Phi(z,t)$ is differentiable at $z=x$ with derivative $(\partial_z\Phi)(x,t)$, such that $(\partial_z\Phi)(x,\cdot)\in C^k([r,r'],\mg)$ holds. The latter condition in particular   ensures that $\ppp^\dind_\infty((\partial_z\Phi)(x,\cdot))< \infty$ holds for each $\pp\in \SEM$ and $\dind\llleq k$, cf.\ \ref{aaaaaw2}.}
\item
\label{saasaassasasa1}
For each $\pp\in \SEM$ and $\dind\llleq k$, there exists $L_{\pp,\dind}\geq 0$, as well as $I_{\pp,\dind}\subseteq I$ open with $x\in I_{\pp,\dind}$, such that
\begin{align*}
	1/|h|\cdot\ppp^\dind_\infty(\Phi(x+h,\cdot)-\Phi(x,\cdot))\leq L_{\pp,\dind}\qquad\quad \forall\: h\in I_{\pp,\dind}[x].
\end{align*}
\vspace{-22pt}
\end{enumerate}}
\endgroup
\noindent
In particular, we have
\begin{align*}
	\textstyle\frac{\dd}{\dd h}\big|_{h=0} \he \innt\Phi(x+h,\cdot)=\textstyle \dd_e\LT_{\innt \Phi(x,\cdot)}\big(\int \Ad_{[\innt_r^s\Phi(x,\cdot)]^{-1}}(\partial_z\Phi(x,s))\:\dd s\he\big)
\end{align*}
\deff the Riemann integral on the right side exists in $\mg$.  
\end{theorem}
\begin{proof}
The last statement follows from the first statement and Lemma \ref{sdsdds} -- just as in the proof of Proposition \ref{saasassansamnmsa}.\ref{saasassansamnmsa2}. 
Now, for $x+h\in I$, we have 
\begin{align*}
	\Phi(x+h,t)=\Phi(x,t)+h\cdot \partial_z\Phi(x,t)+ h\cdot \varepsilon(x+h,t)\qquad\quad\forall\: t\in [r,r'],
\end{align*}
with $\varepsilon\colon I\times [r,r']\rightarrow\mg$ such that
\begingroup
\setlength{\leftmargini}{16pt}{
\renewcommand{\theenumi}{{\rm \roman{enumi}})} 
\renewcommand{\labelenumi}{\theenumi}
\begin{enumerate}
\item
\label{aaaaaw1}
	$\lim_{h\rightarrow 0}\varepsilon(x+h,t)=\varepsilon(x,t)=0$\hspace{128.2pt}$\forall\: t\in [r,r']$,
\item
\label{aaaaaw2}
	$\ppp^\dind_\infty(\varepsilon(x+h,\cdot))\leq L_{\pp,\dind} + \ppp_{\infty}^\dind((\partial_z\Phi)(x,\cdot))=:C_{\pp,\dind}<\infty$\hspace{18.95pt}$\forall\:h\in I_{\pp,\dind}[x]$\quad for all\quad $\pp\in \SEM$,\: $\dind\llleq k$.
\end{enumerate}}
\endgroup
\noindent
We let $\alpha:= \innt_r^\bullet \Phi(x,\cdot)$; and obtain 
\vspace{-6pt}
\begin{align}
\label{kjfdjfdccxvc}
	\textstyle[\innt_r^\bullet \Phi(x,\cdot)]^{-1}[\innt_r^\bullet\Phi(x+h,\cdot)]=\innt_r^\bullet \overbrace{h\cdot \underbrace{\Ad_{\alpha^{-1}}(\partial_z\Phi(x,\cdot))}_{\chi}+  h\cdot \underbrace{\Ad_{\alpha^{-1}}(\varepsilon(x+h,\cdot))}_{\varepsilon_h}}^{\psi_h}
\end{align}
with $\psi_h\in \DIDE_{[r,r']}^k$, because our presumptions ensure that $\chi,\varepsilon_h\in C^k([r,r'],\mg)$ holds. By Lemma \ref{opopsopsdopds} and Lemma \ref{xccxcxcxcxyxycllvovo}, for each $\pp\in \SEM$ and $\dind\llleq k$, there exists some $\pp\leq \qq\in \SEM$ with\footnote{If $k\equiveq \const$ holds, we can just choose $\dind\equiveq 0$ and $\qq\equiveq\pp$, because $G$ is presumed to be abelian in this case.}
\vspace{-3pt}
\begin{align*}
	\ppp_\infty^\dind(\psi_h)\leq |h|\cdot\qqq_\infty^\dind(\partial_z\Phi(x,\cdot)+ \varepsilon(x+h,\cdot))\stackrel{\ref{saasaassasasa1}}{\leq} |h|\cdot L_{\qq,\dind}
	\qquad\quad\forall\: h\in I_{\qq,\dind}[x].
\end{align*}
For each fixed sequence $I[x]\supseteq \{h_n\}_{n\in \NN}\rightarrow 0$, we thus have $\psi_{h_n}\mackarr{\kk} 0$. Since $G$ is Mackey k-continuous, 
this implies
\begin{align}
\label{iosdisdioddsioiosd}
	\textstyle \limin \chart(\innt_r^\bullet  h_n\cdot \chi + h_n\cdot \varepsilon_{h_n})=\limin \chart(\innt_r^\bullet  \psi_{h_n})=0. 
\end{align} 
Now, for $\delta>0$ such small that $\MD_\delta\subseteq I[x]$ holds, by \ref{aaaaaw1} and \ref{aaaaaw2}, $\{\varepsilon_{h}\}_{h\in \MD_\delta}$ fulfills the presumptions in Corollary \ref{pofdpofdop}. We thus have 
\begin{align*}
	\textstyle\limih\he \textstyle\chart\big([\innt_r^\bullet \Phi(x,\cdot)]^{-1}[\innt_r^\bullet\Phi(x+h,\cdot)]\big)= \textstyle\int_r^\bullet (\dd_e\chart\cp \Ad_{[\innt_r^s\Phi(x,\cdot)]^{-1}})(\partial_z\Phi(x,s))\:\dd s \in \comp{E}
\end{align*}
by \eqref{kjfdjfdccxvc}, \eqref{iosdisdioddsioiosd}, as well as the equivalence of \ref{podspods2} and \ref{podspods4} in Corollary \ref{pofdpofdop}.
\end{proof}
\noindent
We immediately obtain
\begin{corollary}
\label{pofdpofdpofd}
Suppose that $G$ is $C^k$-semiregular for $k\in \NN\sqcup\{\lip,\infty\}$; and that $\mg$ is {\rm k}-complete. Let $\Phi\colon I\times [r,r']\rightarrow \mg$ ($I\subseteq \RR$ open) be given with $\Phi(z,\cdot)\in \DIDE^k_{[r,r']}$ for each $z\in I$.  
Then, 
\begin{align*}
	\textstyle\frac{\dd}{\dd h}\big|_{h=0} \he \innt\Phi(x+h,\cdot)=\textstyle \dd_e\LT_{\innt \Phi(x,\cdot)}\big(\int \Ad_{[\innt_r^s\Phi(x,\cdot)]^{-1}}(\partial_z\Phi(x,s))\:\dd s\he\big)
\end{align*}
holds for $x\in I$, provided that the conditions \ref{saasaassasasa2} and \ref{saasaassasasa1} in Theorem \ref{ofdpofdpofdpofdpofd} are fulfilled. 
\end{corollary}
\begin{proof}
This is clear from Theorem \ref{fdfffdfd} and Theorem \ref{ofdpofdpofdpofdpofd}.
\end{proof}
\noindent
We furthermore obtain the following generalization of Corollary 11 in \cite{RGM}.
\begin{corollary}
\label{dspopdspodsposd}
Suppose that $G$ is Mackey {\rm k}-continuous for $k\in \{\infty,\const\}$ -- additionally abelian if $k\equiveq \const$ holds. 
Suppose furthermore that $\MX\colon I\rightarrow \dom[\exp]\subseteq \mg$ is of class $C^1$; and define $\alpha:=\exp\cp\he\MX$. Then, for $x\in I$, we have
\begin{align*}
	\textstyle\frac{\dd}{\dd h}\big|_{h=0} \: \alpha(x+h)=\dd_e\LT_{\exp(\MX(x))}\big(\int_0^1 \Ad_{\exp(-s\cdot \MX(x))}(\dot\MX(x)) \:\dd s \he \big),
\end{align*}
provided that the Riemann integral on the right side exists in $\mg$.  If this is the case for each $x\in I$, then $\alpha$ is of class $C^1$.
\end{corollary}
\begin{proof}
We let $\Phi\colon I\times [0,1]\ni(z,t)\mapsto \MX(z)$; and observe that $\alpha(z)= \innt \Phi(z,\cdot)$ holds for each $z\in I$. 
Then, the first statement is clear from Theorem \ref{ofdpofdpofdpofdpofd}. For the second statement, we suppose that
\begin{align*}
	\textstyle\dot\alpha(x) =\dd_e\LT_{\exp(\MX(x))}\big(\int_0^1 \Ad_{\exp(-s\cdot \MX(x))}(\dot\MX(x)) \:\dd s \he \big)\qquad\quad\forall\: x\in I
\end{align*}
is well defined; i.e., that the Riemann integral on the right side exists for each $x\in I$. We fix $x\in I$ and $\delta>0$ with $[x-\delta,x+\delta]\subseteq I$; and observe that
\begin{align*}
	\ppp(\MX(x+h)-\MX(x))\leq |h|\cdot \sup\{\ppp(\dot\MX(z))\:|\: z\in[x-\delta,x+\delta] \}\qquad\quad\forall\: \pp\in \SEM,\:\: |h|\leq \delta
\end{align*}
holds by \eqref{isdsdoisdiosd1}. For each sequence $I\supseteq \{h_n\}_{n\in \NN}\rightarrow 0$, we thus have 
\begin{align*}
	\DIDED_\const\supseteq \{\phi_n\}_{n\in \NN} \mackarr{\kk} \phi\in \DIDED_\const\qquad\text{for}\qquad \phi:= \expal(\MX(x))\quad\text{and}\quad \phi_n:=\expal(\MX(x+h_n))\quad\forall\: n\in \NN. 
\end{align*}
Moreover, since $\MX$ is of class $C^1$, we have
\begin{align*}
	\DIDED_\const\supseteq \{\psi_n\}_{n\in \NN} \seqarr{0} \psi\in \DIDED_\const\qquad\text{for}\qquad \psi:= \expal(\dot\MX(x))\quad\text{and}\quad \psi_n:=\expal(\dot\MX(x+h_n))\quad\forall\: n\in \NN;
\end{align*}  
so that Lemma \ref{podspodspodsodsp} shows 
\begin{align*}
	\textstyle\lim_{n\rightarrow \infty}\dot\alpha(x+h_n)=\lim_{(n,n)\rightarrow(\infty,\infty)} \xi(\phi_n,\psi_n)=\xi(\phi, \psi)=\dot\alpha(x).
\end{align*}
This shows that $\dot\alpha$ is continuous at $x$. Since $x\in I$ was arbitrary, it follows that $\alpha$ is of class $C^1$.
\end{proof}
\noindent	
For instance, we obtain the following generalization of Remark 2.3) in \cite{RGM}.
\begin{example}
\label{fdoifdoifddof}
Suppose that $G$ is Mackey $\const$-continuous and abelian. Then, for each $\phi\in C^0([r,r'],\mg)$ with $[r,r']\ni t\mapsto \int_r^t \phi(s)\:\dd s\in \dom[\exp]$, we have, cf.\ Appendix \ref{podspospodspodspodspodsp}
	\begin{align}
	\label{podspodspods}
		\textstyle\innt \phi=\exp\big(\int \phi(s)\: \dd s\he\big).
	\end{align}
 In particular, if $\dom[\exp]=\mg$ holds ($G$ admits an exponential map), then $G$ is $C^k$-semiregular for $k\in \NN\sqcup\{\lip,\infty\}$ if $\mg$ is {\rm k}-complete.
 \hspace*{\fill}$\ddagger$
\end{example}

\subsection{Differentiation at Zero}
\label{dsdsdsdsds}
In this subsection, we prove Proposition \ref{rererererr}. 
We start with some general remarks:
\vspace{6pt}

\noindent 
Let $[r,r']\in \COMP$ and $\chi\in C^0([r,r'],\mg)$ be given. For $m\geq 1$ fixed, we define $t_k:= r+ k/m\cdot (r'-r)$ for $k=0,\dots,m$; as well as $X_k:=\chi(t_k)$ for $k=0,\dots,m-1$. We furthermore define  $\chi_m\in C^0([r,r'],\mg)$ by $\chi_m(r):=X_0$ and
\begin{align*}
	\chi_m(t)= X_k + (t-t_k)/(t_{k+1}-t_k)\cdot (X_{k+1}-X_k)\qquad\quad\forall\: t\in (t_k,t_{k+1}],\:\: k=0,\dots,m-1.
\end{align*}
Then, $\{\chi_m\}_{m\geq 1}\subseteq C^0([r,r'],\mg)$ constructed in this way, admits the following properties:
\begingroup
\setlength{\leftmargini}{19pt}
{
\renewcommand{\theenumi}{{\bf\alph{enumi})}} 
\renewcommand{\labelenumi}{\theenumi}
\begin{enumerate}
\item
\label{qwwqwqwq1}
We have $\lim_m \ppp_\infty(\chi- \chi_m)=0$ for each $\pp\in \SEM$.
\item
\label{qwwqwqwq2}
We have $\gamma_{h,m}:=h\cdot \dd_e\chart(\int_r^\bullet \chi_m(s)\: \dd s) \in E$ for each $h\in \RR$ and $m\geq 1$.
\item
\label{qwwqwqwq3}
Since $\im[\chi]\subseteq \mg$ is bounded, also $\{\im[\chi_m]\}_{m\geq 1}\subseteq \mg,\: \{\im[\dd_e\chart(\int_r^\bullet \chi_m(s)\: \dd s)]\}_{m\geq 1}\subseteq E$ are bounded. Thus,
\begingroup
\setlength{\leftmarginii}{12pt}
{
\renewcommand{\theenumi}{{\roman{enumi}})} 
\renewcommand{\labelenumi}{\theenumi}
\begin{itemize}
\item
	For each $\pp\in \SEM$, there exists some $C_\pp>0$ with
\begin{align}
\label{pdspopodspodspoapoapoa}
	\pp_\infty(\gamma_{h,m})\leq |h|\cdot C_\pp\qquad\quad\forall\: h\in \RR,\:\: m\geq 1.
\end{align}
\item
	For $\delta>0$ suitably small,
	\vspace{-6pt}
\begin{align*}
	\mu_{h,m}:=\chartinv\cp \gamma_{h,m}\in C^1([r,r'],G)
\end{align*}
is well defined for each $|h|\leq\delta$, $m\geq 1$; and we define
\begin{align}
\label{podsapoaoipdaiaipoa}
	\textstyle\chi_{h,m}:=\Der(\mu_{h,m})\stackrel{\eqref{kldlkdldsl}}{=} h\cdot \dermapdiff(\gamma_{h,m},\dd_e\chart(\chi_m))\qquad\quad \forall\: |h|\leq \delta,\:\: m\geq 1.
\end{align} 
Moreover, for each fixed open neighbourhood $V\subseteq G$ of $e$, there exists some $0<\delta_V\leq \delta$ with
\begin{align}
\label{fdpofdpfdpigfoihfdih}
	\textstyle V\ni \mu_{h,m}\equiveq \innt_r^\bullet \chi_{h,m} \qquad\quad \forall\: |h|\leq \delta_V,\:\: m\geq 1.
\end{align}
\vspace{-20pt}
\end{itemize}}
\endgroup
\end{enumerate}}
\endgroup
\noindent
Modifying the proof of Proposition 7 in \cite{RGM}, we obtain the
\begin{proof}[Proof of Proposition \ref{rererererr}]
Suppose first that {\it\ref{dpopoffdpolkfdlkfdlklkfd2}} holds; and let 
\begin{align*}
	\textstyle A:=\chart(\innt_r^\bullet h_n\cdot \chi + h_n\cdot \varepsilon_n)\qquad\quad\qquad B:= \int_r^\bullet (\dd_e\chart\cp \chi)(s)\:\dd s. 
\end{align*} 
Then, {\it\ref{dpopoffdpolkfdlkfdlklkfd1}} is clear from 
$\pp_\infty(A)\leq |h_n|\cdot \pp_\infty(1/h_n\cdot A- B)+|h_n|\cdot \pp_\infty(B)$. 
\vspace{6pt}

\noindent
Suppose now that {\it\ref{dpopoffdpolkfdlkfdlklkfd1}} holds -- i.e. that we have $\limin \chart(\innt_r^\bullet \psi_n) =0$ with 
\begin{align*}
	\psi_n:=  h_n\cdot \chi+ h_n\cdot \varepsilon_n\qquad\quad\forall\: n\in \NN.
\end{align*} 
We now have to show that for $\pp\in \SEM$ fixed, the expression  
\begin{align*}
	\Delta_n&:=\textstyle 1/|h_n|\cdot \cpp_\infty\big(\chart(\innt_r^\bullet \psi_n)-h_n\cdot  \int_r^\bullet \dd_e\chart(\chi(s)) \:\dd s\big)
\end{align*}
tends to zero for $n\rightarrow \infty$.  
For this, we choose $\pp\leq \qq\in \SEM$ and $e\in V\subseteq G$ as in Lemma \ref{hghghggh}; 
and let 
\begin{align*}
	\{\chi_m\}_{m\geq 1},\quad\{\chi_{h,m}\}_{m\geq 1},\quad\{\gamma_{h,m}\}_{m\geq 1},\quad\{\mu_{h,m}\}_{m\geq 1},\quad\delta_V>0
\end{align*}
be as above -- with $\delta_V$ additionally such small that $\innt_r^\bullet \psi_{n} \in V$ holds for each $n\in \NN$ with $|h_n|<\delta_V$. We choose $\ell\in \NN$  such large that $\{h_n\}_{n\geq \ell}\subseteq (-\delta_V,\delta_V)$ holds. Then, for each $n\geq \ell$ and $m\geq 1$, we obtain from \eqref{ffdlkfdlkfd} (second step),  \eqref{fdpofdpfdpigfoihfdih} and Lemma \ref{hghghggh} (fifth step), as well as \eqref{podsapoaoipdaiaipoa} (last step) that
\begin{align*}
	\Delta_n&\textstyle\leq 1/|h_n|\cdot \pp_\infty\big(\chart(\innt_r^\bullet \psi_n)-h_n\cdot  \dd_e\chart(\int_r^\bullet \chi_m(s) \:\dd s)\big)\:\:  + \:\: \cpp_\infty\big(\int_r^\bullet \dd_e\chart(\chi(s)) \:\dd s-\int_r^\bullet \dd_e\chart(\chi_m(s)) \:\dd s\big)\\
&\leq\textstyle 1/|h_n| \cdot\textstyle\pp_\infty\big(\chart(\innt_r^\bullet \psi_n)-\gamma_{h_n,m}\big)\hspace{81.1pt}+ \:\:\int (\pp\cp\dd_e\chart)(\chi(s)-\chi_m(s)) \:\dd s\\
	&=\textstyle 1/|h_n| \cdot\textstyle\pp_\infty\big(\chart(\innt_r^\bullet \psi_n)-\chart(\mu_{h_n,m})\big)\hspace{64.35pt}+ \:\: \int \ppp(\chi(s)-\chi_m(s)) \:\dd s\\
		&=\textstyle 1/|h_n| \cdot\textstyle\pp_\infty\big(\chart(\innt_r^\bullet \psi_n)-\chart(\innt_r^\bullet \chi_{h_n,m})\big)\hspace{50.5pt}+ \:\: \int \ppp(\chi(s)-\chi_m(s)) \:\dd s\\
	&\textstyle \leq 1/|h_n|\cdot \int \qqq\big(\psi_n(s)-\chi_{h_n,m}(s)\big)\:\dd s
	\hspace{65.4pt}+ \:\: \int \ppp(\chi(s)-\chi_m(s)) \:\dd s	
	\\[0.2pt]
	&\textstyle \leq\:\:\: \int \qqq(\varepsilon_n(s))\:\dd s \:\:\:+ \:\:\: \int \qqq\big(\chi(s)-\dermapdiff(\gamma_{h_n,m}(s),\dd_e\chart(\chi_m(s)))\big)\:\dd s
	\:\:\:+ \:\:\: (r'-r)\cdot \ppp_\infty(\chi-\chi_m) 
\end{align*}
holds. By Lebesgue's dominated convergence theorem and {\it\ref{saaaaaw1}}, {\it\ref{saaaaaw2}}, the first term tends to zero for $n\rightarrow \infty$; and, by \ref{qwwqwqwq1}, the third term tends to zero for $m\rightarrow \infty$. Thus, $\varepsilon>0$ given, there exists some $\ell_\varepsilon\geq \ell$, such that both the first-, and the third term is bounded by $\varepsilon/4$ for all $m,n\geq \ell_\varepsilon$. 
Moreover, since $\chi=\dermapdiff(0,\dd_e\chart(\chi))$ holds (second step), we can estimate the second term by
\begin{align}
\label{ksdkldsklsksdlklsdlk}
\begin{split}
	\textstyle\int \qqq\big(\chi(s)\:-&\:\dermapdiff(\gamma_{h_n,m}(s),\dd_e\chart(\chi_m(s)))\big)\:\dd s\\[2pt]
	&\leq \textstyle (r'-r)\cdot\qqq_\infty\big(\chi-\dermapdiff(\gamma_{h_n,m},\dd_e\chart(\chi_m))\big)\\	
	&\textstyle=(r'-r)\cdot\qqq_\infty\big(\dermapdiff(0,\dd_e\chart(\chi))-\dermapdiff(\gamma_{h_n,m},\dd_e\chart(\chi_m))\big)
	\\[3pt]
	&\textstyle\leq (r'-r)\cdot\qqq_\infty\big(\dermapdiff(0,\dd_e\chart(\chi))-\dermapdiff(\gamma_{h_n,m},\dd_e\chart(\chi))\big)\textstyle\\[3pt]
	&\textstyle\quad\he +(r'-r)\cdot\qqq_\infty\big(\dermapdiff(\gamma_{h_n,m},\dd_e\chart(\chi-\chi_m))\big).
\end{split}
\end{align}
\begingroup
\setlength{\leftmargini}{12pt}
\begin{itemize}
\item
Since $\im[\chi]$ is compact, increasing $\ell_\varepsilon$ if necessary, by \eqref{pdspopodspodspoapoapoa}, we can achieve that the fourth line in \eqref{ksdkldsklsksdlklsdlk} is bounded by $\varepsilon/4$ for each $n,m\geq \ell_\varepsilon$.
\item
To estimate the last line in \eqref{ksdkldsklsksdlklsdlk}, we choose $\qq\leq \mm\in \SEM$ as in \eqref{omegaklla}; and increase $\ell_\varepsilon$ in such a way (use \eqref{pdspopodspodspoapoapoa}) that  
$\mm_\infty(\gamma_{h_n,m})\leq  1$ holds for all $n,m\geq \ell_\varepsilon$; thus,
\vspace{-8pt}
\begin{align*}
	\qqq_\infty\big(\dermapdiff(\gamma_{h_n,m},\dd_e\chart(\chi-\chi_m))\big) \stackrel{\eqref{omegaklla}}{\leq} \mmm_\infty(\chi-\chi_m).
\end{align*}  
Then, it is clear from \ref{qwwqwqwq1} that for $\ell_\varepsilon'\geq \ell_\varepsilon$ suitably large, the last line in \eqref{ksdkldsklsksdlklsdlk} is bounded by $\varepsilon/4$ for all $m,n\geq \ell'_\varepsilon$.  
\end{itemize}
\endgroup
\noindent
We thus have $\Delta_n\leq \varepsilon$ for each $n\geq \ell'_\varepsilon\in \NN$; which shows $\lim_n \Delta_n=0$.
\end{proof} 

\section{Extension: The Metrizable Category}
\label{sklkdslkdskldslkds}
We recall that a Hausdorff locally convex vector space is said to be \emph{metrizable} \defff it admits a metric that generates the topology thereon. We furthermore recall that $G$ is said to be $C^k$-regular \defff $G$ is $C^k$-semiregular such that $\evol_\kk$ is smooth w.r.t.\ the $C^k$-topology. 

After this paper had been put on the arXiv, the author's attention was drawn by Gl\"ockner and Schmeding to the fact that in metrizable locally convex vector spaces, convergence of a sequence implies its Mackey convergence (and vice versa). Specifically, it was argued that the following two results will hold:
\begin{lemma}
\label{mnxmncxmncxn}
Suppose that $\mg$ is metrizable; and let $k\in \NN\sqcup\{\lip,\infty,\const\}$. Then, the following conditions are equivalent:
\vspace{-5pt}
\begingroup
\setlength{\leftmargini}{20pt}{
\renewcommand{\theenumi}{{\roman{enumi}})} 
\renewcommand{\labelenumi}{\theenumi}
\begin{enumerate}
\item
\label{mnxmncxmncxn1}
$G$ is $C^k$-continuous. 
\item
\label{mnxmncxmncxn2}
$G$ is sequentially {\rm k}-continuous.
\item
\label{mnxmncxmncxn3}
$G$ is Mackey {\rm k}-continuous.
\end{enumerate}}
\endgroup
\end{lemma}
\begin{corollary}
\label{ofdpoopfdopfd}
	Suppose that $\mg$ is a Fr\'{e}chet space; and let $k\in \NN\sqcup\{\infty\}$. Then, $G$ is $C^k$-regular \deff $G$ is $C^k$-semiregular. 
\end{corollary}
\begin{proof}
The one implication is evident. Suppose thus that $G$ is $C^k$-semiregular. Then, $G$ is Mackey k-continuous by Theorem \ref{fdfffdfd}; so that $\evol_\kk$ is $C^k$-continuous by Lemma \ref{mnxmncxmncxn}. Since $\mg$ is complete (thus, integral complete and Mackey complete), Theorem 4 in \cite{RGM} shows that $\evol_\kk$ is smooth, i.e., that $G$ is $C^k$-regular.
\end{proof}
\noindent
The rest of this section is dedicated to a selfcontained proof of Lemma \ref{mnxmncxmncxn}.   
\subsection*{Some Standard Facts:}
Let $F$ be a Hausdorff locally convex vector space, with system of continuous seminorms $\SEMM$. A subsystem $\SEMMM\subseteq \SEMM$ is said to be a fundamental system \defff $\{\B_{\hh,\:\varepsilon}(0)\}_{\hh\in \SEMMM,\varepsilon>0}$ is a local base of zero in $F$. 
We recall that
\begin{lemma}
\label{xcxcxccxxcxc}
Let $\SEMMM\subseteq \SEMM$ be a fundamental system, and $\SEMG\subseteq \SEMM$ a subsystem. Then, the following statements are equivalent:
\begingroup
\setlength{\leftmargini}{17pt}
{
\renewcommand{\theenumi}{\small{\bf\arabic{enumi})}}  
\renewcommand{\labelenumi}{\theenumi}
\begin{enumerate}
\item
\label{xcxcxccxxcxc1}
$\SEMG$ is a fundamental system.
\item
\label{xcxcxccxxcxc2}
To each $\hh\in \SEMMM$, there exist $c>0$ and $\zzs\in  \SEMG$ with $\hh\leq c\cdot \zzs$.
\end{enumerate}}
\endgroup 
\end{lemma}
\begin{proof}
If $\SEMG$ is a fundamental system, then \ref{xcxcxccxxcxc2} follows from Proposition 22.6 in \cite{MV} when applied to the identity $\id_F$. Suppose thus that \ref{xcxcxccxxcxc2} holds; and let $V\subseteq F$ be open with $0\in V$. We choose $\hh\in \SEMMM$ with $\B_{\hh,\varepsilon}(0)\subseteq V$, fix $c>0$ and $\zzs\in \SEMG$ with $\hh\leq c\cdot \zzs$; and observe that $\B_{\zzs,\frac{\varepsilon}{c}}(0)\subseteq \B_{\hh,\varepsilon}(0)\subseteq V$ holds. Since $\B_{\zzs,\frac{\varepsilon}{c}}(0)\subseteq F$ is open, \ref{xcxcxccxxcxc1} follows.
\end{proof}
\begin{lemma}
\label{pofdofdopdf}
The following statements are equivalent:
\begingroup
\setlength{\leftmargini}{17pt}
{
\renewcommand{\theenumi}{\small{\bf\arabic{enumi})}} 
\renewcommand{\labelenumi}{\theenumi}
\begin{enumerate}
\item
\label{pofdofdopdf1}
$F$ is metrizable.
\item
\label{pofdofdopdf2}
There exists a countable fundamental system $\{\qq[m]\:|\: m\in \NN\}\subseteq \SEMM$.
\item
\label{pofdofdopdf3}
There exists $\{\qq[m]\:|\: m\in \NN\}\subseteq \SEMM$ as in \ref{pofdofdopdf2} with $\qq[m]\leq \qq[m+1]$ for each $m\in \NN$.
\end{enumerate}}
\endgroup 
\end{lemma}
\begin{proof}
The equivalence of \ref{pofdofdopdf1} and \ref{pofdofdopdf2} is covered by Proposition 25.1 in \cite{MV}. It is furthermore clear that \ref{pofdofdopdf3} implies \ref{pofdofdopdf2}. Let thus $\{\qq[m]\:|\: m\in \NN\}\subseteq \SEMM$ be as in \ref{pofdofdopdf2}; and define 
\begin{align*}
	\SEMG:=\{\oo[m]\equiv\qq[0]+{\dots}+\qq[m] \:|\: m\in \NN\}\subseteq \SEMM.
\end{align*}
Since $\qq[m]\leq \oo[m]$ holds for each $m\in \NN$, Lemma \ref{xcxcxccxxcxc} shows that $\SEMG$ is a fundamental system; which establishes \ref{pofdofdopdf3}. 
\end{proof}
\noindent
Let $\SEMMM\subseteq \SEMM$ be a fundamental system. 
We write $\{X_n\}_{n\in \NN} \mackarrr X$ for $\{X_n\}_{n\in \NN}\subseteq F$ and $X\in F$ \defff 
\begin{align}
\label{podsopdspodssdnbbn}
	\hh(X-X_n)\leq \mackeyconst_\hh\cdot \lambda_{n} \qquad\quad\forall\: n\geq \mackeyindex_\hh,\:\:\hh\in \SEMMM
\end{align}
holds for certain $\{\mackeyconst_\hh\}_{\hh\in \SEMMM}\subseteq \RR_{\geq 0}$, $\{\mackeyindex_\hh\}_{\hh\in \SEMMM}\subseteq \NN$, and $\RR_{\geq 0}\supseteq \{\lambda_{n}\}_{n\in \NN}\rightarrow 0$.
\begin{remark} 
It is immediate from Lemma \ref{xcxcxccxxcxc} that the definition made in \eqref{podsopdspodssdnbbn} does not depend on the explicit choice of the fundamental system $\SEMMM$.\hspace*{\fill}$\ddagger$
\end{remark}
\noindent
We obtain
\begin{lemma}
\label{pofsdisfdodjjxcycxjsdsd}
Suppose that $F$ is metrizable; and let $\{X_n\}_{n\in \NN}\subseteq F$ be a sequence with $\{X_n\}_{n\in \NN}\rightarrow X\in F$. Then, we have $\{X_n\}_{n\in \NN} \mackarrr X$.  
\end{lemma}
\begin{proof}
Although this statement is well known from the literature (cf., e.g., 4.\ Proposition in Sect.\ 10.1 in \cite{JAR}), for completeness reasons, we provide an elementary proof that is adapted to our particular formulation of Mackey convergence, cf.\ Appendix \ref{podspospodspodsdsdspodspodspsdsss}.
\end{proof}
\noindent
We recall that the $C^k$-topology on 
$\F_k:=C^k([0,1],F)$ for $k\in \NN\sqcup\{\lip,\infty,\const\}$ is  the Hausdorff locally convex topology that is generated by the   seminorms $\SEMMM_k:=\{\qq_\infty^\dind\:|\: \qq\in \SEMM,\: \dind\llleq k\}$ (cf.\ Sect. \ref{dsdsdssdfdffddfdfdd}). Since $\SEMMM_k$ is a fundamental system, the definition made in \eqref{podsopdspodssdnbbn} coincides with the definition made in \eqref{podsopdspodssd}. 
We furthermore recall that
\begin{lemma}
\label{fdfddfdfdf}
If $F$ is metrizable, then $C^k([0,1],F)$ is metrizable for each $k\in \NN\sqcup\{\lip,\infty,\const\}$.
\end{lemma}
\begin{proof}
Confer, e.g., Appendix \ref{podspospodsssdds}. 
\end{proof}
\subsection*{The Proof of Lemma \ref{mnxmncxmncxn}:}
We obtain from Lemma \ref{pofdofdopdf} and Lemma \ref{fdfddfdfdf}:
\begin{corollary}
\label{dffd}
Suppose that $\mg$ is metrizable; and let $k\in \NN\sqcup\{\lip,\infty,\const\}$. Then, $G$ is sequentially {\rm k}-continuous \deff $G$ is Mackey {\rm k}-continuous.
\end{corollary}
\begin{proof}
Let $\{\phi_n\}_{n\in \NN}\subseteq \DIDED_\kk$, and $\phi\in \DIDED_\kk$ be given. 
\begingroup
\setlength{\leftmargini}{12pt}
\begin{itemize}
\item
Evidently, $\{\phi_n\}_{n\in \NN}\mackarr{\kk}\phi$ implies $\{\phi_n\}_{n\in \NN}\seqarr{\kk}\phi$; so that $G$ is Mackey {\rm k}-continuous if $G$ is sequentially {\rm k}-continuous. 
\item
By Lemma \ref{fdfddfdfdf}, $C^k([0,1],\mg)$ is metrizable. Lemma \ref{pofsdisfdodjjxcycxjsdsd} thus shows that $\{\phi_n\}_{n\in \NN}\seqarr{\kk}\phi$ implies $\{\phi_n\}_{n\in \NN}\mackarr{\kk}\phi$. Consequently, $G$ is sequentially {\rm k}-continuous if $G$ is Mackey {\rm k}-continuous. \qedhere  
\end{itemize}
\endgroup
\end{proof}
\noindent
We are ready for the 
\begin{proof}[Proof of Lemma \ref{mnxmncxmncxn}]
The equivalence of \ref{mnxmncxmncxn2} and \ref{mnxmncxmncxn3} is covered by Corollary \ref{dffd}. Moreover, since Lemma \ref{fdfddfdfdf} shows that $C^k([0,1],\mg)$ is metrizable (thus, first countable), the equivalence of \ref{mnxmncxmncxn1} and \ref{mnxmncxmncxn2} is clear from Remark \ref{dssdsdsdds}.\ref{dssdsdsdds1} as well as Remark \ref{dssdsdsdds}.\ref{dssdsdsdds6}.   
\end{proof}

\addtocontents{toc}{\protect\setcounter{tocdepth}{0}}
\appendix

\section*{APPENDIX}

\section{Appendix}
\subsection{Bastiani's Differential Calculus}
\label{Diffcalc}
In this Appendix, we recall the differential calculus from \cite{HA,HG,MIL,KHN}, cf.\ also Sect.\ 3.3.1 in \cite{RGM}. 
\vspace{6pt}

\noindent
Let $E$ and $F$ be Hausdorff locally convex vector spaces. 
A map $f\colon U\rightarrow E$, with $U\subseteq F$ open, is said to be  
differentiable at $x\in U$ \defff 
\begin{align*}
	\textstyle(D_v f)(x):=\lim_{h\rightarrow 0}1/h\cdot (f(x+h\cdot v)-f(x))\in E
\end{align*} 
exists for each $v\in F$. Then, $f$ is said to be differentiable \defff it is differentiable at each $x\in U$. More generally, $f$ is said to be $k$-times differentiable for $k\geq 1$ \defff 
	\begin{align*}
	D_{v_k,\dots,v_1}f\equiv D_{v_k}(D_{v_{k-1}}( {\dots} (D_{v_1}(f))\dots))\colon U\rightarrow E
\end{align*}
is well defined for each $v_1,\dots,v_k\in F$ -- implicitly meaning that $f$ is $p$-times differentiable for each $1\leq p\leq k$. In this case, we define
\begin{align*}
	\dd^p_xf(v_1,\dots,v_p)\equiv \dd^p f(x,v_1,\dots,v_p):=D_{v_p,\dots,v_1}f(x)\qquad\quad\forall\: x\in U,\:\:v_1,\dots,v_p\in F
\end{align*} 	
for $p=1,\dots,k$; and let $\dd f\equiv \dd^1 f$, as well as $\dd_x f\equiv \dd^1_x f$ for each $x\in U$.  
Then, $f$ is said to be  
\begingroup
\setlength{\leftmargini}{11pt}
\begin{itemize}
\item
of class $C^0$ \defff it is continuous -- In this case, we let $\dd^0 f\equiv f$.
\item
of class $C^k$ for $k\geq 1$ \defff it is $k$-times differentiable, such that 
\begin{align*}
	\dd^pf\colon U\times F^p\rightarrow E,\qquad (x,v_1,\dots,v_p)\mapsto D_{v_p,\dots,v_1}f(x)
\end{align*} 
is continuous for each $p=0,\dots,k$.  
In this case, $\dd^p_x f$ is symmetric and $p$-multilinear for each $x\in U$ and $p=1,\dots,k$, cf.\ \cite{HG}.  
\item
of class $C^\infty$ \defff it is of class $C^k$ for each $k\in \NN$. 
\end{itemize}
\endgroup
\noindent
We have the following differentiation rules \cite{HG}:
\begingroup
\setlength{\leftmargini}{16pt}
{
\renewcommand{\theenumi}{{\bf \alph{enumi})}} 
\renewcommand{\labelenumi}{\theenumi}
\begin{enumerate}
\item
\label{iterated}
A map $f\colon F\supseteq U\rightarrow E$ is of class $C^k$ for $k\geq 1$ \defff $\dd f$ is of class $C^{k-1}$ when considered as a map $F' \supseteq U' \rightarrow E$ for $F'\equiv F \times F$ and $U'\equiv U\times F$.
\item
\label{linear}
If $f\colon U\rightarrow F$ is linear and continuous, then $f$ is smooth; with $\dd^1_xf=f$ for each $x\in E$, as well as $\dd^kf=0$ for each $k\geq 2$.  
\item
\label{chainrule}
	Suppose that $f\colon F\supseteq U\rightarrow U'\subseteq F'$ and $f'\colon F'\supseteq U'\rightarrow  F''$ are of class $C^k$ for $k\geq 1$, for Hausdorff locally convex vector spaces $F,F',F''$. Then, $f'\cp f\colon U\rightarrow F''$ is of class $C^k$ with 
	\begin{align*}
		\dd_x(f'\cp f)=\dd_{f(x)}f'\cp \dd_x f\qquad\quad \forall\: x\in U.
	\end{align*}
\item
\label{productrule}
	Let $F_1,\dots,F_m,E$ be Hausdorff locally convex vector spaces, and let $f\colon F_1\times {\dots} \times F_m\supseteq U\rightarrow E$ be of class $C^0$. Then, $f$ is of class $C^1$ \deff  for $p=1,\dots,m$, the ``partial derivative''
	\begin{align*}
		\partial_p f \colon U\times F_p\ni((x_1,\dots,x_m),v_p)&\textstyle\mapsto \lim_{h\rightarrow 0} 1/h\cdot (f(x_1,\dots, x_p+h\cdot v_p,\dots,x_m)-f(x_1,\dots,x_m))
	\end{align*}
	exists in $E$, and is continuous. In this case, we have
	\begin{align*}
		\textstyle\dd_{(x_1,\dots,x_m)} f(v_1,\dots,v_m)&\textstyle=\sum_{p=1}^m\partial_p f((x_1,\dots,x_m),v_p)\\
		&\textstyle= \sum_{p=1}^m \dd f((x_1,\dots,x_m),(0,\dots,0, v_p,0,\dots,0))
	\end{align*}
	for each $(x_1,\dots,x_m)\in U$, and $v_p\in F_p$ for $p=1,\dots,m$.
\end{enumerate}}
\endgroup

\subsection{Proof of Lemma \ref{xccxcxcxcxyxycllvovo}}
\label{asassadsdsdsdsdsdsds}
In this appendix, we prove
\begin{customlem}{\ref{xccxcxcxcxyxycllvovo}}
	Let $[r,r']\in \COMP$, and $\phi\in \DIDE_{[r,r']}$ be fixed. Then, 
	for each $\pp\in \SEM$, there exists some $\pp\leq \qq\in \SEM$ with
		\begin{align*}
		\ppp^\lip_\infty\big(\Ad_{[\innt_r^\bullet\phi]^{-1}}(\psi)\big)\leq \qqq_\infty^\lip(\psi)\qquad\quad\forall\: \psi\in C^\lip([r,r'],\mg).
	\end{align*}
\end{customlem}
\begin{proof}
By definition, there exists some  $\mu\in C^1(I,G)$, for $I$ an open interval containing $[r,r']$, with $\Der(\mu)|_{[r,r']}=\phi$ and $\mu(r)=e$.  
	We now have to show that 
\begin{align*}
		C^\lip([r,r'],\mg)\ni\chi\colon [r,r']\ni t\mapsto \Ad_{\mu^{-1}(t)}(\psi(t)) 
\end{align*}	
 holds, for each fixed 
	$\psi\in C^\lip([r,r'],\mg)$.   
	For this, we let $\pp\in \SEM$ be fixed; and obtain
\begin{align}
\label{opsdopsdopsdopsd0}
\ppp(\chi(t)-\chi(t'))
	&\leq \ppp\big(\Ad_{\mu^{-1}(t)}(\psi(t)-\psi(t'))\big)
	 + \ppp\big(\big(\Ad_{\mu^{-1}(t)}-\Ad_{\mu^{-1}(t')}\big)(\psi(t'))\big).
\end{align} 
\begingroup
\setlength{\leftmargini}{12pt}
\begin{itemize}
\item
	We let $\compact:=\im[\mu^{-1}]$, choose $\pp\leq \ww\in \SEM$ as in \ref{as1} for $\vv\equiv \pp$ there; and obtain
\begin{align}
\label{opsdopsdopsdopsd1}
	\ppp(\chi(t))&\leq \www(\psi(t))\hspace{169.9pt}\forall\: t\in [r,r'],\\[3pt]
\label{opsdopsdopsdopsd2}
	\ppp\big(\Ad_{\mu^{-1}(t)}(\psi(t)-\psi(t'))\big)&\leq  \www(\psi(t)-\psi(t'))\leq \Lip(\www, \psi)\cdot |t-t'|\qquad\quad\forall\: t,t'\in [r,r']. 
\end{align}
\item 
The map $\alpha\colon I\times \mg \ni(s,X)\rightarrow \partial_s\Ad_{\mu^{-1}(s)}(X)$ is well defined, continuous, and linear in the second argument. By Lemma \ref{alalskkskaskaskas} applied to $\compacto\equiv \compact$, there thus exists some $\pp\leq \mm\in \SEM$ with
\begin{align*}
	(\ppp\cp \alpha)(s,X)\leq \mmm(X)\qquad\quad\forall\: s\in [r,r'],\:\: X\in \mg.
\end{align*}
Then, we obtain from \eqref{isdsdoisdiosd1} that
\begin{align}
\begin{split}
\label{opsdopsdopsdopsd3}
	\textstyle\ppp\big(\big(\Ad_{\mu^{-1}(t)}-\Ad_{\mu^{-1}(t')}\big)(\psi(t'))\big)&\textstyle\leq \int_{t'}^{t}  \ppp\big(\partial_s \Ad_{\mu^{-1}(s)}(\psi(t'))\big)\:\dd s\\
	&\textstyle= \int_{t'}^{t}  (\ppp\cp\alpha)(s,\psi(t'))\:\dd s\\
	&\leq \mmm_\infty(\psi)\cdot |t-t'|
\end{split}
\end{align}
holds, for each $t,t'\in [r,r']$ with $t'\leq t$.
\end{itemize}
\endgroup
\noindent
We choose $\qq\in \SEM$ with $\qq\geq 2\cdot \max(\mm,\ww)$ (i.e., $\pp,\mm,\ww\leq \qq$); and obtain
\vspace{-4pt}
\begin{align}
\label{mnmnvmncv}
	\ppp_\infty(\chi)\stackrel{\eqref{opsdopsdopsdopsd1}}{\leq} \www_\infty(\psi)\leq \qqq_\infty(\psi).
\end{align}
We furthermore obtain from \eqref{opsdopsdopsdopsd0}, \eqref{opsdopsdopsdopsd2}, \eqref{opsdopsdopsdopsd3} that
\begin{align*}
	\ppp(\chi(t)-\chi(t'))
\leq \Lip(\www, \psi)\cdot |t-t'|   +\mmm_\infty(\psi)  \cdot |t-t'|\leq   \qqq_\infty^\lip(\psi) \cdot |t-t'|
\end{align*}
holds for each $t,t'\in [r,r']$; thus,
\begin{align*}
	\Lip(\ppp,\chi)\leq \qqq_\infty^\lip(\psi) \qquad\quad\stackrel{\eqref{mnmnvmncv}}{\Longrightarrow}\qquad\quad \ppp_\infty^\lip(\chi)\leq \qqq_\infty^\lip(\psi),
\end{align*} 
which proves the claim.
\end{proof}

\subsection{Proof of Equation \eqref{asnbsanbsanbsanbyxxy}}
\label{asassadsdsdsdsdsdsddssddss}
In this appendix, we show
\begin{align}
\tag{\ref{asnbsanbsanbsanbyxxy}}
	\Lip(\ppp,\phi|_{[t_ {n+1}, 1]})\leq 2\cdot C[\rhon,1]^2\cdot\max\big(\delta_0^{-8}\cdot \ppp_\infty^\lip(\phi_{0}),\dots,\delta_{n}^{-8}\cdot \ppp_\infty^\lip(\phi_{n})\big).
\end{align}
\begin{proof}[Proof of Equation \eqref{asnbsanbsanbsanbyxxy}]
We let $\wph_n:= \rho_n\cdot (\phi_n\cp\varrho_n)$ for each $n\in \NN$; so that
\begin{align*}
	\Lip(\ppp,\wph_n)\leq  2\cdot \delta_n^{-8}\cdot C[\rhon,1]^2\cdot\ppp_\infty^\lip(\phi_n)
\end{align*}
holds by \eqref{dspoopsdopdsaasassddsds}.  
Then, for $t,t'\in [t_{\ell+1},t_{\ell}]$ with $\ell\in \NN$, we have
\begin{align}
\label{wweewewewewddsads12}
\begin{split}
	\ppp(\phi(t)-\phi(t'))&\textstyle = \ppp(\wph_\ell(t)-\wph_\ell(t'))\\
&\leq \Lip(\ppp,\wph_\ell)\cdot |t-t'|\\
&\leq  2\cdot  C[\rhon,1]^2\cdot\delta_\ell^{-8}\cdot\ppp_\infty^\lip(\phi_\ell)\cdot |t-t'|.
\end{split} 
\end{align}
Moreover, for $t\in [t_{(\ell+1)+m},t_{\ell+m}]$ and $t'\in [t_{\ell+1},t_{\ell}]$, with $m\geq 1$ and $\ell\in \NN$, we have
\begin{align}
\label{wweewewewewddsads123}
\begin{split}
	\textstyle\ppp(\phi(t)-\phi(t'))
	&\textstyle\leq 
	\ppp(\phi(t)-\phi(t_{\ell+m}))\\
	&\textstyle\quad\he + \sum_{k=m-1}^{1} \ppp(\phi(t_{(\ell+1)+k})-\phi(t_{\ell + k}))  \\
	&\textstyle\quad\he  + \ppp(\phi(t_{\ell+1})-\phi(t'))\\[2pt]
	&\textstyle\leq 
	\ppp(\wph_{\ell+m}(t)-\wph_{\ell+m}(t_{\ell+m}))\\
	&\textstyle\quad\he + \sum_{k=m-1}^{1} \ppp(\wph_{\ell +k}(t_{(\ell+1)+k})-\wph_{\ell+k}(t_{\ell+k}))\\
	&\quad\he    + \ppp(\wph_{\ell}(t_{\ell+1})-\wph_{\ell}(t'))\\[2pt]
	&\textstyle\leq 
	\Lip(\ppp,\wph_{\ell+m})\cdot |t-t_{\ell+m}|\\
	&\textstyle\quad\he + \sum_{k=m-1}^{1} \Lip(\ppp,\wph_{\ell+k})\cdot |t_{(\ell+1)+k}-t_{\ell+ k}|\\
	&\quad\he  + \Lip(\ppp,\wph_\ell)\cdot |t_{\ell+1}-t'|\\[2pt]
	&\textstyle\leq 
	2\cdot \delta_{\ell+m}^{-8}\cdot C[\rhon,1]^2\cdot\ppp_\infty^\lip(\phi_{\ell+m})\cdot |t-t_{\ell+m}|\\
	&\textstyle\quad\he + \sum_{k=m-1}^{1} 2\cdot \delta_{\ell+k}^{-8}\cdot C[\rhon,1]^2\cdot\ppp_\infty^\lip(\phi_{\ell+k})\cdot |t_{(\ell+1)+k}-t_{\ell+k}|\\
	&\quad\he  + 2\cdot \delta_{\ell}^{-8}\cdot C[\rhon,1]^2\cdot\ppp_\infty^\lip(\phi_{\ell})\cdot |t_{\ell+1}-t'|\\[2pt]
	&\textstyle\leq 2\cdot C[\rhon,1]^2\cdot \max_{0\leq k\leq m}\big(\delta_{\ell + k}^{-8}\cdot \ppp_\infty^\lip(\phi_{\ell + k})\big)\cdot|t-t'|.
\end{split} 
\end{align}
Combining \eqref{wweewewewewddsads12} with \eqref{wweewewewewddsads123}, we obtain \eqref{asnbsanbsanbsanbyxxy}.
\end{proof}

\subsection{Proof of Lemma \ref{podspodspodsodsp}}
\label{asasajkakjsajkjaskjakjs}
In this appendix, we prove 
\begin{customlem}{\ref{podspodspodsodsp}}
Suppose that $G$ is Mackey {\rm k}-continuous for $k\in \NN\sqcup\{\lip,\infty,\const\}$. 
Suppose furthermore that we are given $\DIDE^k_{[r,r']}\supseteq \{\phi_n\}_{n\in \NN}\mackarr{\kk} \phi \in \DIDE^k_{[r,r']}$ as well as $\DIDE^k_{[r,r']}\supseteq \{\psi_\alpha\}_{\alpha\in I}\netarr{0} \psi \in \DIDE^k_{[r,r']}$ for $[r,r']\in \COMP$, such that the expressions
\begin{align*}
\xi(\phi,\psi)&\textstyle:=\dd_e\LT_{\innt \phi}\big(\int \Ad_{[\innt_r^s \phi]^{-1}}(\psi(s))\:\dd s \big)\\
\xi(\phi_n,\psi_\alpha)&\textstyle:=\dd_e\LT_{\innt \phi_n}\big(\int \Ad_{[\innt_r^s \phi_n]^{-1}}(\psi_\alpha(s))\:\dd s \big)\qquad\quad\forall\: n\in \NN
\end{align*} 
are well defined; i.e., such that the occurring Riemann integrals exist in $\mg$. Then, we have
\begin{align}
\label{nmdsnmsdkjdsjdsdsdssdds}
	\textstyle \lim_{(n,\alpha)}\xi(\phi_n,\psi_\alpha)=\xi(\phi,\psi).
\end{align} 
\end{customlem}
\noindent
For this, we first show the following analogue to Lemma 41 in \cite{RGM}.
\begin{lemma}
\label{fdpfdpopofd}
Suppose that $G$ is Mackey {\rm k}-continuous for $k\in \NN\sqcup\{\lip,\infty,\const\}$; and let $[r,r']\in \COMP$ be fixed. Let  
$\Gamma\colon G\times \mg\rightarrow \mg$ be continuous; and define 
\begin{align*}
	\textstyle\wh{\Gamma}\colon \DIDE^k_{[r,r']}\times C^k([r,r'],\mg)\rightarrow \mgc,\qquad (\phi,\psi)\mapsto\int \Gamma\big(\innt_r^{s}\phi,\psi(s)\big)\:\dd s.
\end{align*} 
Then, for each sequence $\DIDE^k_{[r,r']}\supseteq \{\phi_n\}_{n\in \NN}\mackarr{\kk} \phi \in \DIDE^k_{[r,r']}$ and each net $C^k([r,r'],\mg)\supseteq \{\psi_\alpha\}_{\alpha\in I}\netarr{0} \psi \in C^k([r,r'],\mg)$, we have
\begin{align*}
	\textstyle\lim_{(n,\alpha)} \wh{\Gamma}(\phi_n,\psi_\alpha)=\wh{\Gamma}(\phi,\psi).
\end{align*}

\end{lemma}
\begin{proof}
By \eqref{ffdlkfdlkfd}, it suffices to show that for
\begin{align*}
	\textstyle\wt{\Gamma}\colon \DIDE^k_{[r,r']}\times C^k([r,r'],\mg)\rightarrow C^0([r,r'],\mg),\qquad (\phi,\psi)\mapsto \big[t\mapsto \Gamma\big(\innt_r^{t}\phi,\psi(t)\big)\big],
\end{align*}
we have 
	$\textstyle\lim_{(n,\alpha)} \wt{\Gamma}(\phi_n,\psi_\alpha)=\wt{\Gamma}(\phi,\psi)$ 
w.r.t.\ the $C^0$-topology; i.e., that for $\pp\in\SEM$ and $\varepsilon>0$  
 fixed, there exist $N_\varepsilon\in \NN$ and $\alpha_\varepsilon\in I$ with
\begin{align}
\label{podspodspodspodsds}
\ppp_\infty\big(\wt{\Gamma}(\phi_n,\psi_\alpha)-\wt{\Gamma}(\phi,\psi)\big)<\varepsilon\qquad\quad\forall\: n\geq N_\varepsilon,\:\: \alpha\geq \alpha_\varepsilon. 
\end{align}
For this, we let $\mu:=\innt_r^\bullet\phi$, and consider the continuous map
\begin{align*}
	\alpha\colon G\times \mg \times  G\times \mg \rightarrow \mg,\qquad ((g,X),(g',X'))\mapsto \ppp(\Gamma(g,X)-\Gamma(g',X')).
\end{align*}
 Then, for $t\in [r,r']$ fixed, there exists an open neighbourhood $W[t]\subseteq G$ of $e$, as well as $U[t]\subseteq \mg$ open with $0\in U[t]$, such that 
\begin{align}
\label{pofsofspofs}
	\textstyle\alpha((g,X),(g',X'))< \varepsilon\qquad\quad\forall\: (g,X),(g',X')\in \big[\mu(t)\cdot W[t]\big]\times \big[\psi(t) + U[t]\big]
\end{align}
holds. We  choose
\begingroup
\setlength{\leftmargini}{15pt}{
\renewcommand{\theenumi}{{\it\alph{enumi}})} 
\renewcommand{\labelenumi}{\theenumi}
\begin{enumerate}
\item
\label{kjfdkjfdkj1}
	$V[t]\subseteq G$ open with $e\in V[t]$ and $V[t]\cdot V[t]\subseteq W[t]$.
\item
\label{kjfdkjfdkj2}
	$O[t]\subseteq \mg$ open with $0\in O[t]$ and $O[t]+O[t]\subseteq U[t]$.
\item
\label{kjfdkjfdkj3}
	$J[t]\subseteq \RR$ open with $t\in J$, such that for $D[t]:=J[t]\cap [r,r']$, we have 
	\begin{align}
	\label{askljasjaljajskasa}
		\mu(D[t])\subseteq \mu(t)\cdot V[t]\subseteq \mu(t)\cdot \subseteq W[t]\qquad\text{and}\qquad \psi(D[t])\subseteq \psi(t)+O[t]\subseteq \psi(t)+ U(t).
	\end{align}
	\vspace{-22pt}
\end{enumerate}}
\endgroup
\noindent
Since $[r,r']$ is compact, there exist $t_0,\dots,t_n\in [r,r']$, such that $[r,r']\subseteq D_0\cup{\dots}\cup D_n$ holds.   
\begingroup
\setlength{\leftmargini}{12pt}
\begin{itemize}
\item
We define $V:= V[t_0]\cap{\dots}\cap V[t_n]$.  

Since $G$ is Mackey k-continuous, there exists some $N_\varepsilon\in \NN$ with
	\begin{align}
	\label{asljaljsasaljljsaas}
		\textstyle\innt_r^\bullet \phi_n\in \innt_r^\bullet \phi\cdot V\qquad\quad\forall\: n\geq N_\varepsilon.
	\end{align}
\item
	We define $O:=O[t_0]\cap{\dots}\cap O[t_n]$.   

Since $\{\psi_\alpha\}_{\alpha\in I}\netarr{0} \psi$ holds, there exists $\alpha_\varepsilon\in I$ with 
\begin{align}
	\label{asljaljsasaljljsaasxyxyxxy}
	(\psi_\alpha(t)-\psi(t))\in O\qquad\quad\forall\: t\in [r,r'],\:\: \alpha\geq \alpha_\varepsilon.	
\end{align} 
\end{itemize}
\endgroup
\noindent 
Then, for $\tau\in D_p$ with $0\leq p\leq n$, as well as $n\geq N_\varepsilon$ and $\alpha\geq \alpha_\varepsilon$, we obtain from  \eqref{asljaljsasaljljsaas}, \eqref{asljaljsasaljljsaasxyxyxxy}, as well as \eqref{askljasjaljajskasa} 
 for $t\equiv t_p$ there that
\begingroup
\setlength{\leftmargini}{12pt}
\begin{itemize}
\item
$\mu(t_p)^{-1} \cdot \innt_r^{\tau} \phi_n  = \big(\mu(t_p)^{-1}\cdot\mu(\tau)\big) \cdot \big([\innt_r^{\tau} \phi]^{-1} [\innt_r^{\tau} \phi_n]\big)  \in V\cdot V\subseteq W[t_p]$,
\vspace{2pt}
\item
$\psi_\alpha(\tau)-\psi(t_p)= (\psi_\alpha(\tau)-\psi(\tau))+ (\psi(\tau)-\psi(t_p))\in O + O\subseteq U[t_p]$.
\end{itemize}
\endgroup
\noindent 
The claim is thus clear from \eqref{pofsofspofs} and \eqref{askljasjaljajskasa}.  
\end{proof}
\begin{proof}[Proof of Lemma \ref{podspodspodsodsp}]
For each $\chi,\chi'\in \DIDE^k_{[r,r']}$, we have, cf.\ \eqref{LGPR} 
\begin{align*}
	\xi(\chi,\chi')= \dd_{(\innt\chi,e)}\:\mult\big(0,\wh{\Gamma}(\chi,\chi')\big)\qquad\quad\text{for}\qquad\quad  \Gamma\equiv \Ad(\inv(\cdot),\cdot);
\end{align*}
so that \eqref{nmdsnmsdkjdsjdsdsdssdds} holds by Lemma \ref{fdpfdpopofd}, because the Lie group multiplication is smooth.
\end{proof}

\subsection{Proof of Equation \eqref{podspodspods}}
\label{podspospodspodspodspodsp}
In this appendix, we show
\begin{align}
\tag{\ref{podspodspods}}
		\textstyle\innt \phi=\exp\big(\int \phi(s)\: \dd s\he\big).
\end{align}
\begin{proof}[Proof of Equation \eqref{podspodspods}]
We fix $I\equiv (\iota,\iota')$ with $\iota<r<r'<\iota'$, define $\psi\in C^0([\iota,\iota'],\mg)$ by 
\begin{align*}
	\psi|_{[\iota,r)}:=\phi(r)\qquad\qquad \psi|_{[r,r']}:=\phi\qquad\qquad\psi|_{(r',\iota']}:=\phi(r');
\end{align*}
and observe that\footnote{Recall the last statement in Sect.\ \ref{kfdlkfdlkfdscvpdfpofdofd} for the fact that $\im[\MX]\subseteq \dom[\exp]$ holds.} 
\begin{align*}
	\textstyle\MX\colon I\rightarrow \dom[\exp],\qquad  x\mapsto - (r-\iota)\cdot \phi(r) +\int_\iota^x \psi(s)\: \dd s
\end{align*}
fulfills the presumptions in Corollary \ref{dspopdspodsposd}, with
\begin{align*}
	\textstyle\exp\big(\int_r^\bullet \phi(s)\:\dd s\he\big)=\alpha|_{[r,r']} \qquad\quad\text{for}\qquad\quad \alpha\equivdef \exp\cp\he\MX.
\end{align*}  
By Corollary \ref{dspopdspodsposd}, we thus have $\alpha\in C^1(I,G)$, with 
\begin{align*}
	\textstyle\Der(\alpha)(x)&\textstyle=\frac{\dd}{\dd h}\big|_{h=0} \he \alpha(x+h)\cdot \alpha(x)^{-1}\\
	&\textstyle=\frac{\dd}{\dd h}\big|_{h=0} \he\alpha(x)^{-1}\cdot \alpha(x+h)\\
	&=\textstyle \dd_{\exp(\MX(x))}\LT_{\exp(-\MX(x))}\big( \frac{\dd}{\dd h}\big|_{h=0}\he\alpha(x+h)\big)\\
	&=\textstyle \big(\dd_{\exp(\MX(x))}\LT_{\exp(-\MX(x))}\cp \dd_e\LT_{\exp(\MX(x))}\big)\big( \int_0^1 \Ad_{\exp(-s\cdot \MX(x))}(\dot\MX(x))\:\dd s\big)\\
	&\textstyle=\int_0^1 \dot\MX(x)\:\dd s=\dot\MX(x)= \psi(x)= \phi(x)
\end{align*}
for each $x\in [r,r']$. Here, we have used in the second-, and the fifth step that $G$ is abelian.  
\end{proof}

\subsection{Proof of Lemma \ref{pofsdisfdodjjxcycxjsdsd}}
\label{podspospodspodsdsdspodspodspsdsss}
In this appendix, we prove 
\begin{customlem}{\ref{pofsdisfdodjjxcycxjsdsd}}
Suppose that $F$ is metrizable; and let $\{X_n\}_{n\in \NN}\subseteq F$ be a sequence with $\{X_n\}_{n\in \NN}\rightarrow X\in F$. Then, we have $\{X_n\}_{n\in \NN} \mackarrr X$.  
\end{customlem}
\begin{proof}
We choose $\SEMMM\equivdef \{\qq[m]\:|\: m\in \NN\}\subseteq \SEMM$  as in Lemma \ref{pofdofdopdf}.\ref{pofdofdopdf3}; and let $\mackeyindex\colon \NN\rightarrow \NN$ be strictly increasing with 
\begin{align}
\label{dfkdfjkdfjkfd}
		\textstyle\qq[m](X-X_n)\leq \frac{1}{m}\qquad\quad\forall\: n\geq \mackeyindex_{\qq[m]}:=\mackeyindex(m),\:\:  m\in \NN.
\end{align}	   
\begingroup
\setlength{\leftmargini}{12pt}
\begin{itemize}
\item
We define $\lambda_n:=\frac{1}{m}$ for each $n\in \NN$ with $\mackeyindex(m)\leq n < \mackeyindex(m+1)$; and observe that $\lim_n \lambda_n=0$ holds. 
\item
For $m,d,n\in \NN$ with $\mackeyindex(m+d)\leq n< \mackeyindex(m+d+1)$, we obtain
\vspace{-5pt}
\begin{align*}
	\textstyle\qq[m](X-X_n)\leq \qq[m+d](X-X_n)\stackrel{\eqref{dfkdfjkdfjkfd}}{\leq} \frac{1}{m+d}= \lambda_n.
\end{align*} 
This shows that $\qq[m](X-X_n)\leq \lambda_n$ holds for each $n\geq \mackeyindex_{\qq[m]}=\mackeyindex(m)$; thus, 
$\{X_n\}_{n\in \NN}\mackarrr X$.\qedhere
\end{itemize}
\endgroup
\end{proof}

\subsection{Proof of Lemma \ref{fdfddfdfdf}}
\label{podspospodsssdds}
In this appendix, we prove 
\begin{customlem}{\ref{fdfddfdfdf}}
If $F$ is metrizable, then $C^k([0,1],F)$ is metrizable for each $k\in \NN\sqcup\{\lip,\infty,\const\}$.  
\end{customlem}
\begin{proof}
Let $\{\qq[m]\:|\: m\in \NN\}\subseteq\SEMM$ be as in Lemma \ref{pofdofdopdf}.\ref{pofdofdopdf3}.  
	For each $m\in \NN$, we define
\begin{align*}	
	\dind[\lip,m]:=\lip\qquad \dind[\const,m]:=0\qquad \dind[\infty,m]:=m\qquad\quad \text{as well as}\qquad\quad \dind[k,m]:=k \qquad
	\forall\: k\in \NN.
\end{align*}	
 Moreover, for each $k\in \NN\sqcup\{\lip,\infty,\const\}$, we let
\begin{align*}
	 \zzs[k,m]:=\qq[m]_\infty^{\dind[k,m]}\qquad\forall\: m\in \NN \qquad\quad\text{as well as}\qquad\quad \SEMG_k:=\{\zzs[k,m]\: |\: m\in \NN\}.
\end{align*}
Let now $k\in \NN\sqcup\{\lip,\infty,\const\}$, $\qq\in \SEMM$, $\dind\llleq k$ be fixed. 
By Lemma \ref{xcxcxccxxcxc}, there exist $c>0$, $\ell\in \NN$ 
with $\qq\leq c\cdot\qq[\ell]$. We define 
$$
m:=
\begin{cases}
\:\ell & \text{for }\:\: k\in \NN\sqcup\{\lip,\const\},\\
\:\max(\dind,\ell)& \text{for }\:\: k\equiveq \infty,
\end{cases}
$$
observe that $\qq\leq c\cdot \qq[m]$ as well as $\dind\leq \dind[k,m]$ holds; and obtain
\begin{align*}
	\qq^\dind_\infty \leq c\cdot \qq[m]_\infty^{\dind[k,m]}=c\cdot \zzs[k,m].
\end{align*} 
	Then, Lemma \ref{xcxcxccxxcxc} shows that $\SEMG_k$ is a fundamental system; and, since $\SEMG_k$ is countable, the claim follows from Lemma \ref{pofdofdopdf}.
\end{proof}

\end{document}